\documentclass[a4paper]{amsart}

\usepackage{geometry}          
\usepackage{amsmath, amstext, ,amscd, amssymb, mathrsfs, amsthm}
\usepackage{stmaryrd}
\usepackage{mathrsfs}
\usepackage{hyperref}
\usepackage{enumitem}
\usepackage[all]{xy}

\newcommand{\A}{\mathcal{A}}
\newcommand{\B}{\mathcal{B}}
\newcommand{\R}{\mathbb{R}}
\newcommand{\C}{\mathbb{C}}
\newcommand{\T}{\mathbb{T}}
\newcommand{\Z}{\mathbb{Z}}
\newcommand{\N}{\mathbb{N}}

\newcommand{\F}{\mathrm{F}}
\newcommand{\G}{\mathrm{G}}

\newcommand{\smooth}{\mathscr{C}^{\infty}}

\newcommand{\sob}{\boldsymbol{H}}

\newcommand{\HH}{\mathrm{H}}
\newcommand{\ic}{\sqrt{-1}}
\newcommand{\lie}{\mathcal{L} }
\newcommand{\bw}{\overline{w}}

\newcommand{\derpar}[2]{\frac{\partial#1}{\partial#2}}

\newcommand{\e}{\varepsilon}

\newcommand{\E}{\mathbb{E}}

\newcommand{\FF}{\mathscr{F}}

\newcommand{\db}{\bar{\partial}}
\newcommand{\Wedge}{\Lambda^{0,\bullet}}

\newcommand{\End}{\mathrm{End}}
\newcommand{\Sp}{\mathrm{Sp}}

\newcommand{\LL}{\mathscr{L}}
\newcommand{\0}{\mathcal{O}}
\newcommand{\rad}{\mathcal{R}}
\newcommand{\PP}{\mathbb{P}}

\newcommand{\RR}{\mathrm{R}}
\newcommand{\kahler}{\text{K\"{a}hler}}
\newcommand{\n}{\nabla}
\newcommand{\ol}{\overline}
\newcommand{\ul}{\underline}
\newcommand{\trs}{\tr_{s}}
\newcommand{\limarrow}[2]{\xrightarrow[#1\to #2]{}}
\newcommand{\isom}{\overset{\sim}{\longrightarrow}}
\newcommand{\K}{\mathbb{K}}
\newcommand{\toep}{\mathcal{T}}
\newcommand{\Fp}{{\mathcal{F}_p}}
\newcommand{\har}{\mathscr{H}}
\newcommand{\boldrm}[1]{{\boldsymbol{\mathrm{#1}}}}

\renewcommand{\Re}{\mathrm{Re}}
\renewcommand{\Im}{\mathrm{Im}}

\DeclareMathOperator{\Id}{Id} 

\DeclareMathOperator{\tr}{Tr}

\DeclareMathOperator{\ch}{ch}

\newtheorem{prop}{Proposition}[section]
\newtheorem{thm}[prop]{Theorem}

\newtheorem{lemme}[prop]{Lemma}

\newtheorem{cor}[prop]{Corollary}
\newtheorem{assumption}[prop]{Assumption}

\theoremstyle{definition}

\newtheorem{defn}[prop]{Definition}

\theoremstyle{remark}

\newtheorem{rem}[prop]{Remark}

\numberwithin{equation}{section}
\begin{document}

\title{The asymptotic of the holomorphic analytic torsion forms}
\date{\today}

\author{Martin {\sc{Puchol}}} 
\address{Universit\'{e} Paris Diderot--Paris 7, Campus des Grands Moulins, B\^{a}timent Sophie Germain, case 7012, 75205 Paris Cedex 13}
\email{martin.puchol@imj-prg.fr}

\begin{abstract}
The purpose of this paper is first to give an asymptotic formula for the holomorphic analytic torsion forms of a fibration associated with increasing powers of a given line bundle. Secondly, we generalize this formula, thanks to the theory of Toeplitz operators, in the case where the powers of the line bundle is replaced by the direct image of powers of a line bundle on a bigger manifold. In both cases we have to make fiberwise positivity assumption on the line bundle. This results are the family versions of the results of Bimsut and Vasserot on the asymptotic of the holomorphic torsion.
\end{abstract}

\maketitle



\section{Introduction}

The holomorphic analytic torsion was defined in \cite{MR0383463} by Ray and Singer as the complex analogue of its real version for flat vector bundles. It is obtained by regularizing the determinant of the Kodaira Laplacian of holomorphic vector bundles on a compact complex manifold. It appears in the study by Bismut-Gillet-Soul\'{e} of the determinant of the fiberwise cohomology of a holomorphic fibration in \cite{bismut-gillet-soule}.

Analytic torsion has an extension in the family setting: the analytic torsion forms, defined in various degrees of generality by Bismut-Gillet-Soul\'{e} \cite{MR929147}, Bismut-K\"{o}hler \cite{bismut-kohler} and Bismut \cite{MR3099098}. The 0-degree  component of these forms is the analytic torsion of Ray-Singer along the fiber. The analytic torsion forms have found many applications, especially because it was introduced, by Gillet and Soul\'{e} in particular, as the analytic counterpart of the direct image in Arakelov geometry. In deed, the torsion appear in the arithmetic Riemann-Roch theorem \cite{MR1189489} and the torsion forms in the arithmetic Riemann-Roch-Grothendieck theorem in higher degrees  \cite{MR2473633}. An other application of holomorphic torsion is the study of the moduli space of K3 surfaces by Yoshikawa in \cite{yoshikawa2004k3} and his subsequent works. See also the recent works \cite{Gil:2012fk,MR3165730} on analytic torsion classes and their application to the the arithmetic Grothendieck-Riemann-Roch theorem in the case of general projective morphisms between regular arithmetic varieties.

Analytic torsion has an equivariant version, introduced in \cite{ma2000submersions} and \cite{bismut2004holomorphic}. In \cite{MR1872550,MR1881571}, K\"{o}hler and Roessler have used equivariant  torsion in their work on a Lefschetz type fixed point formula in Arakelov geometry.

In \cite{bismut-vasserot-89}, Bismut and Vasserot computed the asymptotic of the analytic torsion associated with increasing powers of a positive line bundle, using the heat kernel method of \cite{MR886814} (see also \cite[Sect. 5.5]{ma-marinescu}). They also extended their result in \cite{bismut-vasserot-90}, in the case where the powers of the line bundle are replaced by the symmetric powers of a positive bundle using a trick due to Getzler \cite{MR894572}. These asymptotics have played an important role in a result of arithmetic ampleness by Gillet and Soul\'{e} \cite{MR1189489} (see also \cite[Chp VIII]{MR1208731}). 

In this paper, we give the family versions at the level of forms of the results Bismut and Vasserot for the analytic torsion forms. We first consider the case of torsion forms of a fibration associated with increasing powers of a given positive line bundle which is positive along the fiber. This correspond to \cite{bismut-vasserot-89}. We will use a similar strategy as in that paper, but some additional difficulties appear due to the horizontal differential forms appearing in the Bismut superconnection (compared to the Kodaira Laplacian) used in the definition of the torsion forms. Indeed, the operators we are dealing with here have a nilpotent part (i.e., the part in positive degree along the basis) that must be taken into account, especially when estimating resultants or heat kernels. Moreover, to give the asymptotic formula we have to compute explicitly super-traces of terms involving an exponential coupling horizontal forms and vertical Clifford variables, which makes the computation much more complicated than in \cite{bismut-vasserot-89}.  Note also that in all our results of smooth convergence, we have to take into account the derivatives along the basis.

Next, we consider the case of torsion forms of a fibration associated with the direct image of powers of a line bundle on a bigger manifold. We have to make some partial positivity assumption on the line bundle. This generalize \cite{bismut-vasserot-90} in two ways. Firstly we work in the family setting. Secondly it is easy to see that the results of \cite{bismut-vasserot-90} apply in fact to the direct image of powers of a line bundle on a bigger manifold given by a principal $G$-bundle with $G$ compact and connected. Here we do not assume that this is the case, and as a consequence, we cannot use the same trick as in \cite{bismut-vasserot-90} to reduce the problem to our first result. Thus, even if the basis is a point, i.e., for the torsion, we get a new result when compared to \cite{bismut-vasserot-90} .

In the general case, we thus use the same heat kernel approach as in our first result. However here, in addition to the difficulties pointed out above, we have to deal with the fact that the dimension of the bundle we are working with grows to infinity. In particular, we cannot hope to have a limiting operator for the rescaled operator, nor limitings coefficients in the development of the heat kernel, and in all our proofs we have to make uniform estimates on spaces that change. To overcome these issues, we will draw inspiration from \cite{BMZweb,MR2838248} and use the formalism of Toeplitz operators of \cite{ma-marinescu}. The idea is to use the operator norm on matrices to have uniform boundedness properties of Toeplitz operators, and to replace the convergence to limiting objects by an approximation by objects with Toeplitz coefficients.

We now give more details about our results. Let $M$ and $B$ be two complex manifolds. Let $\pi \colon M \to B$ be a holomorphic fibration with compact fiber $X$ of dimension $n$. We denote by $TX$ the holomorphic tangent bundle to the fiber, and $T_\R X$ the real tangent bundle. We denote by $T_\C X=T_\R X\otimes \C$ the complexified tangent bundle, and $T^{(1,0)}X,T^{(0,1)}X\subset T_\C X$ the $\pm\ic$-eigenspace of the complex structure $J^{T_\R X}$ of the fiber. Recall that we have a canonical isomorphism $TX\simeq T^{(1,0)}X$. In the sequel, we will use the same notations for all the other tangent bundles.

Let $(\pi, \omega)$ be a structure of Hermitian  fibration in the sense of Section \ref{Sect-Hermitian fibration}, i.e., $\omega$ is a smooth $(1,1)$-form on $M$ which induces a Hermitian metric $h^{T X}$ along the fibers. 

Let $(\xi, h^\xi)$ be a holomorphic Hermitian vector bundle on $M$, and let $(L,h^L)$ be a holomorphic Hermitian line bundle on $M$. We denote the curvature of the Chern connection of $L$ by $R^L$, and we make the following basic assumption:

\begin{assumption}
\label{RLpositive}
The (1,1)-form $\ic R^L$ is positive along the fibers, which means that for any $0\neq U \in T^{(1,0)}X$, we have
\begin{equation}
R^L(U,\ol{U}) > 0. 
\end{equation}
\end{assumption}

Let $\dot{R}^{X,L}\in \End(TX)$ be the Hermitian matrix such that for any $U,V \in T^{(1,0)}X$,
\begin{equation}
\label{defRpoint}
R^L(U,\overline{V}) = \langle \dot{R}^{X,L} U, V \rangle _{h^{TX}}.
\end{equation}
By Assumption \ref{RLpositive}, $\dot{R}^{X,L}$ is positive definite.

 For $p\in \N$, let $L^p$ be the $p^{th}$ tensor power of $L$. We assume that there is a $p_0\in \N$ such that the direct image $R^i\pi_*(\xi\otimes L^p)$ is locally free for all $p\geq p_0$.

 \begin{rem}
 \label{positivite->liberte}
If the basis $B$ is compact, then Assumption \ref{RLpositive} implies that  for $p$ large enough the direct image $R^\bullet\pi_*(\xi\otimes L^p)$ is automatically locally free, and moreover that $R^i\pi_*(\xi\otimes L^p)=0$ for $i>0$. Thus our hypothesis is in fact a uniformity assumption over the compact subsets of $B$.
 \end{rem}
 
  \textbf{In the sequel, all results holds for $p\geq p_0$, and we will not repeat this hypothesis. }
 
 We endow $\xi \otimes L^p$ with the metric $h^{\xi \otimes L^p}$ induced by $h^\xi$ and $h^L$. We can then define (see Definition \ref{def-torsion}) the analytic torsion forms $\mathscr{T}(\omega, h^{\xi\otimes L^p})$ associated with $(\pi, \omega)$ and $(\xi \otimes L^p,h^{\xi \otimes L^p})$.

If $\alpha$ is a form on $B$, we denote by $\alpha^{(k)}$ its component of degree $k$. We can now state our first main result, which is the extension of \cite{bismut-vasserot-89} in the family case:
\begin{thm}
\label{mainthm}
Let $k \in \{0,\dots,\dim B\}$. Then the component of degree $2k$ of the torsion form $\mathscr{T}(\omega, h^{\xi\otimes L^p})$ associated with $\omega$ and $h^{\xi\otimes L^p}$ have the following asymptotic as $p\to +\infty$:
\begin{equation}
\label{mainthm-eq}
 \mathscr{T}(\omega, h^{\xi\otimes L^p})^{(2k)} = \frac{\mathrm{rk}(\xi)}{2}  \left(\int_X  \log \left[ \det\left( \frac{p\dot{R}^{X,L}}{2\pi}\right) \right] e^{p\frac{\ic}{2\pi}R^L}\right)^{(2k)}  +o(p^{k+n}),
\end{equation}
in the topology of $\smooth$ convergence on compact subsets of $B$.
\end{thm}

We now turn to our second result. Let $N$, $M$ and $B$ be three complex manifolds. Let $\pi_1 \colon N \to M$ and $\pi_2 \colon M \to B$ be holomorphic fibrations with compact fiber $Y$ and $X$ respectively.  Then $\pi_3:=\pi_2\circ\pi_1\colon N \to B$ is a holomorphic fibration, whose compact fiber is denoted by $Z$. We denote by $n_X$ (resp. $n_Y$, $n_Z$) the complex dimension of $X$ (resp. $Y$, $Z$). Note that $\pi_1|_{Z}\colon Z\to X$ is a holomorphic fibration with fiber $Y$. This is summarized in the following diagram:

$$  \xymatrix{
 Y\ar[r]  & Z\ar[r] \ar[d]^{\pi_1} & N \ar[d]^{\pi_1} \ar[rd]^{\pi_3} \\
 & X \ar[r]& M\ar[r]_{\pi_2} & B
 }$$

We suppose that we are given $(\pi_2,  \omega^M)$ a structure of Hermitian  fibration (see Section \ref{Sect-Hermitian fibration}). We denote by $T^H_BM = TX^{\perp,\omega^M}$ the corresponding horizontal space.

Let $(\xi, h^\xi)$ be a holomorphic Hermitian vector bundle on $M$, and let $(\eta, h^\eta)$ be a holomorphic Hermitian vector bundle on $N$. Let $(L,h^L)$ be a holomorphic Hermitian line bundle on $N$. We denotes its Chern connection by $\n^L$, and the corresponding curvature by $R^L$.

 As above, we make a positivity assumption on $L$:
\begin{assumption}
\label{RLpositive2}
The (1,1)-form $\ic R^L$ is positive along the fibers of $\pi_3$, that is for any $0\neq U \in TZ$, we have
\begin{equation}
R^L(U,\ol{U}) > 0. 
\end{equation}
\end{assumption}

In particular, $\frac{\ic}{2\pi}R^L$ enables us to define metrics  $g^{T_\R Z}$ and $g^{T_\R Y}$ on  $T_\R Z$ and $T_\R Y$ (see \eqref{def-metrci-fibre}).

We assume that there is $p_0\in \N$ such that for $p\geq p_0$, the direct image $R^\bullet\pi_{1*}( L^p)$ is locally free and $R^i\pi_{1*}( L^p)=0$ for $i>0$.  Then for $p\geq p_0$,
\begin{equation}
F_p:=H^0\big(Y,L^p|_{Y}\big)
\end{equation}
is  a  holomorphic vector bundle on $M$, endowed with the $L^2$ metric $h^{F_p}$ induced by  $h^L$ and $g^{T_\R Y}$. 
 
For $p\geq p_0$, we also assume that  the direct images $R^\bullet \pi_{2*}( F_p)$ and $R^\bullet \pi_{3*}( L^p)$ are locally free. Then an easy spectral sequence argument shows that for all $i\geq 0$, 
\begin{equation}
\label{R=RR}
R^i \pi_{2*}( F_p)\simeq R^i \pi_{3*}( L^p).
\end{equation}

 \begin{rem}
 \label{positivite->liberte2}
If the basis $B$ is compact, then Assumption \ref{RLpositive2} and Kodaira vanishing theorem imply the existence of $p_0$ such that for $p\geq p_0$  the above conditions are satisfied, i.e., the direct images $R^\bullet\pi_{1*}( L^p)$, $R^\bullet \pi_{2*}( F_p)$ and $R^\bullet \pi_{3*}( L^p)$ are locally free and concentrated in degree zero. In particular, 
\begin{equation}
H^\bullet\big(X,F_p|_{X}\big)= H^0\big(X, F_p|_{X}\big)\simeq H^0\big(Z,L^p|_{Z}\big).
\end{equation}

 Thus our hypothesis is again a uniformity assumption over the compact subsets of $B$.
 \end{rem}
 
  \textbf{Here again,  all results in the sequel holds for $p\geq p_0$, and we will not repeat this hypothesis. }

We endow $\xi \otimes F_p$ with the metric $h^{\xi \otimes F_p}$ induced by $h^\xi$ and $h^{F_p}$. Let $\mathscr{T}(\omega^M, h^{\xi\otimes F_p})$ be the holomorphic analytic torsion associated with $\omega^M$ and $(\xi\otimes F_p,h^{\xi\otimes F_p})$ as in Definition \ref{def-torsion}.

Let 
\begin{equation}
T^H_BN = (TZ)^\perp , \quad T^H_MN = (TY)^\perp,
\end{equation}
where the orthogonal complements are taken with respect to $R^L$. Then 
\begin{equation}
T^H_XZ:=T^H_MN\cap TZ
\end{equation}
 is the orthogonal complement of $TY$ in $TZ$. Moreover,
\begin{equation}
T^H_BN\simeq \pi_3^*TB, \quad T^H_MN\simeq \pi_1^*TM \quad \text{and} \quad T^H_XZ\simeq \pi_1^*TX.
\end{equation}

Let $\dot{R}^{X,L}\in  \pi_1^*\End(TX)$  be the Hermitian matrix such that for any $U,V\in TX$, if we denote their horizontal lifts by $U^H,V^H \in T^H_XZ$, then
\begin{equation}
R^L(U^H,\overline{V}^H) = \langle \dot{R}^{X,L} U, V \rangle _{h^{TX}}.
\end{equation}
By Assumption \ref{RLpositive2}, $\dot{R}^{X,L}$ is positive definite.

\begin{rem}
\label{N->Mkahler}
Note that $(\pi_1,-\frac{\ic}{2\pi}R^L)$ and $(\pi_{1}|_Z,-\frac{\ic}{2\pi}R^L|_Z)$ define {\kahler} fibrations in the sense of Section \ref{kahler-fibration}, with respective horizontal spaces $T^H_MN$ and $T^H_XZ$.
\end{rem}

 We can now state the second main result of this paper, which is an extension of Theorem \ref{mainthm}, and the family version of \cite{bismut-vasserot-90} (see the introduction of Section \ref{Sectionimagedirecte}).
\begin{thm}
\label{mainthm2}
Let $k \in \{0,\dots,\dim B\}$. Then the component of degree $2k$ of the torsion form $\mathscr{T}(\omega, h^{\xi\otimes F_p})$ associated with $\omega^M$ and $h^{\xi\otimes F_p}$ have the following asymptotic as $p\to +\infty$:
\begin{equation}
 \mathscr{T}(\omega^M, h^{\xi\otimes F_p})^{(2k)} = \frac{\mathrm{rk}(\xi)\mathrm{rk}(\eta)}{2} \left( \int_Z  \log \left[ \det\left( \frac{p\dot{R}^{X,L}}{2\pi}\right) \right] e^{p\frac{\ic}{2\pi}R^L} \right)^{(2k)}  +o(p^{k+n_Z}),
\end{equation}
in the topology of $\smooth$ convergence on compact subsets of $B$.
\end{thm}

\begin{rem}
 Theorem \ref{mainthm2} is the family version of \cite{bismut-vasserot-90}, with a more general bundle. Indeed, let $V$ is a positive bundle on $M$ in the sense of \cite{bismut-vasserot-90}. Then on the projectivization $N:=\mathbb{P}(V^*)$ of $V^*$ we can define $L$ to be the dual of the universal line bundle. Then $L$ satisfies Assumption \ref{RLpositive2}. Let $Y$ be the fiber of $\mathbb{P}(V^*)\to M$, then for any $p\in \N$, $H^{\bullet}(Y,L^p|_Y) = H^0(Y,L^p|_Y)\simeq S^p(V)$ the $p^{\mathrm{th}}$ symmetric power of $V$. Thus if we apply Theorem \ref{mainthm2} for this fibration $N\to M$ and with $B$ being a point, we find \cite{bismut-vasserot-90}.

When is $(\pi_2,\omega^M)$ is a {\kahler} fibration, we can prove Theorem \ref{mainthm2} modulo $\Im \partial +\Im \db$ from Theorem \ref{mainthm}. In deed, we can use \cite{bismut-kohler} and  \cite{MR1765553} to express $\mathscr{T}(\omega^M, h^{\xi\otimes F_p})$ in terms of torsions associated with $\pi_1^*\xi\otimes \eta \otimes L^p$, then apply Theorem \ref{mainthm} to get the asymptotic. It is important to keep in mind that this method  cannot prove the convergence at the level of forms and that when $B$ is not compact or not {\kahler}, the space $\Im \partial +\Im \db$ is not closed. Thus, this strategy is relevant only when $B$ is compact {\kahler}. 

Note however that in degree zero, i.e., for the torsion of Ray-Singer, we do not have this problem of taking quotient. Thus Theorem \ref{mainthm2} in degree 0 can be seen as a consequence of \cite[Thm. 8]{bismut-vasserot-89}, Theorem \ref{mainthm} (in all degrees) and  \cite[Thm. 3.1]{MR1305280} (which is \cite{MR1765553} in degree 0). In this situation, our approach gives a direct proof.
\end{rem}

As explained above and in Section \ref{Sectionimagedirecte}, we will use the formalism of Toeplitz operators to prove this theorem. 

We now recall the definition given in \cite[Def. 7.2.1]{ma-marinescu} of a Toeplitz operator.

Let $b\in B$. Set $x\in X_b:=\pi_2^{-1}(b)$ and  $Y_x := \pi_1^{-1}(x)$. Let $P_{p,x}$ be the orthogonal projection 
\begin{equation}
P_{p,x} \colon L^2(Y_x,\eta\otimes L^p)\to H^0(Y_x,\eta\otimes L^p),
\end{equation}
\begin{defn}
\label{def-op-de-Toepltiz}
A \emph{Toeplitz operator} on $Y_x$ is a family of operators $T_p\in\End( L^2(Y_x,\eta\otimes L^p))$ satisfying the following two properties:
\begin{enumerate}[label=(\roman*)]
\item for any $p\in \N$, we have
\begin{equation}
T_p=P_{p,x}T_pP_{p,x};
\end{equation}
\item there exists a sequence $f_r \in  \smooth\big(Y, \End(\eta)\big)$ such that for any $k\in \N$ there is  $C_k>0$ with
\begin{equation}
\left \| T_p - \sum_{r=0}^k p^{-r} P_{p,x}f_rP_{p,x} \right\|_\infty \leq C_k p^{-k-1},
\end{equation}
where $\|\cdot\|_\infty$ denotes the operator norm.
\end{enumerate}
\end{defn}  

In the course of the proof of Theorem \ref{mainthm2}, we will prove an important result which is that the heat kernel of the Bismut superconnection is asymptotic to a family of Toeplitz operator. Let us give some detail about this result. Let $B_{u,p}$ be the Bismut superconnection associated with $\omega^M$ and $(\xi\otimes F_p,h^{\xi\otimes F_p})$ (see Definition \ref{Bu}). Then by Theorem \ref{Lichnerowicz}, $B^2_{u,p}$ is a fiberwise elliptic second order differential operator. Let $\exp(-B_{p,u/p}^2)$ be the corresponding heat kernel. For $b\in B$, let $\exp(-B_{p,u/p}^2)(x,x')$ be the smooth Schwartz kernel of $\exp(-B_{p,u/p}^2)$ with respect to $dv_{X_b}(x')$. Then
 \begin{equation}
\exp(-B_{p,u/p}^2)(x,x) \in \End\left( \Lambda^\bullet(T^*_{\R,b} B)\otimes \left(\Wedge(T^*X_b)\otimes \xi \otimes F_p\right) \right).
\end{equation}

For $a>0$,  $\psi_a$ is the automorphism of $\Lambda(T_\R^*B)$ such that if $\alpha \in \Lambda^{q}(T_\R^*B)$, then 
\begin{equation}
\label{defpsia}
\psi_a \alpha = a^q \alpha.
\end{equation}
Let $\Omega_u$ be the form defined in \eqref{defOmegau2}. Then we show that
 \begin{thm}
\label{cvcenoyaudelachaleur2}
Let $k\in \N$. As  $p\to +\infty$, uniformly as $u$ varies in a compact subset of $\R^*_+$ and $(b,x)$ varies in a compact subset of $M$, we have the following asymptotic for the operator norm on $\End\left( \Lambda^\bullet(T^*_{\R,b} B)\otimes \left(\Wedge(T^*X_b)\otimes \xi \otimes F_p\right) \right)$ and the operator norm of the derivatives up to order $k$:
\begin{multline}
\psi_{1/\sqrt{p}}\exp(-B_{p,u/p}^2) (x,x)\\
= \frac{p^{n_X}}{(2\pi)^{n_X}} P_{p,x} e^{-\Omega_{u,(x,\cdot)}}\frac{\det(\dot{R}^{X,L}_{(x,\cdot)})}{\det \big(1-\exp(-u\dot{R}^{X,L}_{(x,\cdot)})\big)}\otimes \Id_{\xi_{x}}P_{p,x} + o(p^{n_X}).
\end{multline}
Here the dot symbolize the variable in $Y_x$.
\end{thm}

In degree 0, $B_{p,u/p}^{2,(0)}=\frac{u}{p}\square_p$, where $\square_p$ denotes the Kodaira Laplacian of $(F_p|_X,h^{F_p|_X})$. We thus get the asymptotic of the heat kernel:
\begin{equation}
\exp(-\frac{u}{p}\square_p) (x,x)
= \frac{p^{n_X}}{(2\pi)^{n_X}} P_{p,x} e^{-\omega_{u,(x,\cdot)}}\frac{\det(\dot{R}^{X,L}_{(x,\cdot)})}{\det \big(1-\exp(-u\dot{R}^{X,L}_{(x,\cdot)})\big)}\otimes \Id_{\xi_{x}}P_{p,x} + o(p^{n_X}).
\end{equation}
where $\omega_u=\Omega_u^{(0)}$. Let $\{ w_j\}$ be  an orthonormal frame of $(TX,h^{TX})$, with dual frame $\{w^j\}$, we then have $\omega_u =  uR^L(w_k^H,\bw_\ell^H)\bw^{\ell}\wedge i_{\bw_k}$. Thus, the asymptotic of the heat kernel is given by a Toeplitz operator associated with a term similar to the one appearing in the classical asymptotic of the heat kernel associated with high powers of a line bundle (see for instance \cite[Thm. 1.6.1]{ma-marinescu}).

\begin{rem}
 Note that in the proof of Theorem \ref{cvcenoyaudelachaleur2} which we give in this paper, we do not use the assumption that $L$ is positive along the fiber $Z$, but only along the fiber $Y$.
\end{rem}

The results of this paper appear (in a more detail-heavy way) in the PhD thesis of the author \cite{thesis_puchol} and were announced in \cite{asymp-torsion-CRAS}.

This paper is organized as follows. In Section \ref{torsion} we recall the definition given in \cite{MR3099098} of the analytic torsion forms, in Section \ref{Sectionlinebundle} we give the asymptotic of the torsion forms associated with increasing powers of a given line bundle and in Section \ref{Sectionimagedirecte} we give the asymptotic of the torsion forms associated with the direct image of powers of a line bundle on a bigger manifold. Sections \ref{Sectionlinebundle} and \ref{Sectionimagedirecte} begin with introductions where the reader can find the notations and assumptions.


\section{The holomorphic analytic torsion forms}
\label{torsion}

In this section, following \cite[Chap. 3-4]{MR3099098}, we will define the holomorphic analytic torsion forms associated to a holomorphic Hermitian (non-necessarily {\kahler}) fibration. This section is organized as follows. In Subsection \ref{Sect-Hermitian fibration} we define Hermitian fibrations, In Subsection \ref{Sect-Bismut-superconnection} we recall the definition of the Bismut superconnection associated with a Hermitian fibration and give the formula for its square, in Subsection \ref{cohomologiedelafibre}, we introduce the cohomology of the fiber as a bundle on the basis and its Chern connection, in Subsection \ref{formesdetorsionsanalytic} we define the analytic torsion forms and finally in Subsection \ref{kahler-fibration} we recall the definition of a  {\kahler} fibration and we specialize the above constructions in this case.

\subsection{A Hermitian  fibration}
\label{Sect-Hermitian fibration}

Let $M$ and $B$ be two complex manifolds of respective dimension $m$ and $\ell$. Let $\pi \colon M \to B$ be a holomorphic fibration with $n$-dimensional compact fiber $X$. Recall that we denote by $TM$ (resp. $TB$) the holomorphic tangent bundle of $M$ (resp. $B$), and by $TX$ the relative holomorphic tangent bundle $TM/B$. We denote the real tangent bundles by $T_{\R}M$, etc. and their complexification by $T_{\C}M$, etc. 

Let $J^{T_\R X}$ be the complex structure on $T_\R X$, and let $\omega$ be a smooth real (1,1)-form on $M$. Let 
\begin{equation}
\label{defomegaX}
\omega^X=\omega|_{T_\R X\times T_\R X}.
\end{equation}
We assume that the formula
\begin{equation}
\langle \cdot \,, \cdot \rangle_{g^{T_\R X}} := \omega^X(J^{T_\R X}\cdot \, ,\cdot)
\end{equation}
defines a Riemannian structure on $T_\R X$. We denote by $h^{TX}$ the associated Hermitian structure on $TX$.

Let $T^HM\subset TM$ be the orthogonal bundle to $TX$ in $TM$ with respect to $\omega$, and $T_\R^HM\subset T_\R M$ be the corresponding real vector bundle. Then we have the isomorphism of smooth vector bundles
\begin{equation}
\label{identifications}
T^HM \simeq \pi^*TB, \quad \text{and} \quad TM = T^HM\oplus TX.
\end{equation}
If $U \in T_\R B$, we denote by $U^H$ its lift in $T_\R^HM$.

The identifications \eqref{identifications} yields to the isomorphism
\begin{equation}
\Lambda^\bullet(T^*_\R M) \simeq \pi^*\Lambda^\bullet(T^*_\R B)\otimes \Lambda^\bullet(T^*_\R X).
\end{equation}
Here, and in this whole paper, $\otimes$ denotes the graded tensor product.

Let 
\begin{equation}
\label{defomegaH}
\omega^H=\omega|_{T_\R^H M\times T_\R^H M}.
\end{equation}
 We extend $\omega^X$ and $\omega^H$ (by zero) to $T_\R^HM\oplus T_\R X$. Then
\begin{equation}
\label{decompo-omega}
\omega=\omega^X+\omega^H.
\end{equation}

We call the data $(\pi,\omega)$ a \emph{Hermitian  fibration}.

\subsection{The Bismut superconnection of a Hermitian  fibration}
\label{Sect-Bismut-superconnection}

Let $(\pi, \omega)$ be a  Hermitian  fibration with associated Hermitian metric along the fibers $h^{TX}$. 

Let $g^{T_\R B}$ be a Riemannian metric on $B$, and let $g^{T_\R M}$ be the metric on $M$ induced by $g^{T_\R B}$, $g^{T_\R Z}$ and the decomposition \eqref{identifications}. Ultimately, the objects we will define will not depend on the choice of $g^{T_\R B}$.

Let $(\xi,h^{\xi})$ be a holomorphic Hermitian vector bundle on $M$. Let $\nabla^{TX}$ and $\nabla^{\xi}$ be the Chern connections on $(TX,h^{TX})$ and $(\xi,h^\xi)$. We denote their curvature by $R^{TX}$ and $L^\xi$ respectively. Let $\n^{\Wedge}$ be the connexion induced by $\n^{TX}$ on $\Wedge(T^*X):=\Lambda^\bullet (T^{*(0,1)}X)$, and $\n^{\Wedge\otimes \xi}$ be the connexion on $\Wedge(T^*X)\otimes \xi$ induced by $\n^{\Wedge}$ and $\n^\xi$.

\begin{defn}
For $0\leq p \leq \dim X$, and $b \in B$, set 
\begin{equation}
E_b^k = \smooth \left( X_b, \left(\Lambda^{0,k}(T^*X)\otimes \xi \right)|_{X_b} \right), \qquad E_b = \bigoplus_{k=0}^{\dim X} E_b^k.
\end{equation}
As in \cite{MR813584} or \cite{MR929147}, we can think of the $E_b$'s as the fibers of a $\Z$-graded infinite dimensional vector bundle $E$ on $B$. In this case, smooth sections of $E$ on $B$ are identified with smooth sections of $\Wedge(T^*X)\otimes \xi$ on $M$.
\end{defn}

Let $dv_{X_b}$ be the volume element of $(X_b , h^{TX}|_{X_b})$. Let $\langle\cdot , \cdot \rangle$ be the Hermitian product on $E$ associated to $h^{TX}$ and $h^\xi$:
\begin{equation}
\label{prodsurEb}
\langle s , s' \rangle_b = \frac{1}{(2\pi)^{\dim X}} \int_{X_b} \langle s , s' \rangle_{\Wedge\otimes \xi}(x) dv_{X_b}(x).
\end{equation}

\begin{defn}
For $U \in T_\R B$ and s a smooth section of $E$ on $B$, set
\begin{equation}
\label{defnablaEeq}
\n^{E}_U = \n^{\Wedge\otimes \xi}_{U^H}s.
\end{equation}
We extend $\n^E$ to an operator on $\smooth \big( M, \pi^*\Lambda^\bullet(T^*_\R B)\otimes\Wedge(T^*X)\otimes \xi\big)$, which will be again denoted by $\n^E$.
Let  ${\n^{E}}'$ and ${\n^{E}}''$ be the holomorphic and anti-holomorphic part of $\n^{E}$.
\end{defn}

Note that $\n^E$ does not necessarily preserve the Hermitian product \eqref{prodsurEb} on $E$.

For $b\in B$, let $\db^{X_b}$ be the Dolbeault operator acting on $E_b$ and let $\db^{X_b,*}$ be its formal adjoint with respect to the Hermitian product \eqref{prodsurEb}. Set
\begin{equation}
D^{X_b} = \db^{X_b}+\db^{X_b,*}.
\end{equation}

Let $C(T_\R X)$ be the Clifford algebra of $(T_\R X, g^{T_\R X})$. The bundle $\Wedge(T^*X)\otimes \xi$ is a $C(T_\R X)$-Clifford module: if $U\in T X \simeq T^{(1,0)}X$, denote by $U^*\in T^{*(0,1)}X$ its dual for the metric $h^{TX}$, and then 
\begin{equation}
c(U)=\sqrt{2}U^*\wedge \quad \text{ and }\quad c(\ol{U}) = -\sqrt{2}i_{\ol{U}}.
\end{equation}

Let $P^{T_\R X}$ be the projection $T_\R M = T_\R^HM \oplus T_\R X \to T_\R X$. For $U,V \in \smooth(B,T_\R B)$ set 
\begin{equation}
\label{def-T(UH,VH)}
T(U,V) = -P^{T_\R X} [ U^H , V^H].
\end{equation}

\begin{defn}
Let $f_1, \dots , f_{2n}$ be a basis of $T_\R B$ and $f^1, \dots , f^{2n}$ its dual basis. Set
\begin{equation}
\label{defcT}
c(T)  = \frac{1}{2} \sum_{\alpha, \beta} f^\alpha f^\beta c\left(T(f_\alpha, f_\beta) \right),
\end{equation}
which is a section of $\left[ \Lambda(T^*_\R B)\otimes \End \big( \Wedge(T^*X)\otimes \xi\big) \right]^{\text{odd}}$.

Let $T^{(1,0)}$ and $T^{(0,1)}$ be the components of the $(1,1)$ form $T$ in $T^{(1,0)}X$ and $T^{(0,1)}X$ respectively. We define $c(T^{(1,0)})$ and $c(T^{(0,1)})$ as in \eqref{defcT}, so that
\begin{equation}
c(T) = c(T^{(1,0)}) +c(T^{(0,1)}).
\end{equation}
\end{defn}

Let $\gamma$ be the one form on $T_\R B$ such that
\begin{equation}
\lie_{A^H}dv_X = \gamma(A)dv_X.
\end{equation}

We assume temporarily that $\det(TX)$ has a square root $\lambda$. Equivalently, $T_\R X$ is equipped with a spin structure. Then $\lambda$ is a holomorphic Hermitian vector bundle on $M$. Let $\n^\lambda$ be the corresponding Chern connection. Let
\begin{equation}
\mathcal{S}^{TX}=\Wedge(T^*X)\otimes \lambda^*
\end{equation}
be the associated $(T_\R X,g^{T_\R X})$-spinor bundle. Let $\n^{\mathcal{S}^{TX},LC}$ be the connection on $\mathcal{S}^{TX}$ induced by $\n^{T_\R X,LC}$, the Levi-Civita connection of $T_\R X$. Finally, let $\n^{\Wedge,LC}$ be the connection on $\Wedge(T^*X)$ induced by $\n^{\mathcal{S}^{TX},LC}$ and $\n^\lambda$, and let $\n^{\Wedge\otimes \xi,LC}$ be the connection  induced by $\n^{\Wedge,LC}$ and $\n^\xi$ on $\Wedge(T^*X)\otimes \xi$.

Note that, as $\det(TX)$ has always locally a square root, the connection $\n^{\Wedge\otimes \xi,LC}$ is in fact always defined.

The reader should be careful about the fact that  in \cite{MR3099098}, the Clifford algebra $C(T_\R X)$ is constructed with respect to $g^{T_\R X}/2$, so that our formulas will differ from those of \cite{MR3099098} by some powers of $1/\sqrt{2}$.

Let $(e_1,\dots,e_{2n})$ be an orthonormal frame of $T_\R X$.

\begin{defn}
We follow here \cite[Defs. 3.7.2, 3.7.4 and 3.7.5]{MR3099098}.
\begin{enumerate}
\item The Dirac operator of the fiber is defined by
\begin{equation}
D^{X,LC}:=c(e_i)\n^{\Wedge\otimes \xi,LC}_{e_i}
\end{equation}
\item For $U\in T_\R B$ and $s$ a smooth section of $E$, let
\begin{equation}
\n^{E,LC}_U s = \n^{\Wedge\otimes \xi,LC}_{U^H}s+ \frac{1}{2}\gamma(U)s.
\end{equation}
\item Finally,  let
\begin{equation}
A^{LC}=\n^{E,LC}+\frac{1}{\sqrt{2}}D^{X,LC} - \frac{c(T)}{2\sqrt{2}}.
\end{equation}
This superconnection on $E$ is called the \emph{Levi-Civita superconnection}.
\end{enumerate}
\end{defn}

Let $(e^1,\dots,e^{2n})$ be the dual frame of $(e_1,\dots,e_{2n})$. We define a map $\alpha  \mapsto \alpha^c $ from $ \Lambda(T^*_\R X)$ to $C(T_\R X)$ by setting for $1\leq i_1<\dots<i_k\leq 2n$:
\begin{equation}
\left(e^{i_1}\wedge\dots\wedge e^{i_k}\right)^c=2^{-k/2}c(e_{i_1})\dots c(e_{i_k}).
\end{equation}
We extend this map to a map (denoted in the same way) from $\Lambda^\bullet(T^*_\R M)\simeq\pi^*\Lambda^\bullet(T^*_\R B)\otimes\Lambda^\bullet(T^*_\R X)$ to $\pi^*\Lambda^\bullet(T^*_\R B)\otimes C(T_\R X)$.

\begin{prop}
The following formula holds
\begin{equation}
D^X=\frac{1}{\sqrt{2}}D^{X,LC}+ \frac{1}{2}\Big( \left( \db^X-\partial^X\right)i\omega^X\Big)^c.
\end{equation}
\end{prop}
\begin{proof}
See \cite[Thm. 3.7.3]{MR3099098} or \cite[Thm. 1.4.5]{ma-marinescu}.
\end{proof}

Recall that for $a>0$,  $\psi_a$ is the automorphism of $\Lambda(T_\R^*B)$ such that if $\alpha \in \Lambda^{q}(T_\R^*B)$, then 
\begin{equation}
\psi_a \alpha = a^q \alpha.
\end{equation}
By \eqref{identifications}, we may also see $\psi_a$ as an automorphism of $\Lambda(T_\R^{H,*}M)$.

We can now define the superconnection of main interest for us.
\begin{defn}
\label{Bu}
For $u>0$, the \emph{Bismut superconnection} $B$ on $E$, and its rescaled version $B_u$ are defined by
\begin{equation}
\label{defBu}
\begin{aligned}
&B=A^{LC}+ \frac{1}{2}\Big( \left( \db^M-\partial^M\right)i\omega\Big)^c, \\
&B_u=\sqrt{u} \psi_{1/\sqrt{u}}B\psi_{\sqrt{u}}.
\end{aligned}
\end{equation}
\end{defn}

Then $B_u$ acts on 
\begin{equation}
\Omega^\bullet (B,E) := \smooth \Big( M, \pi^*\Lambda^\bullet (T^*_\R B)\otimes \Wedge (T^*X)\otimes \xi \Big).
\end{equation}
Moreover, by \cite[(3.3.3), (3.5.17), (3.6.4) and (3.8.1)]{MR3099098}, the part of degree 0 in $\Lambda^\bullet T_\R B$ of $B$ is
\begin{equation}
\label{B^0=D}
B^{(0)}=D^X.
\end{equation}

\begin{rem}
This definition of the Bismut superconnection may not be the more natural and correspond in fact to \cite[Thm. 3.8.1]{MR3099098}. However, for the sake of concision we prefer to define $B$ this way. We refer the reader to \cite[Chap. 3]{MR3099098} for an other definition of $B$.
\end{rem}

Let $\n^{T_\R B,LC}$ be the Levi-Civita connection on $(T_\R B,g^{T_\R B})$. Then $\n^{T_\R B,LC}$ lifts to a connection $\n^{T_\R^HM,LC}$ on $T_\R^HM$, and we define $\n^{T_\R M,\oplus} = \n^{T_\R^HM,LC} \oplus \n^{T_\R X,LC}$. Let $\n^{T_\R M,LC}$ be the Levi-Civita connection of $M$. Set $\mathrm{S}= \n^{T_\R M,LC} - \n^{T_\R M,\oplus}$. Then $\mathrm{S}$ is a one form on $M$ taking values in antisymmetric elements of $\End (T_\R M)$. Moreover, by \cite[Thm. 1.9]{MR813584}, the $(3,0)$-tensor 
\begin{equation}
S(\cdot \, ,\cdot \, ,\cdot \, )  = \langle \mathrm{S}(\cdot)\cdot,\cdot \rangle_{h^{\T_\R M}} 
\end{equation}
does not depend on $g^{T_\R B}$.

From now on, we will always use latin indices $i,j,\dots$ for the vertical variables, and greek indices $\alpha,\beta,\dots$ for the horizontals variables. Let $\{e_i\}$ be an orthonormal basis of $T_\R X$ with dual basis $\{e^{i}\}$ and $\{f_\alpha\}$ a basis of $T_\R B$ with dual basis $\{f^\alpha\}$ (which will be identified with basis of $T_\R^HM$ and $(T_\R^HM)^*)$. For any $(k,0)$-tensor $A$, we will denote by $A_{a_1,\dots,a_k} = A(e_{a_1},\dots,e_{a_k})$ where $e_{a_i} = e_j$ or $f_\alpha$.

Let $K^X$ be the scalar curvature of $(X,TX)$. Set
\begin{equation}
L'^{\xi} = L^\xi +\frac{1}{2} \tr(R^{TX}).
\end{equation}

For $u>0$, define
\begin{equation}
\label{defnablau}
\n_{u,e_i} = \n^{\Wedge\otimes \xi,LC}_{e_i}+\frac{1}{\sqrt{2u}} S_{i,j,\alpha}c(e_j)f^\alpha + \frac{1}{2u} S_{i,\alpha,\beta}f^\alpha f^\beta+\frac{1}{2}\psi_{1/\sqrt{u}}\Big(i_{e_i} \left( \db^M-\partial^M\right)i\omega\Big)^c\psi_{\sqrt{u}},
\end{equation}
which is a fiberwise connection on $\pi^*\Lambda^\bullet (T^*_\R B)\otimes\Wedge(T^*X)\otimes \xi$.

The following theorem is the fundamental Lichnerowicz formula proved in  \cite[Thm. 3.9.3]{MR3099098}.

\begin{thm}
\label{Lichnerowicz}
For $u>0$,
\begin{multline}
\label{Lichnerowiczeq}
B_u^2 = -\frac{u}{2}(\n_{u,e_i})^2 +\frac{uK^X}{8} + \frac{u}{4} c(e_i)c(e_j)L'^\xi_{i,j}+\sqrt{\frac{u}{2}}c(e_i)f^\alpha L'^\xi_{i,\alpha}+ \frac{f^\alpha f^\beta}{2}L'^\xi_{\alpha,\beta} \\
-u\psi_{1/\sqrt{u}}\Big( \db^M\partial^Mi\omega\Big)^c\psi_{\sqrt{u}} - \frac{u}{16}\Big \| \left( \db^X-\partial^X\right)i\omega^X \Big\|^2_{\Lambda^\bullet(T^*_\R X)}.
\end{multline}
Thus, $B_u^2$ is a fiberwise elliptic second order differential operator. In particular, its heat kernel $\exp(-B_u^2)$ exists.
\end{thm}

\begin{rem}
In this theorem, as in the whole article, we use the usual following notation: if $C$ is a smooth section of $T_\R^*X\otimes\End\big(\Wedge(T^*X) \otimes \xi\big)$, then
\begin{equation}
\label{convention}
\left( \n^{\Wedge\otimes \xi}_{e_i} + C(e_i)\right)^2 = \sum_i \left(\n^{\Wedge\otimes \xi}_{e_i} + C(e_i )\right)^2 - \n^{\Wedge\otimes \xi}_{\sum_i \n^{TX}_{e_i} e_i} - C\left(\sum_i \n^{TX}_{e_i} e_i \right).
\end{equation}
\end{rem}


\subsection{The cohomology of the fiber}
\label{cohomologiedelafibre}

We assume that the direct image $R^\bullet \pi_*\xi$ of $\xi$ by $\pi$ is locally free. For $b\in B$, let $H^\bullet(X_b,\xi|_{X_b})$ be the cohomology of the sheaf of holomorphic sections of $\xi$ over $X_b$. Then the $H^\bullet(X_b,\xi|_{X_b})$'s form a $\Z$-graded holomorphic vector bundle $H(X,\xi|_{X})$ on $B$ and $R^\bullet \pi_*\xi=H^\bullet(X,\xi|_{X})$.

For $b\in B$, let $K(X_b , \xi|_{X_b}) = \ker (D^{X_b})$. By Hodge theory, we know that for every $b\in B$
\begin{equation}
\label{iden}
H^\bullet(X_b,\xi|_{X_b}) \simeq K^\bullet(X_b,\xi|_{X_b}), 
\end{equation}

The Hermitian product \eqref{prodsurEb} on $E_b$ restricts to the right and side of \eqref{iden}, so $h^{TX}$ and $h^{\xi}$ induce a metric $h^{H(X,\xi|_{X})}$ on the holomorphic vector bundle $H(X,\xi|_{X})$, for which the $H^k(X,\xi|_{X})$ are mutually orthogonal.

\begin{defn}
Let $\n^{H(X,\xi|_{X})}$ be the Chern connection on $\big(H(X,\xi|_{X}), h^{H(X,\xi|_{X})}\big)$.
\end{defn}

For $\ol{U} \in T^{0,1}B$ and $s\in \smooth(B,E)$, set ${\n_{\ol{U}}^{E,u}}''=\lie_{\ol{U}^H}$. Let ${\n^{E,u}}'$ be the adjoint of ${\n^{E,u}}''$ defined by $\langle {\n^{E,u}}' s , s'\rangle = \langle s, {\n^{E,u}}''s' \rangle$, and let $\n^{E,u}={\n^{E,u}}'+{\n^{E,u}}''$ (see \cite[Chp. 3]{MR3099098}).

Let $P^{K_b}$ be the orthogonal projection form $E_b$ on $K(X_b,\xi|_{X_b})$. We define the connection $\n^{K(X,\xi|_{X})}$ on $K(X,\xi|_{X})$ by
\begin{equation}
\label{n^K}
\n^{K(X,\xi|_{X})} = P^K \n^{E,u} P^K.
\end{equation}

The following proposition is proved in \cite[Prop. 4.10.3]{MR3099098}.

\begin{prop}
Under the identification \eqref{iden}, the connections $\n^{H(X,\xi|_{X})}$ and $\n^{K(X,\xi|_{X})}$ agree.
\end{prop}


\subsection{The analytic torsion forms}
\label{formesdetorsionsanalytic}

\begin{defn}
\label{espacesdeformes}
For any complex manifold $Z$, we denote by $Q^Z$ the vector space of real forms on $Z$ which are sum of forms of type $(p,p)$. We also denote by $Q^{Z,0}$ the subspace of the $\alpha\in Q^Z$ that can be written $\alpha = \partial \beta + \db \gamma$ for some $\beta$, $\gamma$ smooth form on $Z$.
\end{defn}

Let $N_V$ be the number operator defining the $\Z$-grading on $\Wedge(T^*X)\otimes \xi$ and on $E$.
\begin{defn}
For $u>0$, set
\begin{equation}
\label{defNueq}
N_u = N_V + i \frac{\omega^{H}}{u}.
\end{equation}
\end{defn}

Let $\Phi$ be the endomorphism of $\Lambda^{\text{even}}(T_\R^*B)$ defined by 
\begin{equation}
\Phi\colon\alpha \in \Lambda^{2k}(T_\R^*B) \mapsto (2i\pi)^{-k}\alpha.
\end{equation}

Let $\tau$ be the involution defining the $\Z_2$-graduation on $E$. If $H\in \End ( E)$ is trace class, we define its \emph{supertrace} $\trs[H]$ by
\begin{equation}
\trs[H]= \tr[\tau H].
\end{equation}
 We extend the supertrace to get an application $\trs \colon \Omega^\bullet(B,E)\to \Omega^\bullet(B)$. 

\begin{thm}[see {\cite[Thm. 4.5.2]{MR3099098}}]
For any $u>0$, the forms $\Phi \trs \left[ \exp(-B_u^2) \right]$ and \\ $\Phi \trs \left[ N_u \exp(-B_u^2) \right]$ lie in $Q^B$. Moreover the following identity holds in $Q^B$
\begin{equation}
\derpar{}{u} \Phi \trs \left[ \exp(-B_u^2) \right] = -\frac{1}{u} \frac{\db \partial}{2i\pi} \Phi \trs \left[ N_u \exp(-B_u^2) \right].
\end{equation}
\end{thm}

Let $(\alpha_u)_{u\in \R^+}$ and $\alpha$ be smooth forms on $B$. We say that as $u\to + \infty $ (resp. $u\to 0$), $\alpha_u = \alpha + O(f(u))$, if and only if for any compact set $K$ in $B$ and any $k \in \N$ their exists $C>0$ such that for every $u\geq 1$ (resp. $u\leq1$) the norm of all the derivatives of order $\leq k$ of $\alpha_u - \alpha$ over $K$ is bounded by $Cf(u)$.

\begin{thm}[see {\cite[Thm. 4.10.4]{MR3099098}}]
\label{thmutoinf}
As $u\to +\infty$,
\begin{equation}
\begin{aligned}
\Phi \trs \left[ \exp(-B_u^2) \right] &= \Phi \trs \left[ \exp(-(\n^{H(X,\xi|_{X})})^2) \right] + O\left(\frac{1}{\sqrt{u}} \right), \\
\Phi \trs \left[N_u \exp(-B_u^2) \right] &= \Phi \trs \left[ N_V \exp(-(\n^{H(X,\xi|_{X})})^2) \right] + O\left(\frac{1}{\sqrt{u}} \right).
\end{aligned}
\end{equation}
\end{thm}

\begin{thm}[see {\cite[Prop. 4.6.1]{MR3099098}}]
\label{thmuto0}
There exist locally computable forms $(c_j \in Q^B)_{j\geq-m}$ and $(C_j \in Q^B)_{j\geq-m}$ such that for $u\to 0$ and for any $k\in \N$,
\begin{equation}
\label{asymp0}
\begin{aligned}
\Phi \trs \left[ \exp(-B_u^2) \right] =  \sum_{j=-m}^k c_j u^j + O(u^{k+1}), \\
\Phi \trs \left[N_u \exp(-B_u^2) \right] = \sum_{j=-m}^k C_j u^j + O(u^{k+1}).
\end{aligned}
\end{equation}
\end{thm}

Following \cite[Def. 2.19]{MR929147}, \cite[Def. 3.8]{bismut-kohler} and \cite[(4.11.3)]{MR3099098}, we can now define the analytic torsion forms.

For $s\in \C$, $\Re (s)>1$, by Theorem \ref{thmuto0}, we can set
\begin{equation}
\zeta^1(s) = -\frac{1}{\Gamma(s)}\int_0^1 u^{s-1}\Phi \left\{ \trs \left[N_u \exp(-B_u^2) \right]  -   \trs \left[ N_V \exp(-(\n^{H(X,\xi|_{X})})^2) \right]  \right\}du,
\end{equation}
and $\zeta^1$ has a meromorphic extension to $\C$, which is holomorphic on $\{ |\Re(s)|<1/2\}$.

Similarly for $s\in \C$, $\Re (s)<1/2$, Theorem \ref{thmutoinf} allows us to define
\begin{equation}
\zeta^2(s) = -\frac{1}{\Gamma(s)}\int_1^{+\infty} u^{s-1}\Phi \left\{ \trs \left[N_u \exp(-B_u^2) \right]  -   \trs \left[ N_V \exp(-(\n^{H(X,\xi|_{X})})^2) \right]  \right\}du.
\end{equation}
Here again, $\zeta^2$ has a holomorphic extension on $\{ |\Re(s)|<1/2\}$.

Now, for $s\in \C$, $ |\Re(s)|<1/2$, define the holomorphic function
\begin{equation}
\label{defzeta}
\zeta(s) = \zeta^1(s)+\zeta^2(s).
\end{equation}

\begin{defn}
\label{def-torsion}
The \emph{holomorphic analytic torsion form} is the form
\begin{equation}
\mathscr{T}(\omega,h^\xi) := \zeta'(0).
\end{equation}
The components in the different degrees of $\mathscr{T}(\omega,h^\xi)$ are referred to as the \emph{holomorphic analytic torsion forms}.
\end{defn}

Using \eqref{asymp0}, we can write
\begin{equation}
\begin{aligned}
\mathscr{T}(\omega,h^\xi)=& - \int^1_0 \bigg\{ \Phi \trs \left[N_u \exp(-B_u^2) \right] -\sum_{j=-m}^0C_ju^j \bigg\} \frac{du}{u} \\
&-\int^{+\infty}_1 \Phi \left\{ \trs \left[N_u \exp(-B_u^2) \right]  -   \trs \left[ N_V \exp(-(\n^{H(X,\xi|_{X})})^2) \right]  \right\} \frac{du}{u} \\
&+\sum_{j=-m}^{-1}\frac{C_j}{j} +\Gamma'(1) \left( C_0 -  \Phi  \trs \left[ N_V \exp(-(\n^{H(X,\xi|_{X})})^2) \right] \right).
\end{aligned}
\end{equation}

The following analogue to \cite[Thm. 3.9]{bismut-kohler} is proved in \cite[Thm. 4.11.2]{MR3099098}
\begin{thm}
\label{RRG}
The smooth form $\mathscr{T}(\omega,h^\xi)$ lies in $Q^B$. Moreover
\begin{equation}
\frac{\db\partial}{2i\pi} \mathscr{T}(\omega,h^\xi) = \ch\left( H(X,\xi|_{X}), h^{H(X,\xi|_{X})}\right) - c_0.
\end{equation}
\end{thm}


\subsection{The case of a {\kahler}  fibration}
\label{kahler-fibration}

Following \cite[Def. 1.4 and Thm. 1.5]{MR929147}, we say that the Hermitian  fibration $(\pi, \omega)$ is a {\kahler} fibration if $\omega$ is closed.

We assume in this subsection that $\omega$ is closed. Then by \cite[Thms. 3.7.1, 3.7.3 and 3.8.1]{MR3099098} the superconnection $B_u$ agrees we the one define in \cite[Def. 1.7]{bismut-kohler}, which is the usual Bismut superconnection.

Therefor, \eqref{defBu}, \eqref{defnablau} and \eqref{Lichnerowiczeq} turn repectively to
\begin{equation}
\left\{
\begin{aligned}
&B_u=\n^{E} + \sqrt{u} D^X - \frac{c(T)}{2\sqrt{2u}}, \\
&\n_{u,e_i} = \n^{\Wedge\otimes \xi}_{e_i}+\frac{1}{\sqrt{2u}} S_{i,j,\alpha}c(e_j)f^\alpha + \frac{1}{2u} S_{i,\alpha,\beta}f^\alpha f^\beta \qquad \text{and}\\
&B_u^2 = -\frac{u}{2}(\n_{u,e_i})^2 +\frac{uK^X}{8} + \frac{u}{4} c(e_i)c(e_j)L'^\xi_{i,j}+\sqrt{\frac{u}{2}}c(e_i)f^\alpha L'^\xi_{i,\alpha}+ \frac{f^\alpha f^\beta}{2}L'^\xi_{\alpha,\beta}.
\end{aligned}
\right.
\end{equation} 

Moreover, \cite[Thm. 2.2]{MR929147} sharpens \eqref{asymp0}: the forms $c_j$, for $j\leq0$, can be explicitly computed. For any Hermitian vector bundle $(F,h^F)$ with Chern connection $\n^F$ and curvature $R^F$ on $M$, set
\begin{equation}
\mathrm{ch}(F,h^F) = \tr\left[ \exp\left( -\frac{R^F}{2\ic\pi} \right)\right],\quad \mathrm{Td}(F,h^F) = \det \left( \frac{R^F/2\ic\pi}{\exp(R^F/2\ic\pi)-1} \right).
\end{equation}
Then by \cite[Thm. 2.2]{MR929147} we get
\begin{equation}
\left\{
\begin{aligned}
&c_j=0 \text{ for } j<0 \qquad \text{and} \\
&c_0=\int_X \mathrm{Td} ( TX,h^{TX}) \mathrm{ch}(\xi,h^\xi).
\end{aligned}
\right.
\end{equation}

Finally, by \cite[Thm. 1.5]{bismut-kohler} $\n^E$ preserves the metric on $E$ and by \cite[Thm. 3.2]{bismut-kohler} we have
\begin{equation}
\label{nablaHetnablaEkahler}
\n^{H(X,\xi|_{X})} = P^K \n^{E} P^K.
\end{equation}
Then Theorem \ref{RRG} becomes \cite[Thm. 3.9]{bismut-kohler}, that is 
\begin{equation}
\frac{\db\partial}{2i\pi} \mathscr{T}(\omega,h^\xi) = \ch\left( H(X,\xi|_{X}), h^{H(X,\xi|_{X})}\right) - \int_X \mathrm{Td} ( TX,h^{TX}) \mathrm{ch}(\xi,h^\xi).
\end{equation}


\section{The asymptotic of the torsion associated to high power of a line bundle}
\label{Sectionlinebundle}

The purpose of this section is to prove Theorem \ref{mainthm}.

We recall some notations. Let $M$ and $B$ be two complex manifolds. Let $\pi \colon M \to B$ be a holomorphic fibration with compact fiber $X$ of dimension $n$. We suppose that we are given $(\pi, \omega)$ a structure of Hermitian  fibration. 

Let $(\xi, h^\xi)$ be a holomorphic Hermitian vector bundle on $M$, and let $(L,h^L)$ be a holomorphic Hermitian line bundle on $M$. We denote the curvature of the Chern connection of $L$ by $R^L$. Recall that by Assumption \ref{RLpositive}, $R^L$ is assumed to be positive along the fibers. We define
\begin{equation}
\label{ThetaM-X}
\Theta^M = \frac{\ic}{2\pi}R^L \quad \text{and} \quad \Theta^X = \frac{\ic}{2\pi}R^L|_{T_\R X\times T_\R X}.
\end{equation}
We extend $\Theta^X$ to $T_\R M=T_\R X\oplus (T_\R X)^{\perp,\Theta^M}$ by zero. 
 
We have also assumed that the direct image $R^i\pi_*(\xi\otimes L^p)$ is locally free (for $p$ large). We will use all the  constructions of Section \ref{torsion} associated with $(\xi\otimes L^p, h^{\xi\otimes L^p})$ instead of $(\xi, h^\xi)$ (where of course $h^{\xi\otimes L^p}$ is induced by $h^\xi$ and $h^L$). The corresponding objects will be denoted by

\begin{alignat}{3}
\label{defobjetsavecp}
& E_{p,b}^k = \smooth \left( X_b, \left(\Lambda^{0,k}(T^*X)\otimes \xi\otimes L^p \right)|_{X_b} \right), & \quad 
& \n^p = \n^{E_p,LC}, \notag\\
& \db^p = \text{Dolbeault operator of } E_p, &
& D_p = \db^p + \db^{p,*},\\
& B_p = \text{associated superconnection as in \eqref{defBu}}, &
&B_{p,u}=\sqrt{u} \psi_{1/\sqrt{u}}B_p\psi_{\sqrt{u}}.\notag
\end{alignat}

We also denote by  $\mathscr{T}(\omega, h^{\xi\otimes L^p})$  the associated analytic torsion forms.

Theorem \ref{mainthm} is the family version of \cite{bismut-vasserot-89}. The strategy of proof is similar, but differences appear in the proof of the intermediate results due to the horizontal differential forms appearing in $B_p^2$. One of the first consequence is that, unlike $D_p^2$, the operator $B_p^2$ is not self-adjoint, and one has to take a nilpotent part (the part in positive degree along the basis) into account when estimating resolvants or heat kernels (compare for instance the proofs of \cite[(20)]{bismut-vasserot-89} and of Theorem \ref {dominationugrand}). An other consequence is the limit of the heat kernel involves exponential of terms coupling horizontal forms and vertical Clifford variables, which make the computations of the super-traces much more complicated (see Theorem \ref{expressionzetatilde'(0)thm}). Note also that in all our results of smooth convergence, we have to take into account the derivatives along the basis $B$.

To simplify the statements in the following, we will assume that $B$ is compact. However, the reader should be aware of the fact that the constants appearing in the sequel depends on the compact subset of $B$ we are working on.

This section is organized as follows.  In Subsection \ref{localization}, we show that our problem is local. In Subsection \ref{rescalingBpu}, we rescale the Bismut superconnection and compute the limit operator. In Subsection \ref{sectioncvce}, we obtain the corresponding convergence of the heat kernel. In Subsection \ref{Asymptoticofthetorsionforms}, we prove our main theorem, using the result  proved in Subsection \ref{demothm2}.


\subsection{Localization of the problem}
\label{localization}

Fix $b_0\in B$. In this section, we will work along the fiber $X_{b_0}$, which will be denoted simply by $X$.

For $\e>0$ and $x_0\in X$, we denote by $B^X(x_0,\e)$ and $B^{T_{\R,x_0}X}(0,\e)$ the open balls in $X$ and $T_{\R,x_0}X$ with center $x_0$ and $0$ and radius $\e$ respectively. If $\exp^X_{x_0}$ is the exponential map of $X$, then for $\e$ small enough, $ Z\in B^{T_{\R,x_0}X}(0,\e) \mapsto \exp^X_{x_0}(Z) \in B^X(x_0,\e)$ is a diffeomorphism, which gives local coordinates by identifying $T_{\R,x_0}X$ with $\R^{2n}$ via an orthonormal basis $\{e_i \}$ of $T_{\R,x_0}X$. From now on, we will always identify $ B^{T_{\R,x_0}X}(0,\e)$ and $ B^X(x_0,\e)$.

Let $\mathrm{inj}^X$ be the injectivity radius of $X$ and let $\e \in ]0,\mathrm{inj}^X/4[$. Such an $\e$ can be chosen uniformly for $b_0$ varying in a compact subset of $B$.

Let $x_1,\dots x_N$ be points of $X$ such that $\{U_{x_k} = B^X(x_k,\e)\}_{k=1}^N$ is an open covering of $X$. On each $U_{x_k}$ we identify $\xi_Z$, $L_Z$ and $\Wedge(T_Z^*X)$ to $\xi_{x_k}$, $L_{x_k}$ and $\Wedge(T_{x_k}^*X)$ by parallel transport with respect to $\n^\xi$, $\n^L$ and $\n^{\Wedge,LC}$ along the geodesic ray $t \in [0,1]\mapsto tZ$. We fixe for each $k=1,\dots,N$ an orthonormal basis $\{e_i\}_i$ of $T_{\R,x_k}X$ (without mentioning the dependence on $k$).

We denote by $\nabla_U$ the ordinary differentiation operator in the direction $U$ on $T_{x_k}X$.

We define the vector bundle $\E_p$ over $X$ by
\begin{equation}
\E_{p}:=\Lambda_{b_0}^\bullet(T^*_\R B)\otimes \left(\Wedge(T^*X)\otimes \xi \otimes L^p\right).
\end{equation}
Note here that $\Lambda_{b_0}^\bullet(T^*_\R B)$ is a trivial bundle over $X$.

Let $\{\varphi_k\}_k$ be a partition of unity subordinate to $\{U_{x_k}\}_k$. For $\ell \in \N$, we define a Sobolev norm $||\cdot||_{\sob^\ell(p)}$ on the $\ell$-th Sobolev space $\sob^\ell (X,\E_p)$ by
\begin{equation}
\label{defnormeH^l(p)}
||s||^2_{\sob^\ell(p)} = \sum_k \sum_{d=0}^\ell \sum_{i_1,\dots,i_d=1}^d ||\n_{e_{i_1}}\dots \n_{e_{i_d}}(\varphi_k s)||^2_{L^2}.
\end{equation}

\begin{lemme}
\label{estimeeelliptique(p)}
For any  $m\in \N$, there exists $C_m>0$  such that for any $p\in N$, $u>0$ and $s\in \sob^{2m+2}(X,\E_p)$,
\begin{equation}
\label{estimeeelliptique(p)-eq}
||s||^2_{\sob^{2m+2}(p)}\leq C_m  p^{4m+4} \sum_{j=0}^{m+1}p^{-4j}||B_{p}^{2j}s||_{L^2}.
\end{equation}
\end{lemme}

\begin{proof}
Let $\widetilde{e_i}(Z)$ be the parallel transport of $e_i$ with respect to $\n^{T_\R X,LC}$ along the curve $t\in [0,1]\mapsto tZ$. Then $\{\widetilde{e_i}\}_i$ is an orthonormal frame of $T_\R X$.

Let $\Gamma^\xi$, $\Gamma^L$ and $\Gamma^{\Wedge,LC}$ be the corresponding connection form of $\n^\xi$, $\n^L$ and $\n^{\Wedge,LC}$ with respect to any fixed frame for $\xi$, $L$ and $\Wedge(T^*X)$ which is parallel along the curve $t\in [0,1]\mapsto tZ$ under the trivialization on $U_{x_k}$. Let $\n^p_1 = \n_1 \otimes 1 + 1\otimes \n^{L^p}$ be the connection on $\Wedge(T^*X)\otimes L^p \otimes \xi$ corresponding to $\n_u$ in \eqref{defnablau} (with $u=1$), replacing $\xi$ by $\xi \otimes L^p$. Then on $U_{x_k}$ we have 
\begin{multline}
\label{nupaveclesGamma}
\n^p_{1,\widetilde{e_i}} = \n_{\widetilde{e_i}}+(\Gamma^{\Wedge,LC}+\Gamma^\xi+p\Gamma^L)(\widetilde{e_i})+\frac{1}{\sqrt{2}} S(\widetilde{e_i},\widetilde{e_j},f_\alpha)c(\widetilde{e_j})f^\alpha \\
+ \frac{1}{2} S(\widetilde{e_i},f_\alpha,f_\beta) f^\alpha f^\beta +\frac{1}{2}\Big(i_{\widetilde{e_i}} \left( \db^M-\partial^M\right)i\omega\Big)^c.
\end{multline}

Let $D^X=\db^X + \db^{X,*}$ be the Dirac operator on $ \Wedge(T^*X)\otimes \xi$, and $B^\xi$ the superconnection on $B$ associated with $(\omega,\xi,h^\xi)$. Then on $U_{x_k}$, $D^X$ (resp. $B^\xi$) can be seen as an operator on $ \pi^*\Lambda^\bullet (T^*_\R B)\otimes \Wedge (T^*X)\otimes \xi \otimes L^p$ because the bundle $\pi^*\Lambda^\bullet (T^*_\R B) \otimes L^p$ (resp. $L^p$) is trivialized. Then, using \eqref{Lichnerowiczeq}, \cite[Thm. 1.4.7]{ma-marinescu} (which is \eqref{Lichnerowiczeq} in the case where the base $B$ is a point) and \eqref{nupaveclesGamma}, we find that locally,
\begin{align}
\label{Bup2=f(DX2,u,p)}
B_{p}^2 &= B^{\xi,2} + p \0_1 + p\0_0^1+ p^2\0_0^2 \\ 
&=  D^{X,2} + R + p \0_1 + p\0_0^1+ p^2\0_0^2 \notag
\end{align}
where $R$, $\0_1$ (resp. $\0_0^1$, $\0_0^2$) are operators of order $1$ (resp. $0$).

From \eqref{Bup2=f(DX2,u,p)}, there exists $C>0$ such that for $s\in \sob^\ell (X,\E_p)$,
\begin{equation}
\label{estimeeavecBp}
||s||_{\sob^2(p)} \leq C\left( ||B_{p}^2s||_{L^2}+p^2||s||_{L^2} \right).
\end{equation}

Let $Q$ be a differential operator of order $2m$ with scalar principal symbol and with compact support in $U_{x_j}$. Then \eqref{estimeeavecBp} implies 
\begin{equation}
\begin{aligned}
||Qs||_{\sob^2(p)} &\leq C\left( ||B_{p}^2Qs||_{L^2}+p^2||Qs||_{L^2}\right) \\
& \leq C'\left(||QB_{p}^2s||_{L^2}+p||s||_{\sob^{2m+1}(p)}+p^2||s||_{\sob^{2m-1}(p)}+p^2||Qs||_{L^2} \right).
\end{aligned}
\end{equation}
Hence we get \eqref{estimeeelliptique(p)-eq} by induction.
\end{proof}

Let $f\colon \R \to [0,1]$ be a smooth even function such that
\begin{equation}
f(t)=\left \{
\begin{aligned}
&1 \text{ for } |t|<\e/2, \\
& 0 \text{ for } |t|>\e.
\end{aligned}
\right.
\end{equation}

For $u>0$, $\varsigma\geq1$ and $a\in \C$, set
\begin{equation}
\label{defFuGuHu}
\begin{aligned}
&\F_u(a)=\int_\R e^{iv\sqrt{2}a}\exp(-v^2/2)f(\sqrt{u}v)\frac{dv}{\sqrt{2\pi}}, \\
&\G_u(a)=\int_\R e^{iv\sqrt{2}a}\exp(-v^2/2)(1-f(\sqrt{u}v))\frac{dv}{\sqrt{2\pi}}, \\
&\HH_{u,\varsigma}(a)=\int_\R e^{iv\sqrt{2}a}\exp(-v^2/2u)(1-f(\sqrt{\varsigma}v))\frac{dv}{\sqrt{2\pi}}.
\end{aligned}
\end{equation}

These functions are even holomorphic functions, thus there exist holomorphic functions $\tilde{\F}_u$, $\tilde{\G}_u$ and $\tilde{\HH}_{u,\varsigma}$ such that
\begin{equation}
\label{defFuGuHuavectilde}
\tilde{\F}_u(a^2)=\F_u(a)\, , \quad \tilde{\G}_u(a^2)=\G_u(a)\quad \text{and} \quad \tilde{\HH}_{u,\varsigma}(a^2)=\HH_{u,\varsigma}(a).
\end{equation}
Moreover, the restriction of $\tilde{\F}_u$ and $\tilde{\G}_u$  to $\R$ lies in the Schwartz space $\mathcal{S}(\R)$, and
\begin{equation}
\label{liensFGH}
\tilde{\G}_{\frac{u}{p}}\left(\frac{u}{p}a\right)=\tilde{\HH}_{\frac{u}{p},1}(a) \quad \text{and} \quad \tilde{\F}_u(vB_{p}^2)+\tilde{\G}_u(vB_{p}^2)=\exp\left( -vB_{p}^2\right) \text{ for }v>0.
\end{equation}

Let $\tilde{\G}_u(vB_{p}^2)(x,x')$ be the smooth kernel of $\tilde{\G}_u(vB_{p}^2)$ with respect to $dv_X(x')$. 

We still denote by $\pi$ the projection $\pi \colon X\times_B X\to B$ from the fiberwise product $X\times_B X$ to $B$. For $V$, $V'$ two bundle over $M$, we define the bundle $V\boxtimes V'$ on $X\times_B X$ by
\begin{equation}
(V\boxtimes V')_{(b,x,x')} = V_{(b,x)}\otimes V'_{(b,x')}
\end{equation}
for $b\in B$ and $x,x'\in X_b$. Then $\tilde{\G}_u(vB_{p}^2)(\cdot,\cdot)$ is a section of $ \E_p \boxtimes \E_p^*$ over $X\times_B X$. Let $\n^{\E_p}$ be the connection on $\E_p$ induced by $\n^{\Lambda^\bullet(T^*_\R B)}$, $\n^{\Wedge,LC}$, $\n^L$ and $\n^\xi$, and let $\nabla^{ \E_p \boxtimes \E_p^*}$ be the induced connection on $ \E_p \boxtimes \E_p^*$. In the same way, let $h^{\E_p}$ be the metric on $\E_p$  induced by $h^{\Lambda^\bullet(T^*_\R B)}$, $h^{\Wedge,LC}$, $h^L$ and $h^\xi$, and let $h^{ \E_p \boxtimes \E_p^*}$ be the induced metric on $ \E_p \boxtimes \E_p^*$.
\begin{prop}
\label{lepbestlocal}
For any $m\in \N$, $\e>0$, there exist  $C>0$ and $N\in \N$ such that for any $u> 0$ and any $p\in \N^*$,
\begin{equation}
\label{lepbestlocal-eq}
\left|   \tilde{\G}_{\frac{u}{p}} \Big(\frac{u}{p}B_{p}^2\Big)(\cdot\, ,\cdot) \right|_{\mathscr{C}^m} \leq Cp^N  \exp \left( -\frac{\e^2p}{16u} \right),
\end{equation}
Where the $\mathscr{C}^m$-norm is induced by  $\nabla^{ \E_p \boxtimes \E_p^*}$ and $h^{ \E_p \boxtimes \E_p^*}$.
\end{prop}

\begin{proof}
Observe first that by \eqref{liensFGH}
\begin{equation}
\label{onserameneau=1}
 \tilde{\G}_{\frac{u}{p}}  \left(\frac{u}{p}B_{p}^2\right) = \tilde{\HH}_{\frac{u}{p},1}(B_{p}^2).
\end{equation}

Moreover, as $i^ma^me^{iva} = \derpar{^m}{v^m}e^{iva}$, we can integrate by part the expression of $a^m\HH_{u,\varsigma}(a)$ given in \eqref{defFuGuHu} to obtain that for any $m\in \N$ and $c>0$, there is a $C_{m,c}>0$ such that $u>0$ and $\varsigma\geq 1$,
\begin{equation}
\label{sup(a^mH(a))}
\sup_{|\Im(a)|\leq c} |a^m\HH_{u,\varsigma}(a)| \leq C_{m,c} \varsigma^{\frac{m}{2}}\exp\left(-\frac{\e^2}{16u\varsigma}\right).
\end{equation}
For $c>0$, let $V_c$ be the image of $\{a\in \C \, : \, |\Im(a)|\leq c\}$ by the map $a\mapsto a^2$. Then
\begin{equation}
\label{Vc}
V_c = \Big\{a\in \C \, : \, \Re(a) \geq \frac{1}{4c^2} \Im(a)^2 - c^2\Big\}.
\end{equation}
Form \eqref{defFuGuHuavectilde} and \eqref{sup(a^mH(a))}, we deduce that
\begin{equation}
\label{sup(a^mHtilde(a))}
\sup_{a\in V_c} |a^m\tilde{\HH}_{u,\varsigma}(a)| \leq C_{m,c} \varsigma^{\frac{m}{2}}\exp\left(-\frac{\e^2}{16u\varsigma}\right).
\end{equation}

We will prove  Proposition~\ref{lepbestlocal} thanks to \eqref{onserameneau=1}, \eqref{sup(a^mHtilde(a))}  and Lemma \ref{estimeeelliptique(p)}. We first need the following lemma.
\begin{lemme}
\label{f(Bp2)}
Let $m\in \N$ and $\phi(a)=a^m\tilde{\HH}_{\frac{u}{p},1}(a)$,  then there exist $K_m>0$  and an integer $k_m\in \N$ such that
\begin{equation}
||\phi(B_{p}^2)s||_{L^2}\leq K_m p^{k_m}\exp \left( -\frac{\e^2p}{16u} \right)||s||_{L^2}.
\end{equation}
\end{lemme}

\begin{proof}
By Bismut's Lichnerowicz formula \eqref{Lichnerowiczeq}, \cite[Thm. 1.4.7]{ma-marinescu}  and \eqref{nupaveclesGamma}, we have
\begin{equation}
\label{Bp2etDp2}
\begin{aligned}
&B_{p}^2 = D_p^2+R_p,\\
&R_p \in \C[p] \otimes \Lambda^{\geq 1}(T^*_\R B) \otimes \mathrm{Op}^{\leq 1}_X\big(\Wedge(T^*X)\otimes \xi\big),
\end{aligned}
\end{equation}
where $\mathrm{Op}^{\leq 1}_X\big(\Wedge(T^*X)\otimes \xi \big)$ denotes the set of differential operators along the fiber $X$ on $\Wedge(T^*X)\otimes \xi $ of order $\leq 1$. We deduce the following fundamental fact:
\begin{equation}
\label{SpBp1}
\Sp(B_{p}^2)=\Sp(D_p^2).
\end{equation}
Here, $\Sp$ is our notation for the spectrum. Indeed, as $R_p$ has positive degree in $\Lambda^{\bullet}(T^*_\R B)$, we have for $\lambda \notin \Sp(D_p^2)$
\begin{equation}
\label{resolvantesBp2etDp2}
(\lambda - B_{p}^2)^{-1} = (\lambda - D_p^2)^{-1} + (\lambda - D_p^2)^{-1}R_p(\lambda - D_p^2)^{-1} + \dots \quad \text{(finite sum)}.
\end{equation}
Now, $(\lambda - D_p^2)^{-1}$ is elliptic of order 2, so increases the Sobolev regularity by 2, and $R_p$ is of order 1, thus $(\lambda - B_{p}^2)^{-1}$ is a bounded operator when acting on the Sobolev space of order 0. This proves that $\lambda \notin \Sp(B_{p}^2)$. Exchanging the role of $B_{p}^2$ and $D_p^2$, we also prove that if $\lambda \notin \Sp(B_{p}^2)$, then $\lambda \notin \Sp(D_p^2)$, which shows \eqref{SpBp1}.

By \cite[Thm 1.5.8]{ma-marinescu}, there exist $C_L>0$ and $\mu_0>0$ such that 
\begin{equation}
\label{spdeDp2}
\Sp(D_p^2)\subset \{0\} \cup ]2p\mu_0-C_L,+\infty[.
\end{equation}

\begin{figure}[!h]
\hfill
\setlength{\unitlength}{4cm}
\begin{picture}(2.5,1.125)
\put(0.25,0.5){\vector(1,0){1.85}}
\put(0.70,0.025){\vector(0,1){0.95}}
\thicklines \put(0.50,0.25){\line(0,1){0.5}}
\thicklines \put(0.50,0.25){\line(1,0){1,55}}
\thicklines \put(0.50,0,75){\line(1,0){1,55}}
\put(0.20,0.55){$-C_L$}
\put(0.50,0.5){\circle*{0.029}}
\put(1.25,0.25){${}_\blacktriangleright$}
\put(1.25,0.75){${}_\blacktriangleleft$}
\put(1.25,0.85){$\mathcal{C}$}
\put(0.72,0.78){$i$}
\put(0.72,0.15){$-i$}
\put(0.63,0.4){0}
\end{picture}
\hfill \hfill
\caption{}
\label{contour-C}

\end{figure}

Let $\mathcal{C}$ be the contour in $\C$ defined by Figure \ref{contour-C}. By \eqref{Vc}, $\mathcal{C} \subset V_c$ for $c$ big enough. 

Note that by \eqref{spdeDp2} and the self-adjointness of $D_p^2$, there exists $C>0$ such that for $\lambda \in \mathcal{C}$,
\begin{equation}
\label{estimation-resolvante-Dp2}
|| (\lambda - D_p^2)^{-1}s ||_{L^2} \leq C ||s||_{L^2}.
\end{equation}
Moreover, for $\lambda \in \mathcal{C}$ and $x\in \R_+$, we have $\frac{x}{|\lambda-x|} \leq \frac{|\lambda|}{|\lambda-x|}+1 \leq C|\lambda|$, where $C$ does not depend on $x\in \R_+$. In particular, we have
\begin{equation}
\label{estimation-Dp2-resolvante-Dp2}
||D_p^2(\lambda - D_p^2)^{-1}s||_{L^2}\leq C|\lambda| ||s||_{L^2}.
\end{equation}

Now by \cite[(1.6.8)]{ma-marinescu} --which is \eqref{estimeeelliptique(p)-eq} with $m=0$ in the case where $B$ is a point-- and by \eqref{estimation-resolvante-Dp2} and \eqref{estimation-Dp2-resolvante-Dp2}, there is $l\in \N$ and and $C'>0$ such that for $\lambda\in \mathcal{C}$,
\begin{equation}
\label{estimationRpresolvanteDp2}
|| R_p(\lambda - D_p^2)^{-1}s ||_{L^2} \leq Cp^l ||(\lambda - D_p^2)^{-1}s||_{\sob^2(p)}  \leq C'|\lambda|p^{l+2} ||s||_{L^2}.
\end{equation}
 Thus, by \eqref{resolvantesBp2etDp2}, and \eqref{estimationRpresolvanteDp2}, we find
\begin{equation}
\label{estimation-resolvante-Bp2}
|| (\lambda - B_p^2)^{-1}s ||_{L^2} \leq C|\lambda|^k p^{k'} ||s||_{L^2}.
\end{equation}

By \eqref{sup(a^mHtilde(a))}, we have $|\phi(\lambda)|\leq C_{m+k+2,c} \exp \left( -\frac{\e^2p}{16u} \right)|\lambda|^{-(k+2)}$ for $\lambda\in\mathcal{C} \subset V_c$. Moreover, \begin{equation}
\phi(B_{p}^2) =\frac{1}{2i\pi}\int_\mathcal{C} \phi(\lambda)(\lambda - B_{p}^2)^{-1}d\lambda,
\end{equation}
Using this facts,  we get Lemma \ref{f(Bp2)} from \eqref{estimation-resolvante-Bp2}.
\end{proof}

Let $Q$ be a differential operator of order $2m$, $m\in \N$ with scalar principal symbol and with compact support in $U_{x_i}$. Observe that Lemmas \ref{estimeeelliptique(p)} and \ref{f(Bp2)} are still true if we replace $B_p$ therein by $B_p^*$, because $B_p^{*,2}$ has the same structure as in \eqref{Bup2=f(DX2,u,p)} and is equal to $D_p^2 +R_p^*$. Thus, using Lemmas \ref{estimeeelliptique(p)} and \ref{f(Bp2)}, we find that for $m'\in 2\N$,
\begin{equation}
\begin{aligned}
\left| \langle B_{p}^{m'} \tilde{\HH}_{\frac{u}{p},1}(B_{p}^2)Q s, s'\rangle \right | &= \left| \langle s, Q^* \tilde{\HH}_{\frac{u}{p},1}(B_{p}^{*,2})B_{p}^{*,m'}s'\rangle \right|  \\
& \leq CK p^{4m+k_m} \exp \left( -\frac{\e^2p}{16u} \right) \|s\|_{L^2} \|s'\|_{L^2}.
\end{aligned}
\end{equation}
Thus,
\begin{equation}
\left\| B_{p}^{m'} \tilde{\HH}_{\frac{u}{p},1}(B_{p}^2)Q s \right \|_{L^2} \leq CK p^{4m+k_m} \exp \left( -\frac{\e^2p}{16u} \right) \|s\|_{L^2}.
\end{equation}
We deduce from this estimate -- and using once again Lemmas \ref{estimeeelliptique(p)} and \ref{f(Bp2)} -- that if $P,Q$ are differential operators of order $2m',2m$ respectively and with compact support in $U_{x_i},U_{x_j}$ respectively, then there is a positive constant $C_{m,m'}$ such that
\begin{equation}
\label{estimationPH(B)Q}
\left\| P \tilde{\HH}_{\frac{u}{p},1}(B_{p}^2)Q s \right \|_{L^2} \leq C_{m,m'} p^{4m+k_m} \exp \left( -\frac{\e^2p}{16u} \right) \|s\|_{L^2}.
\end{equation}
By the Sobolev inequality and \eqref{estimationPH(B)Q}, we get
\begin{equation}
\label{normeCm-fibre}
\left|   \tilde{\HH}_{\frac{u}{p},1} (B_{p}^2)(\cdot \,,\cdot) \right|_{\mathscr{C}^m(X\times X)} \leq Cp^N  \exp \left( -\frac{\e^2p}{16u} \right).
\end{equation}
With this estimate and \eqref{onserameneau=1}, we get \eqref{lepbestlocal-eq} for the $\mathscr{C}^m$-norm in the directions of the fiber $X$.

We now turn to the derivatives in the directions of the base $B$.

Let $k\in \N$. Using \eqref{sup(a^mHtilde(a))} (see \cite[(11.57)]{MR1316553}), we see that there is a unique holomorphic function $\tilde{\HH}_{u,\varsigma,k}$ defined  on a neighborhood of $V_c$ such that 
\begin{equation}
\frac{\tilde{\HH}_{u,\varsigma,k}^{(k-1)}(a)}{(k-1)!}= \tilde{\HH}_{u,\varsigma}(a)
\end{equation}
and for $u>0$ and $\varsigma\geq 1$,
\begin{equation}
\label{sup(a^mHtildek(a))}
\sup_{a\in V_c} |a^m\tilde{\HH}_{u,\varsigma,k}(a)| \leq C \varsigma^{\frac{m}{2}}\exp\left(-\frac{\e^2}{16u\varsigma}\right).
\end{equation}

For any $q,k\in \N$ and $U\in T_\R B$, we have
\begin{equation}
\label{repintegralle-derivation}
\big(\n^{\E_p\boxtimes\E_p^*}_{U^H}\big)^q \tilde{\G}_{\frac{u}{p}} \Big(\frac{u}{p}B_{p}^2\Big) = \frac{1}{2i\pi} \int_\mathcal{C} \tilde{\HH}_{\frac{u}{p},1,k}(\lambda) \big(\n^{\E_p\boxtimes\E_p^*}_{U^H}\big)^q (\lambda - B_p^2)^{-k}d\lambda,
\end{equation}
where $U^H$ denotes the horizontal lift of $U$ in $T_{B,\R}^HM$.

We now prove the analogue of Lemma \ref{f(Bp2)} for $\big(\n^{\E_p\boxtimes\E_p^*}_{U^H}\big)^q \tilde{\G}_{\frac{u}{p}} \Big(\frac{u}{p}B_{p}^2\Big)$:
\begin{lemme}
\label{f(Bp2)-derivation}
Let $q,m,m'\in \N$. There exist $K_{q,m,m'}>0$  and an integer $k_{q,m,m'}\in \N$ such that
\begin{equation}
\left\|B_p^{2m}\big(\n^{\E_p\boxtimes\E_p^*}_{U^H}\big)^q \tilde{\G}_{\frac{u}{p}} \Big(\frac{u}{p}B_{p}^2\Big)B_p^{2m'}s\right\|_{L^2}\leq K_{q,m,m'} p^{k_{q,m,m'}}\exp \left( -\frac{\e^2p}{16u} \right)||s||_{L^2}.
\end{equation}
\end{lemme}

\begin{proof}
We choose $k\in \N$ so that $k\geq 2(m+m')+q+1$. Then $B_p^{2m}\big(\n^{\E_p\boxtimes\E_p^*}_{U^H}\big)^q (\lambda - B_p^2)^{-k}B_p^{2m'}$ can be written as a sum of terms
\begin{equation}
A_1(\lambda - B_p^2)^{-1}A_2(\lambda - B_p^2)^{-1} \dots A_{k+1}(\lambda - B_p^2)^{-1},
\end{equation}
where 
\begin{equation}
A_i \in \Big\{1,B_p, \big(\n^{\E_p\boxtimes\E_p^*}_{U^H}\big)^{q'}B_p^2, \Big[ B_p,\big(\n^{\E_p\boxtimes\E_p^*}_{U^H}\big)^{q'}B_p^2\Big] \,: \, 0\leq q'\leq q\Big\}.
\end{equation}
In any case, $A_i$ is a polynomial in $p$ with values in the differential operators along the fiber of order less than 2 (for the last type of term in the above list, we use that $B_p$ is of order 1 and that $B_p^2$ as a scalar principal symbol). As a consequence, we find from \eqref{resolvantesBp2etDp2}, \eqref{estimation-Dp2-resolvante-Dp2}, \eqref{estimationRpresolvanteDp2} and \eqref{estimation-resolvante-Bp2} that
\begin{equation}
\label{estimationAresolvanteBp2}
|| A_i(\lambda - B_p^2)^{-1}s ||_{L^2} \leq Cp^l ||(\lambda - B_p^2)^{-1}s||_{\sob^2(p)}  \leq C|\lambda|^a p^b ||s||_{L^2}.
\end{equation}
By the decomposition indicated in the begging of the proof, this yields
\begin{equation}
\label{Bp-der-Bp}
\left\| B_p^{2m}\big(\n^{\E_p\boxtimes\E_p^*}_{U^H}\big)^q (\lambda - B_p^2)^{-k}B_p^{2m'} s \right\|_{L^2} \leq C|\lambda|^c p^d ||s||_{L^2}.
\end{equation}

From \eqref{sup(a^mHtildek(a))}, \eqref{repintegralle-derivation} and \eqref{Bp-der-Bp}, we deduce Lemma \ref{f(Bp2)-derivation}.
\end{proof}

Using Lemma \ref{f(Bp2)-derivation} in the same way as we used Lemma \ref{f(Bp2)} to prove \eqref{estimationPH(B)Q} and \eqref{normeCm-fibre}, we find
\begin{equation}
\label{normeCm-fibre+base}
\left| \big(\n^{\E_p\boxtimes\E_p^*}_{U^H}\big)^q  \tilde{\HH}_{\frac{u}{p},1} (B_{p}^2)(\cdot \,,\cdot) \right|_{\mathscr{C}^m(X\times X)} \leq Cp^N  \exp \left( -\frac{\e^2p}{16u} \right).
\end{equation}
Which completes the proof of Proposition~\ref{lepbestlocal}.
\end{proof}

\begin{cor}
 \label{lepbestlocal-cor}
For any $m\in \N$, $\e>0$, there exist  $C(u)>0$ a rational fraction in $\sqrt{u}$ and $N\in \N$ such that for any $u> 0$ and any $p\in \N^*$,
\begin{equation}
\left| \psi_{1/\sqrt{p}}  \tilde{\G}_{\frac{u}{p}} (B_{p,u/p}^2)(\cdot,\cdot) \right|_{\mathscr{C}^m} \leq C(u)p^N  \exp \left( -\frac{\e^2p}{16u} \right).
\end{equation}
\end{cor}

\begin{proof}
As $B_{p,u}=\sqrt{u} \psi_{1/\sqrt{u}}B_p\psi_{\sqrt{u}}$,  Corollary \ref{lepbestlocal-cor} follows from Proposition \ref{lepbestlocal}.
\end{proof}


\subsection{Rescaling $B_{p}$}
\label{rescalingBpu}

 Fix  $b_0 \in B$ and $x_0 \in X_{b_0}$. In this section, we will again work along the fiber $X_{b_0}$, which will be again  denoted simply by $X$. For the rest of this section, we fix $\{ w_j\}$  an orthonormal basis of $T_{x_0}^{(1,0)}X$, with dual basis  $\{ w^j\}$, and we construct an orthonormal basis $\{ e_i\}$ of $T_{\R,x_0}X$ from $\{ w_j\}$ as follows:
\begin{equation}
\label{defnbase}
e_{2j-1}=\frac{1}{\sqrt{2}} \left( w_j + \bw_j \right) \text{ and } e_{2j}=\frac{\ic}{\sqrt{2}} \left( w_j - \bw_j \right), \text{ for } 1\leq j\leq n.
\end{equation}

 For $\e>0$ small enough, we identify $ B^{T_{\R,x_0}X}(0,\e)$ and $ B^X(x_0,\e)$ as in Section \ref{localization}. Note that in this identification, the radial vector field $\rad =  \sum_i Z_i e_i$ becomes $\rad = Z$, so $Z$ can be viewed as a point or as a tangent vector.

Recall that $\n^p_1 = \n_1 \otimes 1 + 1\otimes \n^{L^p}$ is the connection on $\Lambda^\bullet(T^*_{\R} B)\otimes\Wedge(T^*X)\otimes L^p \otimes \xi$ corresponding to $\n_u$ in \eqref{defnablau}, replacing $\xi$ by $\xi \otimes L^p$ and taking $u=1$.

For $Z\in B^{T_{\R,x_0}X}(0,\e)$, we identify  $(\Wedge_Z(T^*X)\otimes \xi_Z, h_Z^{\Wedge\otimes \xi})$   with $(\Wedge_{x_0}(T^*X)\otimes \xi_{x_0},h^{E}_{x_0})$ and $(L_Z,h_Z^L)$ with $(L_{x_0},h^L_{x_0})$ by parallel transport  along the geodesic ray $t \in [0,1]\mapsto tZ$ with respect to the connection $\n_1$ and $\n^L$ respectively. We denote by $\Gamma_1 $ and $\Gamma^L$ the corresponding connection forms.

We denote by $\nabla_U$ the ordinary differentiation operator in the direction $U$ on $T_{x_0}X$.

Let $\rho \colon \R \to [0,1]$ be a smooth even function such that
\begin{equation}
\label{def-rho}
\rho(v)=\left \{
\begin{aligned}
&1 \text{ for } |v|<2, \\
& 0 \text{ for } |v|>4.
\end{aligned}
\right.
\end{equation}

On the trivial bundle 
\begin{equation}
\E_{p,x_0}=\Lambda^\bullet(T^*_{\R,b_0} B)\otimes \left(\Wedge(T^*X)\otimes \xi \otimes L^p\right)_{x_0}
\end{equation}
over $T_{x_0}X$, we define the connection
\begin{equation}
\label{nablax0}
\n^{\E_{p,x_0}} = \n + \rho(|Z|/\e)\left( p\Gamma^L +\Gamma_1\right),
\end{equation}
which is a Hermitian connection.
 
Let $g^{T_\R X_0}$ be a Riemannian metric on $X_0:=T_{\R,x_0}X=\R^{2n}$ such that
\begin{equation}
g^{T_\R X_0}=\left\{
\begin{aligned}
&g^{T_\R X} \text{ on } B^{T_{\R,x_0}X}(0,2\e),\\
&g^{T_{\R,x_0}X} \text{ outside of } B^{T_{\R,x_0}X}(0,4\e),
\end{aligned}
\right.
\end{equation}
and let $dv_{X_0}$ be the associated volume form. Let $dv_{TX}$ be the Riemannian volume form of $(T_{x_0}X, g^{T_{x_0}X})$, and $\kappa(Z)$ be the smooth positive function defined by $\kappa(0)=1$ and
\begin{equation}
\label{defkappa}
dv_{X_0} (Z) = \kappa(Z) dv_{TX}(Z).
\end{equation}

Let $\Delta^{\E_{p,x_0}}$ be the Bochner Laplacian associated with $\n^{\E_{p,x_0}}$ and $g^{T_\R X_0}$. By definition, if $\nabla^{T_\R X_0,LC}$ is the Levi-Civita connection on $(X_0,g^{T_\R X_0})$ and if  $(g^{ij}(Z))$ is the inverse of the matrix $(g_{ij}(Z))=(g^{T_\R X_0}_Z(e_i,e_j))$, we have
\begin{equation}
\label{deflaplaciendeBochner}
\Delta^{\E_{p,x_0}}= -g^{ij}(Z)  \left(\n^{\E_{p,x_0}}_ {e_i}\n^{\E_{p,x_0}}_ {e_j} -\n^{\E_{p,x_0}}_ {\nabla ^{T_\R X_0,LC}_{e_i}e_j} \right).
\end{equation}

Recall that $\{f_\alpha\}$ denotes a frame of $T_\R B$, with dual frame $\{f^\alpha\}$.  Let $\widetilde{e_i}(Z)$ be the parallel transport of $e_i$ with respect to $\n^{T_\R X_0,LC}$ along the curve $t\in [0,1]\mapsto tZ$. Then $\{\widetilde{e_i}\}_i$ is an orthonormal frame of $T_\R X_0$. Set
\begin{multline}
\label{defPhi}
\Phi = \frac{K^X}{8} + \frac{1}{4} c(\widetilde{e_i})c(\widetilde{e_j})L'^\xi(\widetilde{e_i},\widetilde{e_j})+\frac{1}{\sqrt{2}}c(\widetilde{e_i})f^\alpha L'^\xi(\widetilde{e_i},f_\alpha)+ \frac{f^\alpha f^\beta}{2}L'^\xi(f_\alpha,f_\beta) \\
-\Big( \db^M\partial^Mi\omega\Big)^c  - \frac{1}{16}\Big \| \left( \db^X-\partial^X\right)i\omega^X \Big\|^2_{\Lambda^\bullet(T^*_\R X)}
\end{multline}
and
\begin{multline}
\label{Mpu}
M_{p,x_0}  =  \frac{1}{2}\Delta^{\E_{p,x_0}} + \rho(|Z|/\e)  \Phi \\
+p\rho(|Z|/\e) \left ( \frac{1}{4} c(\widetilde{e_i})c(\widetilde{e_j})R^L(\widetilde{e_i},\widetilde{e_j})+\frac{1}{\sqrt{2}}c(\widetilde{e_i})f^\alpha R^L(\widetilde{e_i},f_\alpha)+ \frac{f^\alpha f^\beta}{2}R^L(f_\alpha,f_\beta)\right).
\end{multline}
Then $M_{p,x_0}$ is a second order elliptic differential operator acting on $\smooth\left(T_{\R,x_0}X,\E_{p,x_0}\right)$. Moreover, using Theorem \ref{Lichnerowicz}, \eqref{nablax0}, \eqref{deflaplaciendeBochner}, \eqref{defPhi} and \eqref{Mpu}, we see that $M_{p,x_0}$ and $B_{p}^2$ coincide over $B^{TX}(0,2\e)$.

Let $S_L$ be a unit vector of $L_{x_0}$. It gives an isometry $L^p_{x_0} \simeq \C$, which yields to an isometry 
\begin{equation}
\label{def-E-sans-p}
\E_{p,x_0} \simeq \Lambda^\bullet(T^*_{\R,b_0} B)\otimes (\Wedge (T^*X)\otimes \xi)_{x_0} =: \E_{x_0}.
\end{equation}
We endow $\E$ with the connection $\n^{\E}$ induce by $\n^{\Lambda^\bullet(T^*_\R B)}$, $\n^{\Wedge,LC}$ and $\n^\xi$ and with the metric $h^{\E}$ induce by $h^{\Lambda^\bullet(T^*_\R B)}$, $h^{\Wedge,LC}$ and $h^\xi$.

\begin{rem}
\label{independancedeSL}
In this trivialization, $B_p^2$ acts on $\E_{x_0}$, but this action may \emph{a priori} depends on the choice of $S_L$. However, thanks to Theorem \ref{Lichnerowicz} we see that the operator $B_{p}^2$ has it coefficients in $\End(\E_{p,x_0})$ which is canonically isomorphic to $\End(\E)_{x_0}$ (by the natural identification $\End(L^p) \simeq \C$), thus all our formulas do not depend on this choice. Under this identification, we will consider $M_{p,x_0}$ as an operator acting on $\smooth(T_{x_0}X, \E_{x_0})$. 
\end{rem}

Let $\exp(-B_{p}^2)(Z,Z')$ and $\exp(-M_{p,x_0})(Z,Z')$ be the smooth heat kernels of $B_{p}^2$ and $M_{p,x_0}$ with respect to $dv_{X_0}(Z')$.
\begin{lemme}
\label{noyauBupetMup}
For any $m\in \N$, $\e>0$, there exist $C>0$  and $N\in \N$ such that for  any $p\in \N^*$,
\begin{equation}
\left| \exp\Big(-\frac{u}{p}B_{p}^2\Big)(x_0,x_0) -  \exp\Big(-\frac{u}{p}M_{p,x_0}\Big)(0,0)\right|_{\mathscr{C}^m(M)} \leq Cp^N \exp(-\frac{\e^2p}{16u}),
\end{equation}
where $|\cdot|_{\mathscr{C}^m(M)}$ denotes the $\mathscr{C}^m$-norm in the parameters $b_0\in B$ and $x_0\in X$ induced by $\n^{\End(\E)}$ and $h^{\End(\E)}$.
\end{lemme}
\begin{proof}
By \eqref{Mpu}, $M_{p,x_0}$ has the same structure as $B_{p}^2$. Thus Lemma \ref{estimeeelliptique(p)} and Proposition \ref{lepbestlocal} are still true if we replace $B_{p}^2$  therein by $M_{p,x_0}$. From the fact that $M_{p,x_0}$ and $B_{p}^2$ coincide near 0 and the finite propagation speed of the wave equation (see e.g. \cite[Thm. D.2.1]{ma-marinescu}), we know that
\begin{equation}
\tilde{\F}_{\frac{u}{p}} \left( \frac{u}{p}B_{p}^2 \right)(x_0,\cdot) = \tilde{\F}_{\frac{u}{p}} \left( \frac{u}{p}M_{p,x_0} \right)(0,\cdot),
\end{equation}
so we get our Lemma by \eqref{liensFGH}.
\end{proof}

We will now make the change of parameter $t=\frac{1}{\sqrt{p}}\in \,]0,1]$. 

\begin{defn}
For $s\in \smooth(T_{\R,x_0}X, \E_{x_0})$ and $Z\in\R^{2n}$ set
\begin{equation}
\label{defnrescaled}
\begin{aligned}
&(S_ts)(Z)  = s(Z/t), \:\:\qquad \qquad \n_t = tS_t^{-1} \kappa^{1/2}  \nabla^{\E_{p,x_0}}\kappa^{-1/2} S_t, \\
& \n_0= \nabla + \frac{1}{2} R^L_{x_0}( Z, \cdot) , \quad\qquad \LL_{t}=   t^2 S_t^{-1}  \kappa^{1/2}M_{p,x_0}  \kappa^{-1/2} S_t, \\
&\LL_{0} = - \frac{1}{2}\sum_i \left(\n_{0,e_i}\right)^2+ \frac{1}{4} c(e_i)c(e_j)R^L_{i,j}(x_0)+\frac{1}{\sqrt{2}}c(e_i)f^\alpha R^L_{i,\alpha}(x_0)+ \frac{f^\alpha f^\beta}{2}R^L_{\alpha,\beta}(x_0).
\end{aligned}
\end{equation}
\end{defn}	

From now on we will denote $c(e_i)$ by $c^i$ to simplify the notation in the computations.

\begin{prop}
\label{asympL}
When $t\to0$, we have
\begin{equation}
\label{asympLeq}
\n_{t,e_i} = \n_{0,e_i} + O(t) \text{ and} \quad \LL_{t} = \LL_0 +O(t).
\end{equation}
\end{prop}

\begin{proof}
By \eqref{nablax0} and \eqref{defnrescaled}, we have
\begin{equation}
\label{natiup}
\n_{t,e_i}(Z) = \kappa^{1/2}\left(tZ\right)\left\{ \n_{e_i} + \rho(t|Z|/\e)\left( t^{-1}\Gamma^L_{tZ}(e_i) +t\Gamma_{1,tZ}(e_i)\right)\right\}\kappa^{-1/2}\left(tZ\right).
\end{equation}

It is a well known fact (see for instance \cite[Lemma 1.2.4]{ma-marinescu}) that for if $\Gamma= \Gamma^L$ (resp. $\Gamma_1$) and $R=R^L$ (resp. $R_{1}$ the curvature of $\n_1$), then
\begin{equation}
\label{dvlptGamma}
\Gamma_Z(e_i) = \frac{1}{2}R_{x_0}(Z,e_i)+O(|Z|^2).
\end{equation}
Thus,
\begin{equation}
\label{dvlptGammaup}
\begin{aligned}
&t\Gamma_{1,tZ}(e_i) =    O\left( t^2 \right),  \\
&t^{-1}\Gamma^L_{tZ}(e_i) = \frac{1}{2}R^L_{x_0}(Z,e_i) +O\left( t \right).
\end{aligned}
\end{equation}

The first asymptotic development in Proposition \ref{asympL} follows from $\rho(0)=\kappa(0)=1$, \eqref{natiup}, \eqref{dvlptGamma} and \eqref{dvlptGammaup}. Moreover, with this asymptotic,  \eqref{deflaplaciendeBochner} and the fact that $g^{ij}(0)=\delta_{ij}$ we find
\begin{equation}
\label{dvptLB}
\begin{aligned}
t^2S_t^{-1} \kappa^{1/2} \Delta^{\E_{p,x_0}} \kappa^{-1/2} S_t &= -g^{ij}(tZ)  \left(\n_{t,e_i}\n_{t,e_j} -t\n_ {t,\nabla ^{TX_0}_{e_i}e_j} \right) \\
&=\sum_i \left(\n_{0,e_i}\right)^2+O(t) .
\end{aligned}
\end{equation}

On the other hand, by \eqref{Mpu}, we have
\begin{align}
\label{dvlptAup}
t^2S_t^{-1}  \kappa^{1/2}&\Big( M_{p,x_0}-\frac{1}{2}\Delta^{\E_{p,x_0}} \Big)   \kappa^{-1/2}S_t  \notag \\
& =\rho(t|Z|/\e))\left \{ \kappa^{1/2}\left(t^2\Phi+  \frac{1}{4} c(\widetilde{e_i})c(\widetilde{e_j})R^L(\widetilde{e_i},\widetilde{e_j})+\frac{1}{\sqrt{2}}c(\widetilde{e_i})f^\alpha R^L(\widetilde{e_i},f_\alpha)\right.\right.\notag \\
& \qquad \qquad \qquad \qquad \qquad \qquad \qquad\qquad\qquad\qquad\qquad\quad \left.\left.+ \frac{f^\alpha f^\beta}{2}R^L(f_\alpha,f_\beta)\right)\kappa^{-1/2} \right\}_{tZ}   \\
&=\frac{1}{4} c^ic^jR^L_{i,j}(x_0)+\frac{1}{\sqrt{2}}c^if^\alpha R^L_{i,\alpha}(x_0)+ \frac{f^\alpha f^\beta}{2}R^L_{\alpha,\beta}(x_0) + O\left( t \right).\notag
\end{align}

With \eqref{defnrescaled}, \eqref{dvptLB}, \eqref{dvlptAup}, and the first part of \eqref{asympLeq} that we have already proved, the proof of the proposition is completed.
\end{proof}


\subsection{Convergence of the heat kernel}
\label{sectioncvce}

In this section, we use the notations and definitions of Section \ref{rescalingBpu}. In particular,  $b_0 \in B$ and $x_0 \in X_{x_0}$ are fixed. Set 
\begin{equation}
\label{defOmegau}
\Omega_u =  uR^L(w_k,\bw_\ell)\bw^{\ell}\wedge i_{\bw_k} + \sqrt{\frac{u}{2}}c(e_i)f^\alpha R^L_{i,\alpha}+ \frac{f^\alpha f^\beta}{2}R^L_{\alpha,\beta}.
\end{equation}

The purpose of this section is to prove the following result:

\begin{thm}
\label{cvcenoyaudelachaleur}
Let $k\in \N$. Then there is  $\epsilon>0$ such that  as $p\to +\infty$, uniformly as $u$ varies in a compact subset of $\R^*_+$, we have the following asymptotic for the $\mathscr{C}^k$-norm on $\smooth(M,\End(\E))$:
\begin{multline}
\psi_{1/\sqrt{p}}\exp(-B_{p,u/p}^2) (x_0,x_0)= \\
(2\pi)^{-n} \exp(-\Omega_{u,x_0})\frac{\det(\dot{R}^{X,L}_{x_0})}{\det \big(1-\exp(-u\dot{R}^{X,L}_{x_0})\big)}\otimes \Id_{\xi} p^n + O(p^{n-\epsilon}).
\end{multline}
\end{thm}

To prove this theorem, we will adapt the method of \cite[Sect. 1.6]{ma-marinescu}.

\begin{rem}
In fact, this theorem holds without any positivity assumption on $L$. In this case, we have to take the convention that if an eigenvalue of $\dot{R}^{X,L}_{x_0}$ is zero, then its contribution to $\frac{\det(\dot{R}^{X,L}_{x_0})}{\det \big(1-\exp(-u\dot{R}^{X,L}_{x_0})\big)}$ is $\frac{1}{2u}$ and we have to use \cite[(E.2.5)]{ma-marinescu} in addition to \cite[(E.2.4)]{ma-marinescu} to get \eqref{noyaudelachaleurLinfini}.
\end{rem}

\begin{rem}
\label{rem-o(p^n)=>O(p^n-1)}
 As pointed out in \cite[Thm. 4.2.3 and Rem. 4.2.4]{ma-marinescu}, we can use the results of this section combined with the techniques of \cite[Sect. 4.1]{ma-marinescu} to get $O(p^{n-1})$ instead of $O(p^{n-\epsilon})$ in Theorem \ref{cvcenoyaudelachaleur}. However, we do not need this improvement and leave it to the reader.
\end{rem}

The following Lemma is an easy consequence of the Arzel\`{a}-Ascoli theorem, which we will use several time.
\begin{lemme}
\label{Ascoli}
Let $Y$ be a compact manifold and let $(E,h^E,\n^E)$ be a Hermitian bundle with connection over $Y$. We can then define, for $k\in \N$, the $\mathscr{C}^k$-norm $|\cdot|_{\mathscr{C}^k}$ on $\smooth (Y,E)$. Let $f_n \in \smooth (Y,E)$ be a sequence converging weakly to some distribution $f$. If for any $k\in \N$ there is $C_k>0$ such that $\sup_n |f_n|_{\mathscr{C}^k} \leq C_k$, then $f$ is smooth and $f_n$ converges in the $\smooth$ topology to $f$.
\end{lemme}

In the sequel, when we add a superscript $(0)$ to the objects introduced above, we mean their part of degree 0 in $\Lambda^\bullet(T^*_{\R,b_0} B)$.

Let $||\cdot ||_{0}$ be the $L^2$ norm on $\smooth(T_{\R,x_0}X,\E_{x_0})$ induced by $h^{\Lambda^\bullet(T^*_\R B)}_{x_0}$, $h^{\Wedge}_{x_0}$, $h^{\xi}_{x_0}$ and the volume form $dv_{TX}(Z)$. For $s\in\smooth(X_0,\E_{x_0})$, $m\in \N^*$, and $t\geq 0$, set
\begin{equation}
\label{defnorme(p,m)}
\begin{aligned}
&||s||_{t,0}^2=||s||_{0}^2, \\
&||s||_{t,m}^2=\sum_{\ell\leq m} \sum_{i_1,\dots,i_\ell}  ||\n^{(0)}_{t,e_{i_1}}\cdots\n^{(0)}_{t,e_{i_\ell}}s||_{0}^2.
\end{aligned}
\end{equation}
We denote by $\sob^m_t$ the Sobolov space $\sob^m(X_0,\E_{x_0})$ endowed with the norm $||\cdot ||_{t,m}$, and by $\sob^{-1}_t$ the Sobolev space of order $-1$ endowed with the norm
\begin{equation}
||s||_{t,-1}=\sup_{s'\in\sob^1_p\setminus\{0\}} \frac{\langle s,s'\rangle_{t,0}}{||s'||_{t,0}}\, .
\end{equation}

Finally, if $A\in \LL(\sob^k_t, \sob^m_t)$, we denote by $||A||^{k,m}_t$ the operator norm of $A$ associated with $||\cdot||_{t,k}$ and $||\cdot||_{t,m}$.

Let
\begin{equation}
\mathscr{R}_{t} = \LL_{t} -\LL_{t}^{(0)}. 
\end{equation}

\begin{prop}
\label{estimationLupprop}
There exist $C_1,C_2,C_3>0$ such that for any $t>0$ and any $s,s'\in\smooth(X_0,\E_{x_0})$,
\begin{equation}
\label{estimationsLup}
\begin{aligned}
&\langle \LL^{(0)}_{t} s, sÊ\rangle_{t,0} \geq C_1||s||_{t,1}^2 - C_2||s||_{t,0}^2, \\
& \left|\langle \LL^{(0)}_{t} s, s'Ê\rangle_{t,0}\right| \leq C_3 ||s||_{t,1} ||s'||_{t,1}, \\
& \left \| \mathscr{R}_{t} s\right \|_{t,0} \leq C_4 ||s||_{t,1}.
\end{aligned}
\end{equation}
\end{prop}

\begin{proof}
By \eqref{B^0=D}, the operators $\n_t^{(0)}, \LL_{t}^{(0)}$ are the operators corresponding to $\n_t,\LL_{t}$ in the case where $B$ is a point, thus the first two lines of \eqref{estimationsLup} are proved in \cite[Thm. 1.6.7]{ma-marinescu}. 

By \eqref{defnablau}, \eqref{natiup} and \eqref{dvlptGammaup}, we have
\begin{equation}
\label{diffnatietnati(0)}
\n_{t,e_i}-\n_{t,e_i}^{(0)} = O_0\left( t^2\right),
\end{equation}
where by $O_0(t^\alpha)$ we mean an operator of order 0 which is a $O(t^\alpha)$. Thus, by \eqref{dvptLB}, \eqref{dvlptAup} and \eqref{diffnatietnati(0)}, we have
\begin{equation}
\label{dvlpmt-Rt}
\mathscr{R}_{t} =  \n_{t,e_i} O_0\left( t\right) + O_0(1).
\end{equation}
This immediately yields to the last estimate of \eqref{estimationsLup}.
\end{proof}

Let $\Gamma$ be the contour in $\C$ defined in Figure \ref{contour-Gamma}.

\begin{figure}[!h]
\hfill
\setlength{\unitlength}{4cm}
\begin{picture}(2.5,1.125)
\put(0.25,0.5){\vector(1,0){1.85}}
\put(0.70,0.025){\vector(0,1){0.95}}
\thicklines \put(0.50,0.25){\line(0,1){0.5}}
\thicklines \put(0.50,0.25){\line(1,0){1,55}}
\thicklines \put(0.50,0,75){\line(1,0){1,55}}
\put(0.20,0.55){$-2C_2$}
\put(0.50,0.5){\circle*{0.029}}
\put(1.25,0.25){${}_\blacktriangleright$}
\put(1.25,0.75){${}_\blacktriangleleft$}
\put(1.25,0.85){$\Gamma$}
\put(0.72,0.78){$i$}
\put(0.72,0.15){$-i$}
\put(0.63,0.4){0}
\end{picture}
\hfill \hfill
\caption{}
\label{contour-Gamma}
\end{figure}

\begin{prop}
\label{estimationsresoprop}
There exist $C>0$, $a,b \in \N$ such that for any $t>0$ and any $\lambda \in \Gamma$, the resolvant $\left( \lambda-\LL_{t}\right)^{-1}$ exists and
\begin{equation}
\label{estimationsreso}
\begin{aligned}
&\left\| \left( \lambda-\LL_{t}\right)^{-1} \right\|^{0,0}_t \leq C(1+|\lambda|^2)^a, \\
&\left\| \left( \lambda-\LL_{t}\right)^{-1} \right\|^{-1,1}_t \leq C(1+|\lambda|^2)^b.
\end{aligned}
\end{equation}
\end{prop}

\begin{proof}
The fact that \eqref{estimationsreso} holds for $\LL_t^{(0)}$ is proved in \cite[Thm. 1.6.8]{ma-marinescu} as a consequence of the first two lines of \eqref{estimationsLup}. For $\lambda \in \Gamma$, we have
\begin{equation}
\label{resolvantesLp2etLp(0)2}
(\lambda - \LL_{t})^{-1} = (\lambda - \LL_{t}^{(0)})^{-1} + (\lambda - \LL_{t}^{(0)})^{-1}\mathscr{R}_{t}(\lambda - \LL_{t}^{(0)})^{-1} + \dots \quad \text{(finite sum)}.
\end{equation}
Moreover, by the third estimate of \eqref{estimationsLup}, we know that $\| \mathscr{R}_{t}\|_t^{1,-1} \leq C_4$. Thus, \eqref{estimationsreso} follows from \eqref{estimationsreso} for $\LL_t^{(0)}$ and \eqref{resolvantesLp2etLp(0)2}.
\end{proof}

\begin{prop}
\label{estimationcommutateursprop}
Take $m\in \N^*$. Then there exists a contant $C_m>0$ such that for any $t>0$,  $Q_1,\dots,Q_m \in \left\{  \n^{(0)}_{t,e_i},Z_i\right\}_{i=1}^{2n}$ and  $s,s'\in \smooth_c(T_{\R,x_0}X, \E_{x_0})$,
\begin{equation}
\label{estimationcommutateurs}
\left| \Big\langle [Q_1,[Q_2,\dots[Q_m,\LL_{t}]\dots]]s,s'\Big\rangle_{t,0} \right| \leq C_m||s||_{t,1}||s'||_{t,1}.
\end{equation}
\end{prop}

\begin{proof}
First, note that $[\n^{(0)}_{t,e_i},Z_j]=\delta_{ij}$. Thus by \eqref{dvptLB} and \eqref{dvlptAup},  $[Z_j,\LL_{t}]$ satisfies \eqref{estimationcommutateurs}.

Let $R_{1,\rho}$ and $R_{\rho}^L$ be the curvatures of the connections $\n + \rho(|Z|/\e)\Gamma_1$ and $\n + \rho(|Z|/\e)\Gamma^L$. Then by \eqref{nablax0} and \eqref{defnrescaled}, we have
\begin{equation}
\label{comavecnati}
\big[\n^{(0)}_{t,e_i},\n^{(0)}_{t,e_j}\big] = \big( R_{\rho}^L+ t^2R_{1,\rho} \big)^{(0)}_{tZ}(e_i,e_j).
\end{equation}

By \eqref{dvptLB}, \eqref{dvlptAup} and \eqref{comavecnati}, we find that $\big[\n^{(0)}_{t,e_i},\LL_{t}\big]$ has the same structure as $\LL_{t}$ for $t\in \, ]0,1]$, by which we mean that it is of the form
\begin{equation}
\label{structure}
\sum_{i,j} a_{ij}(t,tZ) \n^{(0)}_{t,e_i}\n^{(0)}_{t,e_i} + \sum_{i}  b_{i}(t,tZ) \n^{(0)}_{t,e_i} + c(t,tZ),
\end{equation}
where $a_{ij}$, $b_i$, $c$ are polynomials in the first variable, and have all their derivatives in the second variable uniformly bounded for $Z\in T_{\R,x_0}X$ and $t\in [0,1]$.

The adjoint connection $(\n^{(0)}_t)^*$ of $\n^{(0)}_t$ with respect to $\langle \cdot \, , \cdot \rangle_{t,0}$ is given by
\begin{equation}
\label{adjointdenati}
(\n^{(0)}_t)^* = - \n^{(0)}_t -t \big(\kappa^{-1} \n \kappa\big)(tZ).
\end{equation}
Note that the last term of \eqref{adjointdenati} and all its derivative in $Z$ are uniformly bounded for $Z\in T_{\R,x_0}X$ and $t\in [0,1]$. Thus, by \eqref{structure} and \eqref{adjointdenati}, we find that \eqref{estimationcommutateurs} holds when $m=1$.

Finally, we can prove by induction that $[Q_1,[Q_2,\dots[Q_m,\LL_{t}]\dots]]$ has also the same structure as in \eqref{structure}, and thus satisfies \eqref{estimationcommutateurs} thanks to \eqref{adjointdenati}.
\end{proof}

\begin{prop}
\label{norme(m,m+1)resolvante}
For any $t>0$, $\lambda \in \Gamma$ and $m\in \N$,
\begin{equation}
(\lambda-\LL_{t})^{-1}\big( \sob^m_t\big) \subset \sob^{m+1}_t.
\end{equation}

Moreover, for any $\alpha \in \N^{2n}$, there exist $K\in \N$ and $C_{\alpha,m}>0$ such that for any $t\in \, ]0,1]$, $\lambda \in \Gamma$ and $s\in \smooth_c(T_{\R,x_0}X,\E_{x_0})$,
\begin{equation}
\left\| Z^\alpha (\lambda-\LL_{t})^{-1}s\right\|_{t,m+1} \leq C_{\alpha,m}(1+|\lambda|^2)^K\sum_{\alpha'\leq \alpha} ||Z^{\alpha'}s||_{t,m}.
\end{equation}
\end{prop}

\begin{proof}
Proposition \ref{norme(m,m+1)resolvante}  follows from Propositions \ref{estimationsresoprop} and  \ref{estimationcommutateursprop} exactly as \cite[Thm. 1.6.10]{ma-marinescu} follows from \cite[Thm. 1.6.8 and Prop. 1.6.9]{ma-marinescu}, the horizontal part of $\LL_t$ making no difference.
\end{proof}

Let $e^{-\LL_{t}}(Z,Z')$ be the smooth kernel of the operator $e^{-\LL_{t}}$ with respect to $dv_{TX}(Z')$. Let $\mathrm{pr}_M \colon T_\R X\times_M T_\R X\to M$ be the projection from the fiberwise product $T_\R X\times_M T_\R X$ onto $M$, then $e^{-\LL_{t}}(\cdot,\cdot)$ is a section of $\mathrm{pr}_M^*\left(\End\left( \E \right)\right)$ over $T_\R X\times_M T_\R X$. Recall that $\n^{\End(\E)}$ and $h^{\End(\E)}$ have been defined below \eqref{def-E-sans-p}, and let $\nabla^{\mathrm{pr}_M^*\End(\E)}$ (resp. $h^{\mathrm{pr}_M^*\End(\E)}$) be the induced connection (resp. metric) on $\mathrm{pr}_M^*\End(\E)$. 
\begin{thm}
\label{estimationderiveesnoyaudeLpu}
Let $u>0$ be fixed. For any $m,m'\in \N$, there is $C>0$ such that for any $t>0$, $Z,Z'\in T_{x_0}X$ with $|Z|,|Z'|\leq 1$,
\begin{equation}
\label{estimationderiveesnoyaudeLpu:eq}
\sup_{|\alpha|,|\alpha'|\leq m} \left| \frac{\partial^{|\alpha|+|\alpha'|}}{\partial {Z}^{\alpha}\partial Z'^{\alpha'}} e^{-u\LL_{t}}(Z,Z')\right|_{\mathscr{C}^{m'}(M)} \leq C,
\end{equation}
where $|\cdot|_{\mathscr{C}^{m'}(M)}$ denotes the $\mathscr{C}^{m'}$ norm with respect to the parameters $b_0$ and $x_0\in X_{b_0}$ induced by $\nabla^{\mathrm{pr}_M^*\End(\E)}$ and $h^{\mathrm{pr}_M^*\End(\E)}$.
\end{thm}

\begin{proof}
By \eqref{estimationsreso}, we know that for $k\in \N^*$,
\begin{equation}
\label{e(-Lup)=integralledecontour}
e^{-u\LL_{t}} = \frac{(-1)^{k-1}(k-1)!}{2i\pi u^{k-1}} \int_\Gamma e^{-u\lambda}(\lambda-\LL_{t})^{-k}d\lambda.
\end{equation}
Thus, Theorem \ref{estimationderiveesnoyaudeLpu} can be proved from Proposition \ref{norme(m,m+1)resolvante} exactly as \cite[Thm. 1.6.11]{ma-marinescu} is proved from \cite[Thm. 1.6.10]{ma-marinescu}.
\end{proof}

\begin{thm}
\label{diffdesresolvantes-thm}
There are constants $C>0$ and $M\in \N^*$ such that for $t>0$, 
\begin{equation}
\label{diffdesresolvantes}
\left\| \big((\lambda-\LL_{t})^{-1} - (\lambda-\LL_{0}
)^{-1}\big)s \right\|_{0,0} \leq Ct (1+|\lambda|^2)^M\sum_{|\alpha|\leq3}||Z^\alpha s||_{0,0}.
\end{equation}
\end{thm}

\begin{proof}
This is proved from \eqref{natiup}, \eqref{dvptLB}, \eqref{dvlptAup} and \eqref{defnorme(p,m)} using a Taylor expansion as done in \cite[Thm. 1.6.12]{ma-marinescu}.
\end{proof}

\begin{thm}
\label{cvcenoyauLup->noyauLuinfini}
For $u>0$ fixed, there exists $C>0$ such that for $t>0$ and $Z,Z'\in T_{x_0}X$ with $|Z|,|Z'|\leq 1$,
\begin{equation}
\label{cvcenoyauLup->noyauLuinfini:eq}
\left| \left(e^{-u\LL_{t}}-e^{-u\LL_{0}
}\right)(Z,Z') \right| \leq Ct^{1/(2n+1)}.
\end{equation}
\end{thm}

\begin{proof}
This theorem follows from Theorems \ref{estimationderiveesnoyaudeLpu} and \ref{diffdesresolvantes-thm} exactly as \cite[Thm. 1.6.13]{ma-marinescu} follows from \cite[Thms. 1.6.11, 1.6.12]{ma-marinescu}.
\end{proof}

We can now prove Theorem \ref{cvcenoyaudelachaleur}.

By \eqref{defkappa} and \eqref{defnrescaled}, we have 
\begin{equation}
\label{noyauxdeMupetdeLup}
e^{-u\LL_{t}}(Z,Z')=p^{-n} e^{-\frac{u}{p}M_{p,x_0}}(tZ,tZ')\kappa^{1/2}(tZ)\kappa^{-1/2}(tZ').
\end{equation}

Define 
\begin{equation}
\label{defLinfiniu}
\LL_{0,u} = u\psi_{1/\sqrt{u}} \LL_{0}\psi_{\sqrt{u}}.
\end{equation}
Then by the last line of \eqref{defobjetsavecp}, Lemmas \ref{noyauBupetMup} and \ref{Ascoli}, Theorem \ref{estimationderiveesnoyaudeLpu} and \ref{cvcenoyauLup->noyauLuinfini} and \eqref{noyauxdeMupetdeLup} we get that  for every fixed $u>0$ and for the $\mathscr{C}^k$-norm on $\smooth(M,\End(\E))$,
\begin{equation}
\label{BupetLinfini}
p^{-n}\psi_{1/\sqrt{p}} e^{-B_{p,u/p}^2}(x_0,x_0)=p^{-n} \psi_{1/\sqrt{u}}e^{-\frac{u}{p}M_{p,x_0}}(0,0)=e^{-\LL_{0,u}
}(0,0) + O(p^{-\epsilon}),
\end{equation}
with $\epsilon = \frac{1}{4n+2}$.

Finally, using the fact that
\begin{equation}
\label{relationaveccicj}
\frac{1}{4}\sum_{ij}c(e_i)c(e_j)R^L(e_i,e_j) = \sum_{l,m}R^L(w_l,\bw_m)\bw^m\wedge i_{\bw_l}- \frac{1}{2}\sum_jR^L(w_j,\bw_j)
\end{equation}
and \eqref{defRpoint}, \eqref{defnrescaled}, \eqref{defOmegau} and \eqref{defLinfiniu}  we find
\begin{align}
\LL_{0,u}
=&-\frac{u}{2}\sum_i \left(\n+\frac{1}{2}R^L_{x_0}(Z,e_i)\right)^2+u\left( \sum_{l,m}R_{x_0}^L(w_l,\bw_m)\bw^m\wedge i_{\bw_l}- \frac{1}{2}\sum_jR_{x_0}^L(w_j,\bw_j) \right) \notag \\
&+\sqrt{\frac{u}{2}}c(e_i)f^\alpha R^L_{i,\alpha}(x_0)+ \frac{f^\alpha f^\beta}{2}R^L_{\alpha,\beta}(x_0)  \\
=&-\frac{u}{2}\sum_i \left(\n+\frac{1}{2}\langle{\dot{R}}^{X,L}_{x_0}Z,e_i\rangle\right)^2 +\Omega_u(x_0)-\frac{u}{2}\tr({\dot{R}}^{X,L}_{x_0}). \notag
\end{align} 
The formula for the heat kernel of a harmonic oscillator (see \cite[(E.2.4)]{ma-marinescu} for instance) gives
\begin{equation}
\label{noyaudelachaleurLinfini}
e^{-\LL_{0,u}
}(0,0) = (2\pi)^{-n} \exp(-\Omega_{u,x_0})\frac{\det(\dot{R}^{X,L}_{x_0})}{\det \big(1-\exp(-u\dot{R}^{X,L}_{x_0})\big)}\otimes \Id_{\xi},
\end{equation}
which implies Theorem \ref{cvcenoyaudelachaleur} by \eqref{BupetLinfini}.


\subsection{Asymptotic of the torsion forms}
\label{Asymptoticofthetorsionforms}

Let $b_0 \in B$ be fixed. Again we denote $X_{b_0}$ by $X$. Recall that $\omega^{H}$ and $N_u$ are defined respectively in \eqref{decompo-omega} and \eqref{defNueq}. let $d=\dim M$.

For $x\in X$, set
\begin{equation}
\label{defRu}
\begin{aligned}
&\Lambda_u(x) = (2\pi)^{-n} \exp(-\Omega_{u,x})\frac{\det(\dot{R}^{X,L}_{x})}{\det \big(\Id-\exp(-u\dot{R}^{X,L}_{x})\big)},\\ 
&\RR_u(x)=  \tr_s\left[ N_u\Lambda_u(x)\right].
\end{aligned}
\end{equation}

Let $A_j\in \smooth\big(X, \End(\Lambda^\bullet(T^*_{\R,b_0} B)\otimes\Wedge(T^*X))\big)$ be such that $A_{-d-1} = 0$ and as $u\to 0$
\begin{equation}
\label{defAj}
\Lambda_u (x) = \sum_{j=-d}^k A_j(x)u^j + O(u^{k+1}).
\end{equation}

\begin{thm}
\label{dvlptuniformedunoyaudelachaleur}
There exist $A_{p,j}\in \smooth\big(X,\Lambda(T_{\R}^*B)\otimes\End(\Wedge(T^*X)\otimes \xi)\big)$ such that for any $k,\ell\in \N$, there exist $C>0$ such that for any $u\in ]0,1]$ and $p\geq 1$, 
\begin{equation}
\label{dvlptuniformedunoyaudelachaleureq1}
\left|p^{-n}\psi_{1/\sqrt{p}} \exp\left( -B_{p,u/p}^2 \right)(x,x) - \sum_{j=-d}^k A_{p,j}(x)u^j\right|_{\mathscr{C}^\ell(M)}\leq Cu^{k+1}.
\end{equation}
Here,  $\mathscr{C}^\ell(M)$ denotes the $\mathscr{C}^\ell$-norm in the parameter $(b,x)\in M$.

Moreover, as $p\to+\infty$, we have for any $j\geq-d$
\begin{equation}
\label{dvlptuniformedunoyaudelachaleureq2}
A_{p,j}(x)=A_j(x)\otimes \Id_\xi + O\left( \frac{1}{\sqrt{p}}\right),
\end{equation}
where the convergence is in the $\smooth$ topology on $M$.
\end{thm}

\begin{proof}
 Theorem \ref{dvlptuniformedunoyaudelachaleur} is proved using the same techniques as \cite[Thm. 5.5.9]{ma-marinescu}. Let us give the mains ideas of the proof, in which it is clear that the part in positive degree of $\LL_t$ has no incidence.
 
 First, we localize the problem near $x_0\in X$ with the same method as in Section \ref{localization}, in particular Proposition~\ref{lepbestlocal}. Then we rescale the superconnection as in Section \ref{rescalingBpu} to get an operator, denoted here by $\LL_{t,x_0}$ to make the dependance in $x_0$ clearer.
 
 By the finite propagation speed of the wave operator \cite[Thm D.2.1]{ma-marinescu}, for $t$ small, $\tilde{\F}_u(u\LL_{t,x_0}(0,\cdot))$ only depend on the restriction of $\LL_{t,x_0}$ on $B^{T_{\R,x_0}X}(0,2\e)$ and is supported in $B^{T_{\R,x_0}X}(0,2\e)$.

Now consider a sphere bundle $V=\{(z,c)\in T_{\R}X\times \R\, : \, |z|^2+c^2=1\}$ over $X$. We embed $B^{T_{\R,x_0}X}(0,2\e)$ in $V_{x_0}$ by sending $z$ to $(z,\sqrt{1-|z|^2})$ and we extend $\LL_{t,x_0}$ to a generalized Laplacian $\widetilde{\LL}_{t,x_0}$ on $V_{x_0}$ with values in $\mathrm{pr}_M^*\left(\End\left( \E \right)\right)$. Then, similarly as Lemma \ref{noyauBupetMup}, we have for $0<u\leq 1$
\begin{equation}
\label{aveclelaplaciensurlefibrecompact}
\left| e^{- u\LL_{t,x_0} }(0,0)- e^{ -u\widetilde{\LL}_{t,x_0} }(0,0) \right|_{\mathscr{C}^m(M\times [0,1])} \leq C\exp \left( -\frac{\e^2}{32u} \right).
\end{equation}

Finally, as the total space of $V$ is compact, the heat kernel $\exp\left(- u\widetilde{\LL}_{t,x_0} \right)(0,0)$ has an asymptotic expansion (starting with $u^{-n}$) when $u\to 0$ which depends smoothly on the parameters $x_0$ and $t$ (see for instance \cite[(D.1.24)]{ma-marinescu}). Thus, thanks to Lemma \ref{noyauBupetMup}, \eqref{BupetLinfini} and \eqref{aveclelaplaciensurlefibrecompact} we find \eqref{dvlptuniformedunoyaudelachaleureq1} and  $A_{p,j}=A_{\infty,j}+O(1/\sqrt{p})$. Moreover, we get $A_{\infty,j}=A_j\otimes \Id_\xi$ from \eqref{noyaudelachaleurLinfini}.
\end{proof}

For $j \geq -d-1$, set
\begin{equation}
\label{defAtildej}
\tilde{A}_j(x) =  \tr_s\left[ N_VA_j(x)+i\omega^{H}A_{j+1}(x)\right].
\end{equation}
Then by \eqref{defNueq}, \eqref{defRu} and \eqref{defAj}, we have
\begin{equation}
\label{Atildej}
\RR_u (x) = \sum_{j=-d-1}^k \tilde{A}_j(x)u^j + O(u^{k+1}).
\end{equation}
 Set also
 \begin{equation}
 \label{defBpj/Bj}
 \begin{aligned}
 &B_{p,j}=\int_X\tr_s\left[N_VA_{p,j}(x)+i\omega^{H}A_{p,j+1}(x)\right]dv_X(x), \\
 &B_j=\int_X\tilde{A}_j(x)dv_X(x).
 \end{aligned}
 \end{equation}
 
 \begin{cor}
 \label{cordvlptuniforme}
 For any $k,\ell\in \N$, there exists $C>0$ such that for any $u\in ]0,1]$ and $p\geq 1$, 
\begin{equation}
\left|p^{-n}\psi_{1/\sqrt{p}}\tr_s \left[ N_{u/p}\exp\left( -B_{p,u/p}^2 \right)\right] - \sum_{j=-d-1}^k B_{p,j}u^j\right|_{\mathscr{C}^\ell(B)} \leq Cu^{k+1}.
\end{equation}
Moreover, as $p\to+\infty$, we have for any $j\geq-d-1$
\begin{equation}
B_{p,j}=\mathrm{rk}(\xi)B_j + O\left( \frac{1}{\sqrt{p}}\right),
\end{equation}
where the convergence is in the $\smooth$ topology on $B$.
 \end{cor}

\begin{proof}
This is a consequence of Theorem \ref{dvlptuniformedunoyaudelachaleur}, using \eqref{defAtildej}-\eqref{defBpj/Bj} and  $\psi_{1/\sqrt{p}}N_{u/p} = N_u$.
\end{proof}

\begin{thm}
\label{dominationugrand}
There exists $C>0$ such that for $u\geq 1$ and $p\geq 1$,
\begin{equation}
\label{dominationeq}
\left|p^{-n} \psi_{1/\sqrt{p}} \tr_s \left[ N_{u/p} \exp\left( -B_{p,u/p}^2 \right) \right] \right|_{\mathscr{C}^\ell(B)}\leq \frac{C}{\sqrt{u}}.
\end{equation}
\end{thm}

Theorem \ref{dominationugrand} will be proved in Section \ref{demothm2}.

Recall that we assumed in the introduction that there is a $p_0\in \N$ such that the direct image $R^i\pi_*(\xi\otimes L^p)$ is locally free for all $p\geq p_0$ and $i\in \{1,\dots,n\}$, and vanishes for $i>0$. In particular, for $p\geq p_0$,
\begin{equation}
H^i\left(X,(\xi\otimes L^p)|_{X}\right)=0 \quad \text{for} \quad i>0.
\end{equation}

For $p\geq p_0$, set
\begin{equation}
\tilde{\zeta}_p(s) = -\frac{p^{-n}}{\Gamma(s)}\int_0^{+\infty} u^{s-1}\psi_{1/\sqrt{p}}\Phi\left\{ \trs \left[N_{u/p} \exp(-B_{p,u/p}^2) \right]   \right\}du.
\end{equation}
Here we make an abuse of notation: we should split the integral in two part as in \eqref{defzeta}. Clearly, if $\zeta_p$ denotes the zeta function \eqref{defzeta} associated with $B_{p,u}$, we have
\begin{equation}
p^{-n}\psi_{1/\sqrt{p}}\zeta_p(s) = p^{-s}\tilde{\zeta}_p(s).
\end{equation}
We deduce that
\begin{equation}
\label{zetaetzetatilde}
p^{-n}\psi_{1/\sqrt{p}}\zeta'_p(0) =  \log(p)B_{p,0}+\tilde{\zeta}'_p(0).
\end{equation}

On the other hand, we have for $p\geq p_0$,
\begin{equation}
\label{expressionzetatilde'}
\begin{aligned}
\tilde{\zeta}'_p(0)=& - \int^1_0p^{-n}\Phi \left\{ \psi_{1/\sqrt{p}} \trs \left[N_{u/p} \exp(-B_{p,u/p}^2) \right] -\sum_{j=-d-1}^0B_{p,j}u^j \right\} \frac{du}{u} \\
&-\int^{+\infty}_1 p^{-n}\Phi\psi_{1/\sqrt{p}} \trs \left[N_{u/p} \exp(-B_{p,u/p}^2) \right]   \frac{du}{u} -\sum_{j=-d-1}^{-1}\frac{B_{p,j}}{j} +\Gamma'(1)B_{p,0}.\\
\end{aligned}
\end{equation}

Let $\tilde{\zeta}(s)$ be the Mellin transform of $u\mapsto - \int_X \RR_u(x)dv_X(x)$, i.e., for $\Re(s)>n$:
\begin{equation}
\tilde{\zeta}(s)=-\frac{1}{\Gamma(s)}\int_0^{+\infty}\int_X \RR_u(x)dv_X(x)u^{s-1}du.
\end{equation}
Then $\tilde{\zeta}$ has a holomorphic extension near 0.

By Theorem \ref{dvlptuniformedunoyaudelachaleur} and Theorem \ref{dominationugrand}, we can apply the dominated convergence theorem to \eqref{expressionzetatilde'}, and with Theorem \ref{cvcenoyaudelachaleur} we find 

\begin{equation}
\label{zetatildeptozetatilde}
\tilde{\zeta}'_p(0) \limarrow{p}{+\infty} \mathrm{rk}(\xi)\Phi\tilde{\zeta}'(0).
\end{equation}

\begin{thm}
\label{expressionzetatilde'(0)thm}
Let $T^{H'}M$ be the orthonormal complement of $TX$ with respect to $R^L$ and let $R^{L,H'}=R^L|_{T_\R^{H'} M\times T_\R^{H'} M}$. Then
\begin{equation}
\label{expressionzetatilde'(0)}
\tilde{\zeta}'(0) =\frac{1}{2} \int_X \det\left( \frac{\dot{R}^{X,L}}{2\pi}\right) \log \left[ \det\left( \frac{\dot{R}^{X,L}}{2\pi}\right) \right] e^{-R^{L,H'}}dv_X.
\end{equation}
\end{thm}

\begin{proof}
This Theorem is the analogue of \cite[(53)]{bismut-vasserot-89} (see also \cite[(5.5.60)]{ma-marinescu}) in the family setting. The main new feature here is the presence in the exponential of terms $c(e_i)f^\alpha$ coupling horizontal and vertical variables. This terms make the computations of the super-traces much more complicated. To deal with them, we draw our inspiration form \cite{MR836726}.

We first compute
\begin{equation}
\label{RuavectrsOmegau}
\RR_u = (2\pi)^{-n} \trs \left( N_ue^{-\Omega_{u}}\right)\frac{\det(\dot{R}^{X,L})}{\det \big(\Id-\exp(-u\dot{R}^{X,L})\big)}.
\end{equation}

Let
\begin{equation}
\widetilde{\Omega}_u = \frac{u}{4}c(e_i)c(e_j)R^L_{ij}+\sqrt{\frac{u}{2}}c(e_i)f^\alpha R^L_{i\alpha}.
\end{equation}
Then by \eqref{defOmegau} and \eqref{relationaveccicj}, we have
\begin{equation}
\label{trsOmegauetOmegatildeu}
\begin{aligned}
&\Omega_u=\widetilde{\Omega}_u+ \frac{f^\alpha f^\beta}{2}R^L_{\alpha\beta}+\frac{u}{2}\tr(\dot{R}^{X,L}),\\
&\trs(N_ue^{-\Omega_u})=\trs(N_ue^{-\widetilde{\Omega}_u}) e^{-\frac{f^\alpha f^\beta}{2}R^L_{\alpha\beta}-\frac{u}{2}\tr(\dot{R}^{X,L})}.
\end{aligned}
\end{equation}

As $c(e_i)c(e_j)\omega_{ij} = 2\ic (\bw^j\wedge i_{\bw_j}-i_{\bw_j}\bw^j)$  , we have (see \cite[(2.15)]{MR929147})
\begin{equation}
\label{formuleNV}
N_V=\frac{n}{2}-\frac{\ic}{4}c(e_i)c(e_j) \omega_{ij}
\end{equation}

Recall that $\omega^X$ is defined in \eqref{defomegaX}.  Set
\begin{equation}
\label{defomegaub}
\begin{aligned}
&R_u(b)=-\frac{1}{2}uR^L-\frac{\ic b}{2}\omega^X, \\
&\omega_u(b) = - \widetilde{\Omega}_u -\frac{ib}{2}\omega^X=\frac{1}{2}c(e_i)c(e_j)R_u(b)_{ij}-\sqrt{\frac{u}{2}}c(e_i)f^\alpha R^L_{i\alpha}.
\end{aligned}
\end{equation}
Then by \eqref{defNueq}, \eqref{formuleNV} and \eqref{defomegaub} we have 
\begin{equation}
\label{trsOmegatildeu}
\trs(N_ue^{-\widetilde{\Omega}_u})=\left( \frac{n}{2} + \frac{\ic \omega^{H}}{u}\right)\trs(e^{\omega_u(0)}) + \left.\derpar{}{b}\right|_{b=0}\trs(e^{\omega_u(b)}).
\end{equation}

Note that the matrix $\big( R_u(b)_{ij} \big)_{ij}$ is invertible for $b$ small enough. We denote the coefficients of its inverse by $R_u(b)^{ij}$. Let
\begin{equation}
\label{defVi}
\begin{aligned}
&V_i=\sum_\alpha f^\alpha R^L_{i\alpha} \, , \qquad V_{u,i}=\sqrt{\frac{u}{2}}V_i, \\
&\widetilde{V}_{u,i} = \sum_k R_u(b)^{ik}V_{u,k}.
\end{aligned}
\end{equation}
A computation shows that
\begin{equation}
\begin{aligned}
\omega_u(b) &= \frac{1}{2}c(e_i)c(e_j)R_u(b)_{ij} + V_{u,i}c(e_i) \\
&= \frac{1}{2}\sum_{ij} (c(e_i)-\widetilde{V}_{u,i})R_u(b)_{ij}(c(e_j)-\widetilde{V}_{u,j})+ \frac{1}{2}\sum_{ij}V_{u,i}V_{u,j}R_u(b)^{ij}.
\end{aligned}
\end{equation}
Hence,
\begin{equation}
\trs(e^{\omega_u(b)})=\trs\left(e^{\frac{1}{2} (c(e_i)-\widetilde{V}_{u,i})R_u(b)_{ij}(c(e_j)-\widetilde{V}_{u,j})}\right)e^{\frac{1}{2}V_{u,i}V_{u,j}R_u(b)^{ij}}.
\end{equation}
Using this equation and \cite[Lem. 2.12]{MR836726}, we find
\begin{equation}
\label{trsexpomegaub}
\trs(e^{\omega_u(b)})=\trs\left(e^{\frac{1}{2} c(e_i)c(e_j)R_u(b)_{ij}}\right)e^{\frac{1}{2}V_{u,i}V_{u,j}R_u(b)^{ij}}.
\end{equation}

We now compute the term $\trs\left(e^{\frac{1}{2} c(e_i)c(e_j)R_u(b)_{ij}}\right)$. We may assume that $\dot{R}^{X,L}$ (see \eqref{defRpoint}) is the diagonal matrix $\mathrm{diag}(a_1,\dots,a_n)$ in the basis $\{w_j\}_j$. Then
\begin{align}
\label{calcultrsavecRub}
\trs\left(e^{\frac{1}{2} c(e_i)c(e_j)R_u(b)_{ij}}\right)&= \trs\left(\exp\left(-\frac{u}{4} c(e_i)c(e_j)R^L_{ij} -\frac{\ic b}{4}c(e_i)c(e_j)\omega_{ij}\right) \right) \notag \\
&= \trs\left(e^{-u \sum_j a_j\bw^j\wedge i_{\bw_j} +bN_V} \right) e^{\frac{u}{2}\tr(\dot{R}^{X,L})-\frac{nb}{2}}  \\
&=\trs\left( e^{\sum_j(b-u) a_j\bw^j\wedge i_{\bw_j}  } \right) e^{\frac{u}{2}\tr(\dot{R}^{X,L})-\frac{nb}{2}}.\notag
\end{align}
We have
\begin{equation}
\trs\left( e^{\sum_j(b-u a_j)\bw^j\wedge i_{\bw_j}  } \right)= \sum_{I\subset\{1,\dots,n\}} (-1)^{|I|}e^{\sum_{i\in I}(b-ua_i)} =\det \left( \Id-e^{b\Id-u\dot{R}^{X,L}}\right),
\end{equation}
hence \eqref{trsexpomegaub} and \eqref{calcultrsavecRub} give
\begin{equation}
\label{trsexpomegaubprecis}
\trs(e^{\omega_u(b)}) = \det \left( \Id-e^be^{-u\dot{R}^{X,L}}\right)e^{\frac{u}{2}\tr(\dot{R}^{X,L})-\frac{nb}{2}}e^{\frac{1}{2}V_{u,i}V_{u,j}R_u(b)^{ij}}.
\end{equation}

We now turn to the computation of the derivative at $b=0$ of \eqref{trsexpomegaubprecis}. Set
\begin{equation}
\begin{aligned}
&T_{\mathrm{I}}=\left( \left. \derpar{}{b}\right|_{b=0}\det \left( \Id-e^be^{-u\dot{R}^{X,L}}\right) \right)e^{-\frac{1}{2}V_{i}V_{j}(R^L)^{ij}}, \\
&T_{\mathrm{II}}=\det \left( \Id-e^{-u\dot{R}^{X,L}}\right)\left( \left. \derpar{}{b}\right|_{b=0}e^{\frac{1}{2}V_{u,i}V_{u,j}R_u(b)^{ij}} \right).
\end{aligned}
\end{equation}
Here $(R^L)^{ij}$ denotes the coefficients of the inverse of the matrix $(R^L_{ij})_{ij}$. 

By \eqref{trsexpomegaubprecis} we have
\begin{equation}
\label{deriveetrsavecTIetTII}
\left. \derpar{}{b}\right|_{b=0}\trs(e^{\omega_u(b)}) = -\frac{n}{2}\trs(e^{\omega_u(0)}) + \left( T_{\mathrm{I}}+T_{\mathrm{II}} \right)e^{\frac{u}{2}\tr(\dot{R}^{X,L})}.
\end{equation}
First, we get easily
\begin{equation}
\label{TI}
T_{\mathrm{I}}= \det\left(\Id-e^{-u\dot{R}^{X,L}}\right)\tr\left[\left(\Id-e^{u\dot{R}^{X,L}}\right)^{-1}\right]e^{-\frac{1}{2}V_{i}V_{j}(R^L)^{ij}}.
\end{equation}
Secondly, if we define
\begin{equation}
\left(\omega^X_{R^L}\right)_{ij} = \sum_{kl} (R^L)^{ik}\omega_{kl}(R^L)^{kj} \quad \text{and} \quad \omega^X_{R^L}(V,V)=V_iV_j\left(\omega^X_{R^L}\right)_{ij},
\end{equation}
then we have
\begin{equation}
\label{TII}
T_{\mathrm{II}}= \frac{\ic}{2u}\det\left(\Id-e^{-u\dot{R}^{X,L}}\right) \omega^X_{R^L}(V,V)e^{-\frac{1}{2}V_{i}V_{j}(R^L)^{ij}}.
\end{equation}

Finally, using \eqref{RuavectrsOmegau}, \eqref{trsOmegauetOmegatildeu}, \eqref{trsOmegatildeu}, \eqref{trsexpomegaubprecis}, \eqref{deriveetrsavecTIetTII}, \eqref{TI} and \eqref{TII}, and defining
\begin{equation}
\FF^H = e^{-\frac{1}{2}\big(f^\alpha f^\beta R^L_{\alpha\beta}+V_{i}V_{j}(R^L)^{ij}\big)},
\end{equation}
we find
\begin{equation}
\label{expressionRu}
\RR_u = \left\{\frac{\ic}{u}\Big(\omega^{H} +\frac{1}{2}\omega^X_{R^L}(V,V)\Big) + \tr\left[ \big(\Id-e^{u\dot{R}^{X,L}}\big)^{-1} \right]\right\}\det\left( \frac{\dot{R}^{X,L}}{2\pi}\right)\FF^H.
\end{equation}

In the sequel, we will denote with a subscript $\{*\}$ the objects corresponding to the objects defined above in the case where $B$ is a point (e.g. $\RR_u^{\{*\}}$, $\tilde{A}_j^{\{*\}}$, ...). This objects are in fact the ones appearing in \cite{bismut-vasserot-89} and \cite[Sect. 5.5.4]{ma-marinescu}, and are the part of degree 0 of our objects. By \eqref{Atildej}, \eqref{expressionRu} and \cite[(5.5.37)-(5.5.40)]{ma-marinescu} we have
\begin{equation}
\label{liensaveclecasdupoint1}
\begin{aligned}
&\RR_u = \left\{ \frac{\ic}{u}\Big(\omega^{H} +\frac{1}{2}\omega^X_{R^L}(V,V)\Big)\det\left( \frac{\dot{R}^{X,L}}{2\pi}\right) + \RR_u^{\{*\}} \right\} \FF^H \\
& \tilde{A}_j = \tilde{A}_j^{\{*\}}\FF^H \text{ for } j\neq-1\, , \\
& \tilde{A}_{-1}=\left\{\tilde{A}_{-1}^{\{*\}}+\ic\Big(\omega^{H} +\frac{1}{2}\omega^X_{R^L}(V,V)\Big)\det\left( \frac{\dot{R}^{X,L}}{2\pi}\right)\right\} \FF^H. \\
\end{aligned}
\end{equation}
In particular,
\begin{equation}
\label{liensaveclecasdupoint2}
\begin{aligned}
&\tilde{A}_j = 0 \text{ for } j\leq-2, \\
& \RR_u - \frac{\tilde{A}_{-1}}{u}-\tilde{A}_0 = \left\{\RR_u^{\{*\}} - \frac{\tilde{A}^{\{*\}}_{-1}}{u}-\tilde{A}^{\{*\}}_0\right\} \FF^H.
\end{aligned}
\end{equation}

Since $\dot{R}^{X,L} \in \End(T ^{(1,0)}X)$ has positive eigenvalues, we find using \eqref{liensaveclecasdupoint1}, \eqref{liensaveclecasdupoint2} and $\RR_u^{\{*\}}=\tr\left[ \big(\Id-e^{u\dot{R}^{X,L}}\big)^{-1} \right]\det\left( \frac{\dot{R}^{X,L}}{2\pi}\right)$ that for $\Re(z) > 1$,
\begin{equation}
\label{expressionzetatilde-riemann}
\tilde{\zeta}(z)= \left( \int_X  \det\left( \frac{\dot{R}^{X,L}}{2\pi}\right) \tr\left[\big(\dot{R}^{X,L}\big)^{-z}\right] \FF^Hdv_X\right)\frac{1}{\Gamma(z)}\int_{0}^{+\infty}u^{z-1}\frac{e^{-u}}{1-e^{-u}}du.
\end{equation}

Let $\zeta(z)=\sum_{n=0}^{+\infty} \frac{1}{n^z}$ be the Riemann zeta function. Then classically, we have
\begin{equation}
\label{zetadeRiemann}
\begin{aligned}
& \zeta(z)=\frac{1}{\Gamma(z)}\int_{0}^{+\infty}u^{z-1}\frac{e^{-u}}{1-e^{-u}}du, \\
& \zeta(0)=-\frac{1}{2} , \quad \zeta'(0)=-\frac{1}{2}\log(2\pi). 
\end{aligned}
\end{equation}

Finally, \eqref{expressionzetatilde-riemann} and \eqref{zetadeRiemann} yields to
\begin{equation}
\begin{aligned}
\tilde{\zeta}'(0) & = -\zeta(0) \int_X  \det\left( \frac{\dot{R}^{X,L}}{2\pi}\right) \tr\left[\log\big(\dot{R}^{X,L}\big)\right] \FF^Hdv_X+ n\zeta'(0)\int_X\det\left( \frac{\dot{R}^{X,L}}{2\pi}\right) \FF^Hdv_X \\
&=\frac{1}{2} \int_X \det\left( \frac{\dot{R}^{X,L}}{2\pi}\right) \log \left[ \det\left( \frac{\dot{R}^{X,L}}{2\pi}\right) \right] \FF^Hdv_X.
\end{aligned}
\end{equation}

To prove \eqref{expressionzetatilde'(0)}, we now  have to prove that $\FF^H=e^{-R^{L,H'}}$, i.e.
\begin{equation}
\label{independancedeTHM}
f^\alpha f^\beta R^L_{\alpha\beta}+V_{i}V_{j}(R^L)^{ij} = f'^\alpha f'^\beta R^L(f'_\alpha, f'_\beta)
\end{equation}
for some basis $\{f'_\alpha\}_\alpha$ of $T_\R^{H'}M$ (the right hand side does not depend on the choice of $\{f'_\alpha\}_\alpha$).

We choose $f'_\alpha$ so that $f'_\alpha-f_\alpha = u_\alpha \in T_\R X$. Recall that $f^\alpha \in T^*_\R M$ is in fact $f^{\alpha,H}$ with $(\cdot)^H \colon T_\R^*B \isom T^{H,*}_\R M$. On the other hand, if we extend $f'^\alpha \in T^{H',*}_\R M$ to $T_\R^* M = T^*_\R X \oplus T^{H',*}_\R M$ in the obvious way. Then we obtain easily 
\begin{equation}
\label{f'alpha=falpha}
f'^\alpha=f^\alpha \in T^*_\R M.
\end{equation}

Write $u_\alpha=  \sum_i u_\alpha^ie_i$. By \eqref{f'alpha=falpha}, we have on the one hand
\begin{equation}
\label{RH'}
\begin{aligned}
R^L(f'_\alpha, f'_\beta)f'^\alpha f'^\beta &=R^L(f'_\alpha,f_\beta+u_\beta^j e_j)f^\alpha f^\beta \\
&=R^L(f'_\alpha,f_\beta)f^\alpha f^\beta  =\left( R^L_{\alpha\beta}+ u^i_\alpha R^L_{i\beta}\right)f^\alpha f^\beta.
\end{aligned}
\end{equation}
 On the other hand,
 \begin{equation}
 R^L_{i,\beta}=R^L(e_i,f'_\beta-u^k_\beta e_k)=- u^k_\beta R^L_{ik},
 \end{equation}
so we have by \eqref{defVi}
\begin{equation}
\label{sommeViVjRij}
\begin{aligned}
V_{i}V_{j}(R^L)^{ij}&=R^L_{i\alpha}R^L_{j\beta}(R^L)^{ij}f^\alpha f^\beta \\
&=u^k_\alpha R^L_{ik}R^L_{j\beta}(R^L)^{ij}f^\alpha f^\beta =u^j_\alpha R^L_{j\beta}f^\alpha f^\beta.
\end{aligned}
\end{equation}

By \eqref{RH'} and \eqref{sommeViVjRij}, we get \eqref{independancedeTHM}. Theorem \ref{expressionzetatilde'(0)thm} is proved.
\end{proof}

We can now finish the proof of Theorem \ref{mainthm}. Recall that $\Theta^X$ is defined in \eqref{ThetaM-X}. Then
\begin{equation}
\label{detdvx}
\det\left( \frac{\dot{R}^{X,L}}{2\pi} \right)dv_X = \frac{\Theta^{X,n}}{n!}.
\end{equation}

By  \eqref{expressionRu}  we have
\begin{equation}
\label{Atilde0}
\tilde{A}_0  = \frac{n}{2} \det\left( \frac{\dot{R}^{X,L}}{2\pi} \right)\FF^H.
\end{equation}

Now by Corollary \ref{cordvlptuniforme},  \eqref{zetaetzetatilde}, \eqref{zetatildeptozetatilde}, Theorem \ref{expressionzetatilde'(0)thm}, \eqref{detdvx} and \eqref{Atilde0}, we have in the smooth topology on $B$ as $p\to +\infty$
\begin{align}
\label{ccldumainthm}
\psi_{1/\sqrt{p}}\zeta'_p(0) &=  \log(p)p^nB_{0}+p^n\Phi\tilde{\zeta}'(0) +o(p^n) \notag \\
&= \frac{\mathrm{rk}(\xi)}{2}\Phi \left\{  \int_X  \log \left[ \det\left( \frac{p\dot{R}^{X,L}}{2\pi}\right) \right] e^{-R^{L,H'}} \frac{(p\Theta^X)^n}{n!} \right\}+o(p^n)  \\
&=\frac{\mathrm{rk}(\xi)}{2}  \int_X  \log \left[ \det\left( \frac{p\dot{R}^{X,L}}{2\pi}\right) \right] \exp\left(\frac{\ic}{2\pi}R^{L,H'}+p\Theta^X\right)  +o(p^n),\notag
\end{align}
which is \eqref{mainthm-eq}. Thanks to Corollary \ref{cordvlptuniforme}, Theorem \ref{dominationugrand}, \eqref{zetaetzetatilde}and \eqref{expressionzetatilde'}, we can apply  Lemma \ref{Ascoli} to get Theorem \ref{mainthm}.


\subsection{Proof of Theorem \ref{dominationugrand}}
\label{demothm2}

We will use here the notations of Section \ref{localization} and in particular of \eqref{Bp2etDp2}. Let
\begin{equation}
C_p = \frac{1}{p}B_{p}^2=  \frac{1}{p} (D_p^2+R_p).
\end{equation}
By the last line of \eqref{defobjetsavecp}, we have
\begin{equation}
\label{termeadomineravecCp}
p^{-n}\psi_{1/\sqrt{p}}\trs \left[ N_{u/p}e^{-B_{p,u/p}^2} \right] = p^{-n} \trs \left[ N_u \psi_{1/\sqrt{u}}e^{-uC_p}\psi_{\sqrt{u}}\right].
\end{equation}

By \eqref{SpBp1} and \eqref{spdeDp2}, there exists  $\nu>0$ such that for $p$ large
\begin{equation}
\label{SpCp}
\begin{aligned}
&\Sp(D_p/\sqrt{p}) \subset \, ]\!-\infty,-\sqrt{\nu}\,] \cup \{0\} \cup [\sqrt{\nu},+\infty[,  \\
&\Sp(C_p) \subset \{0\} \cup [\nu,+\infty[.
\end{aligned}
\end{equation}
In the sequel, we will assume that \eqref{SpCp} holds for $p\geq 1$. Let $\delta$ be the counterclockwise oriented circle in $\C$ centered at 0 and of radius $\nu/2$, and let $\Delta$ be the contour in $\C$ defined in Figure \ref{contours-deltas}.

\begin{figure}[!h]
\hfill
\setlength{\unitlength}{4cm}
\begin{picture}(2.5,1.125)
\put(0.25,0.5){\vector(1,0){1.95}}
\put(0.60,0.125){\vector(0,1){0.75}}
\multiput(0.60,0.25)(0.09,0){4}{\line(1,0){0.05}}
\multiput(0.60,0.75)(0.09,0){4}{\line(1,0){0.05}}
\thicklines \put(0.90,0.25){\line(1,0){1.25}}
\thicklines \put(0.90,0,75){\line(1,0){1.25}}
\thicklines \put(0.90,0.25){\line(0,1){0.5}}
\put(0.925,0.405){$\frac{3\nu}{4}$}
\put(0.90,0.5){\circle*{0.029}}
\put(1.5,0.25){${}_\blacktriangleright$}
\put(1.5,0.75){${}_\blacktriangleleft$}
\put(1.5,0.35){$\Delta$}
\put(0.4,0.61){$\delta$}
\put(0.54,0.42){0}
\put(0.5,0.22){-1}
\put(0.52,0.72){1}
\put(0.6,0.5){\circle{0.5}}
\put(0.775,0.5){\circle*{0.029}}
\put(0.775,0.55){$\frac{\nu}{2}$}
\end{picture}
\hfill \hfill
\caption{}
\label{contours-deltas}
\end{figure}

Set
\begin{equation}
\begin{aligned}
& \PP_{p,u} = \frac{1}{2i\pi} \psi_{1/\sqrt{u}} \int_\delta e^{-u\lambda}(\lambda-C_p)^{-1}d\lambda, \\
& \K_{p,u} = \frac{1}{2i\pi} \psi_{1/\sqrt{u}} \int_\Delta e^{-u\lambda}(\lambda-C_p)^{-1}d\lambda.
\end{aligned}
\end{equation}
Then 
\begin{equation}
\label{trsavecBupenfctiondePpuetKpu}
p^{-n}\psi_{1/\sqrt{p}}\trs \left[ N_{u/p}e^{-B_{p,u/p}^2} \right] = p^{-n} \trs \left[ N_u (\PP_{p,u}+\K_{p,u})\right].
\end{equation}
We will deal separately with the terms $\PP_{p,u}$ and $\K_{p,u}$.

In the rest of this section, we will work on a subset of $B$ small enough so that we can assume that $M=B\times X$.

\subsubsection*{The term involving $\K_{u,p}$}

\begin{defn}
\label{defnomreendo}
For $A \in \Lambda^\bullet (T^*_\R B)\otimes \End \left( \Omega^{0,\bullet}(X,\xi\otimes L^p) \right)$, let $||A||_\infty$ be the norm of operator of $A$ viewed as an endomorphism of $ L^2(X,\E_p) $ and for $q\in \N^*$, let
\begin{equation}
||A||_q=\left( \tr\left[ (A^*A)^{q/2} \right] \right)^{1/q}.
\end{equation}
\end{defn}

Note that if $||A||_q$ and $||A'||_\infty$ exist, then
\begin{equation}
\label{normeqetinfinie}
||AA'||_q \leq ||A||_q ||A'||_\infty.
\end{equation}

\begin{rem}
We do not  specify the dependance in $b\in B$ or $p\in \N^*$ of the norm $||\cdot||_q$ to make the notations lighter.
\end{rem}

\begin{lemme}
\label{estimeeaveclambda0}
Let $\lambda_0\in \R_-^*$. Then there exists $q_0$ such that for $q\geq q_0$, for $U\in T_\R B$ and $\ell \in \N$, there is a $C>0$ such that for $p\geq 1$
\begin{equation}
\label{estimeeaveclambda0eq}
p^{-n}\left\| \big(\n^{\End(\E_p)}_{U}\big)^\ell(\lambda_0-C_p)^{-q} \right\|_1 \leq C.
\end{equation}
\end{lemme}

\begin{proof}
Set
\begin{equation}
H_p =   D_p^2/p - \lambda_0 .
\end{equation}
Then $H_p$ is a self-adjoint positive generalized Laplacian on $X$. By \cite[Thm. 2.38]{berline-getzler-vergne}, we know that for $k> 1+\frac{\dim_\R X + r}{2}$, the operator $H_p^{-k}$ has a $\mathscr{C}^r$ kernel given for $(x,x')\in X\times X$ by
\begin{equation}
\label{noyauHp-k}
H_p^{-k}(x,x') = \frac{1}{(k-1)!} \int _0^{+\infty}e^{-tH_p}(x,x')t^{k-1}dt.
\end{equation}
Thus,
\begin{equation}
\label{tr(Hp-k)-integral}
\begin{aligned}
\tr\left[H_p^{-k}\right]&= \frac{1}{(k-1)!}\int_X \int _0^{+\infty} \tr\left[e^{-tH_p}(x,x)\right]t^{k-1}dtdv_X(x) \\
&= \frac{1}{(k-1)!} \int _0^{+\infty} \tr\left[e^{-tH_p}\right]t^{k-1}dt.
\end{aligned}
\end{equation}

Now, using the degree 0 of Theorem \ref{cvcenoyaudelachaleur} we find that $p^{-n}\tr\left[e^{-\frac{1}{p}D_p^2}\right]$ converges (along with its derivatives) when $p\to +\infty$. In particular, $p^{-n}\tr\left[e^{-\frac{1}{p}D_p^2}\right]$ and its derivative are bounded. Moreover, $D_p^2$ is  positive. Thus, for $\ell \in \N$, there is  $C>0$ such that for $t\geq 1$ and $p\in \N^*$,
\begin{equation}
\label{tr(Hp-k)-t-grand}
\begin{aligned}
p^{-n} \left|\tr\left[e^{-tH_p}\right]\right|_{\mathscr{C}^\ell(B)} &= p^{-n} \left|\tr\left[e^{-\frac{t}{p}D_p^2}\right]\right|_{\mathscr{C}^\ell(B)}e^{\lambda_0t}\\
& = p^{-n} \left|\tr\left[e^{-\frac{t-1}{p}D_p^2}e^{-\frac{1}{p}D_p^2}\right]\right|_{\mathscr{C}^\ell(B)}e^{\lambda_0t} \\
& \leq p^{-n} \left|\tr\left[e^{-\frac{1}{p}D_p^2}\right]\right|_{\mathscr{C}^\ell(B)}e^{\lambda_0t} \leq Ce^{\lambda_0t}.
\end{aligned}
\end{equation}
Moreover, using the part of degree 0 in Theorem \ref{dvlptuniformedunoyaudelachaleur}, we find  that for any $k,\ell\in \N$, there exist $a_{p,j}\in \R$ and  $C>0$ such that for any $t\in ]0,1]$ and $p\geq 1$, 
\begin{equation}
\label{tr(Hp-k)-t-petit}
\left|p^{-n}\tr \left[ \exp\left( -\frac{t}{p}D_p^2 \right)\right] - \sum_{j=-n-1}^k a_{p,j}t^j\right|_{\mathscr{C}^\ell(B)} \leq Ct^{k+1}.
\end{equation}
To remove the $N_V$ operator in the trace in the above equation, we used that $D_p^2$ preserves the vertical degree.

Splitting the integral in \eqref{tr(Hp-k)-integral} at $t=1$ and using \eqref{tr(Hp-k)-t-grand} and \eqref{tr(Hp-k)-t-petit}, we find that for $k$ large enough,
\begin{equation}
p^{-n}\left| \tr\left[H_p^{-k}\right] \right|_{\mathscr{C}^\ell(B)}\leq C.
\end{equation}

Thus, there exists $q_0\in \N$ such that for $q\geq q_0$ there is $C>0$ such that 
\begin{equation}
\label{estimeedeg0}
p^{-n}\left\| ( \lambda_0-D_p^2/p)^{-q} \right\|_1  =  p^{-n}\tr\left[  H_p^{-q} \right]  \leq C.
\end{equation}

Moreover, by \eqref{SpCp}  there is a $C'>0$ such that for $p\geq 1$,
\begin{equation}
\label{estimee-lambda'0}
\left\| ( \lambda_0-D_p^2/p)^{-1} \right\|_\infty \leq C'.
\end{equation}

A closer look at Bismut's Lichnerowicz formula \eqref{Lichnerowiczeq} and \eqref{convention} enables us to sharpen \eqref{Bp2etDp2}: locally, under the trivialization on $U_{x_k}$ (see Section \ref{localization}), we have
\begin{equation}
\label{structure-Rp/p}
\frac{1}{p}R_p = \frac{1}{p} \0_1 +\0_0,
\end{equation}
were $\0_k$ is a differential operator of order $k$ (which does not depend on $p$). Moreover, in the same way as in Lemma \ref{estimeeelliptique(p)}, we can easily prove from \eqref{nupaveclesGamma} (when $B$ is a point) that
\begin{equation}
||s||_{\sob^1(p)} \leq C\big(||D_ps||_{L^2}+ p||s||_{L^2}\big).
\end{equation}
Consequently, if $s$ is an eigenfunction of $D_p/\sqrt{p}$ for the eigenvalue $\mu$,
\begin{align}
\label{normeinfinede1/pRp}
\frac{1}{p}||R_ps||_{L^2} &\leq \frac{1}{p}||s||_{\sob^1(p)} + ||s||_{L^2} \notag \\
& \leq C\frac{1}{p}||D_ps||_{L^2}+ C'||s||_{L^2}  \\
&\leq C\Big(1+\frac{|\mu|}{\sqrt{p}}\Big) ||s||_{L^2} \leq C(1+|\mu|) ||s||_{L^2} .\notag
\end{align}
This estimate yields to
\begin{equation}
\label{estimation1/pRp}
\frac{1}{p}\left\|R_p(\lambda_0-D_p^2/p)^{-1}\right\|_\infty  \leq C \sup_{\mu\in [\sqrt{\nu},+\infty[} \frac{1+\mu}{|\lambda_0 - \mu^2|}\leq C'.
\end{equation}

As in \eqref{resolvantesBp2etDp2}, we have
\begin{equation}
\label{resolvantesCpetDp2}
(\lambda_0 - C_p)^{-1} = (\lambda_0 - D_p^2/p)^{-1} + (\lambda_0 - D_p^2/p)^{-1}(R_p/p)(\lambda_0 - D_p^2/p)^{-1}+\: \cdots,
\end{equation}
with only finitely many terms (as $R_p$ is sum of elements of positive degree in $\Lambda^\bullet(T^*_\R B)$). Thus, for $q\in \N^*$, $(\lambda_0 - C_p)^{-q}$ is a sum of terms of the form
\begin{equation}
\label{termes-ds-resolvante^a}
(\lambda_0 - D_p^2/p)^{-k_0}R_p/p \cdots R_p/p (\lambda_0 - D_p^2/p)^{-k_i},
\end{equation}
with $0\leq i \leq \dim_\R B$, $k_j\geq 1$ and $\sum_j k_j = q+i$. In particular, there exist $j_0$ such that $k_{j_0}\geq \frac{q}{\dim_\R B+1}$. Thus, if $q$ is large enough, then $(\lambda_0 - C_p)^{-q}$ is a sum of product of terms of the form \eqref{termes-ds-resolvante^a} -- which are bounded for $\|\cdot\|_\infty$ by \eqref{estimee-lambda'0} and \eqref{estimation1/pRp} -- and of $(\lambda_0 - D_p^2/p)^{-q_0}$. Thus, form \eqref{normeqetinfinie} and \eqref{estimeedeg0},  we get Lemma \ref{estimeeaveclambda0} for $\ell=0$.

Using \eqref{resolvantesCpetDp2}, we find that $\n^{\End(\E_p)}_{U}(\lambda_0 - C_p)^{-q}$ is a sum of terms
\begin{equation}
\label{termes-ds-der-resolvante^a}
(\lambda_0 - D_p^2/p)^{-k_0}A_{k_1}(p)\cdots A_{k_i}(p) (\lambda_0 - D_p^2/p)^{-k_i},
\end{equation}
with $0\leq i \leq \dim_\R B+1$, $k_j\geq 1$, $\sum_j k_j = q+i$ and
\begin{equation}
A_{k_j}(p) \in \left\{ R_p/p, \n^{\End(\E_p)}_{U} R_p/p,\n^{\End(\E_p)}_{U} D_p^2/p \right\}.
\end{equation}

Thus, using the same reasoning as above with $R_p/p$ replaced by $A_{k_j}(p)$, to prove Lemma \ref{estimeeaveclambda0} for $\ell=1$, we only have to show that there exists $C>0$ such that for any $p\in \N^*$
\begin{equation}
\label{normeinfinie-Aij(p)xresolvante}
\left \| A_{k_j}(p)(\lambda_0 - D_p^2/p)^{-1} s\right\|_{L^2} \leq C\|s\|_{L^2}.
\end{equation}
By \eqref{estimation1/pRp}, estimation \eqref{normeinfinie-Aij(p)xresolvante} holds if $A_{k_j}(p) =R_p/p$. Also, as $\n^{\End(\E_p)}_{U} R_p/p$ has the same structure as $R_p/p$ in \eqref{structure-Rp/p}, we can show that \eqref{normeinfinie-Aij(p)xresolvante} holds if $A_{k_j}(p) =\n^{\End(\E_p)}_{U}R_p/p$. We only have the case $A_{k_j}(p)=\n^{\End(\E_p)}_{U} D_p^2/p$ left to treat.

First, observe that for any operator $A$, it is equivalent to show that $\|As\|_{L^2} \leq C \|s\|_{L^2}$ for any section or for any section supported in a ball of radius $\e>0$. We fix $x_0 \in X$, and $\e>0$ as in Section \ref{rescalingBpu}, and we consider a section $s$ supported in $B^X(x_0,\e)$. We will use here all the notations, identifications an trivializations of Section \ref{rescalingBpu}. We extend $s$ by 0 to get an element of $\smooth_c(T_{\R,x_0}X, \E_{x_0})$. To simplify, let us denote $\n^{\End(\E_p)}_{U} D_p^2/p(\lambda_0 - D_p^2/p)^{-1}$ by $A_p(\lambda_0)$. Let $\sigma_t = S_t^{-1}\kappa^{1/2}s$ and $\mathscr{A}_t(\lambda_0) = S_t^{-1} \kappa^{1/2} A_p(\lambda_0) \kappa^{-1/2} S_t$. We have
\begin{equation}
\| A_p(\lambda_0)s\|_{L^2}^2 = t^{2n} \int_{\R^{2n}} \left | \kappa^{1/2}(A_p(\lambda_0)s)\right|^2(tZ)dv_{TX}(Z) = t^{2n} \int_{\R^{2n}} \left | \mathscr{A}_t(\lambda_0)\sigma_t \right|^2(Z) dv_{TX}(Z).
\end{equation}
Thus, if we prove that
\begin{equation}
\label{norme-de-A_t}
\|  \mathscr{A}_t(\lambda_0) \|_{t}^{0,0} \leq C,
\end{equation}
we will find
\begin{equation}
\label{A_p(lambda_0)-bornee}
\| A_p(\lambda_0)s\|_{L^2}^2 \leq C t^{2n} \int_{\R^{2n}} \left |\sigma_t \right|^2(Z) dv_{TX}(Z) = C \int_X |s|^2(x)dv_X(x) = C \|s\|_{L^2}^2,
\end{equation}
which is the estimate we needed. To prove \eqref{norme-de-A_t}, observe that over $B^{T_{\R,x_0}X}(0,\e)$ and under the identification $\E_p \simeq \E$, we have
 \begin{equation}
 \begin{aligned}
&\n^{\End(\E_p)} = \n^{\End(\E)} = \n + [\Gamma_1, \cdot],\\
&\n^{\End(\E_p)} (\n^p_{e_i}) = p (\n_U\Gamma^L)(e_i)+ R^\E(U,e_i).
\end{aligned}
\end{equation}
Hence, $\n^{\End(\E_p)}D_p^2/p$ has the form  
\begin{equation}
\label{structure-derivee-Dp2/p}
\n^{\End(\E_p)}D_p^2/p  = a_{i,j}(Z)\frac{1}{p}\n^{p,(0)}_{e_i}\n^{p,(0)}_{e_j} + \Big(\frac{1}{\sqrt{p}}b_j(Z)+\sqrt{p}c_j(Z)\Big)\frac{1}{\sqrt{p}}\n^{p,(0)}_{e_j}+ \frac{1}{p}d(Z)+e(Z),
\end{equation}
where $a_{i,j}$, $b_j$, $c_j$, $d$ and $e$ are bounded (along with their derivatives). Moreover, observe that $(\n_U\Gamma^L)(e_i)(Z)=O(|Z|)$ (apply \cite[(1.2.30)]{ma-marinescu} and observe that $\n_U$ only differentiate the parameter of the basis $B$), and that $c_j(Z)$ comes from the terms $(\n_U\Gamma^L)(e_i)$, so we have $c_j(0)=0$. Using this fact and \eqref{structure-derivee-Dp2/p}, we find that $t^{-1}c_j(tZ)$ is bounded as $t\to 0$ and that
\begin{multline}
\label{structure-derivee-Dp2/p-rescaled}
S_t^{-1} \kappa^{1/2}\Big(\n^{\End(\E_p)}D_p^2/p\Big)\kappa^{-1/2} S_t = \\
a_{i,j}(tZ)\n^{(0)}_{t,e_i}\n^{(0)}_{t,e_j} + \big(b_j(tZ)+t^{-1}c_j(tZ)\big)\n^{(0)}_{t,e_j}+t^2d(tZ)+e(tZ).
\end{multline}
Using this structure, the fact that $ \mathscr{A}_t(\lambda_0)=S_t^{-1} \kappa^{1/2}\big(\n^{\End(\E_p)}D_p^2/p\big)\kappa^{-1/2} S_t\big(\lambda_0-\LL_t^{(0)}\big)^{-1}$ and arguments similar to those in the proof of  Propositions \ref{estimationsresoprop}-\ref{norme(m,m+1)resolvante} (see \cite[Thms. 1.6.8-1.6.10]{ma-marinescu}), we find \eqref{norme-de-A_t}.

We have proved Lemma \ref{estimeeaveclambda0} for $\ell=1$. The case $\ell\geq 1$ is similar.
\end{proof}

\begin{prop}
\label{termeavecKpu}
For any $\ell \in \N$, there exist $a,C>0$ such that for $p\geq 1$ and $u\geq 1$,
\begin{equation}
\label{termeavecKpu-eq}
p^{-n}\big| \trs\left[N_u \K_{p,u}\right] \big|_{\mathscr{C}^\ell(B)} \leq Ce^{-au}.
\end{equation}
\end{prop}

\begin{proof}
First, note that  \eqref{estimee-lambda'0} is still true if we replace $\lambda_0$ by $\lambda\in\delta\cup \Delta$, and that the constant in the right hand side can be chosen independently of $\lambda_0\in \Delta$, that is: there exists $C>0$ such that 
\begin{equation}
\label{estimee-lambda}
\left\| ( \lambda-D_p^2/p)^{-1} \right\|_\infty \leq C, \quad \forall \lambda \in \delta\cup\Delta.
\end{equation}
In the same way,  \eqref{estimation1/pRp} is also true if we replace $\lambda_0$ by $\lambda\in \Delta$ and we have $\sup_{\mu\geq \sqrt{c}} \frac{1+\mu}{|\lambda - \mu^2|}\leq C|\lambda|$, hence there exists $C>0$ such that for $\lambda \in \delta\cup\Delta$,
\begin{equation}
\label{estimation1/pRp-lambda}
\frac{1}{p}\left\|R_p(\lambda-D_p^2/p)^{-1}\right\|_\infty  \leq C |\lambda|.
\end{equation}
 Thus by \eqref{resolvantesCpetDp2}, \eqref{estimee-lambda} and \eqref{estimation1/pRp-lambda}, there exists $C>0$ such that for $p\geq 1$ and $\lambda \in \delta\cup\Delta$,
\begin{equation}
\label{estimeenormeinfinieunif}
\left\| (\lambda - C_p)^{-1} \right\|_\infty \leq C|\lambda|.
\end{equation}

For $\lambda \in \Delta$ and $\lambda_0\in \R_{-}^*$, we have
\begin{equation}
\label{lambdaetlambda02-puissanceq}
\begin{aligned}
&(\lambda - C_p)^{-1}=(\lambda_0 - C_p)^{-1}-(\lambda-\lambda_0)(\lambda_0 - C_p)^{-1}(\lambda - C_p)^{-1}, \\
&(\lambda - C_p)^{-q}=(\lambda_0 - C_p)^{-q}\left(1-(\lambda-\lambda_0)(\lambda - C_p)^{-1}\right)^q.
\end{aligned}
\end{equation}
From \eqref{normeqetinfinie}, \eqref{estimeeaveclambda0eq} \eqref{estimeenormeinfinieunif} and \eqref{lambdaetlambda02-puissanceq} we find that for $\lambda \in \delta\cup\Delta$,
\begin{equation}
\label{estimationresolvanteCp^q}
\begin{aligned}
\left\| (\lambda-C_p)^{-q} \right\|_1 &\leq \left\| (\lambda_0-C_p)^{-q} \right\|_1  \left\| \left(1-(\lambda-\lambda_0)(\lambda - C_p)^{-1}\right)^q \right\|_\infty \\
&\leq C|\lambda|^{2q}\left\| (\lambda_0-C_p)^{-1} \right\|_q \leq C|\lambda|^{2q}p^{n}.
\end{aligned}
\end{equation}

On the other hand,  we have
\begin{equation}
\label{autreecrituredeKpu}
\K_{p,u} = \frac{1}{2i\pi} \psi_{1/\sqrt{u}} \int_\Delta \frac{(q-1)!}{(-u)^{q-1}}e^{-u\lambda}(\lambda-C_p)^{-q}d\lambda,
\end{equation}
and there exist $\kappa,K>0$ such that for $\lambda \in \delta\cup\Delta$,
\begin{equation}
\label{Relambda}
\Re(\lambda)\geq K |\lambda|\geq \kappa.
\end{equation}
From \eqref{estimationresolvanteCp^q}, \eqref{autreecrituredeKpu} and \eqref{Relambda} we deduce that there exist $a,C>0$ such that for $p\in\N^*, u\geq 1$,
\begin{equation}
\begin{aligned}
p^{-n}\big| \trs\left[N_u \K_{p,u}\right] \big| &\leq p^{-n}C(1+\sqrt{u}^{-n/2}) \left\| \int_\Delta \frac{(q-1)!}{(-u)^{q-1}}e^{-u\lambda}(\lambda-C_p)^{-q}d\lambda \right\|_1 \notag \\
& \leq p^{-n}C \int_\Delta |\lambda|^{2q}e^{-uK|\lambda|}\left\| (\lambda-C_p)^{-q} \right\|_1d\lambda  \leq Ce^{-au}.
\end{aligned}
\end{equation}
Proposition \ref{termeavecKpu} is proved in the case where $\ell=0$.

We now turn to the case $\ell=1$. Equation \eqref{lambdaetlambda02-puissanceq} implies
\begin{multline}
\label{lambdaetlambda02-puissanceq-derivee}
\n^{\End(\E_p)}_U(\lambda - C_p)^{-q}=\left[\n^{\End(\E_p)}_U(\lambda_0 - C_p)^{-q}\right]\left(1-(\lambda-\lambda_0)(\lambda - C_p)^{-1}\right)^q \\
+ (\lambda_0 - C_p)^{-q}\left[\n^{\End(\E_p)}_U \left(1-(\lambda-\lambda_0)(\lambda - C_p)^{-1}\right)^q\right].
\end{multline}

We claim that there is $C,N>0$ such that for $\lambda \in \delta\cup\Delta$
\begin{equation}
\label{norme-infinie-autre-bout}
\left\|\n^{\End(\E_p)}_U \left(1-(\lambda-\lambda_0)(\lambda - C_p)^{-1}\right)^q\right\|_\infty \leq C|\lambda|^N. 
\end{equation}
Indeed, the arguments of Propositions \ref{estimationsresoprop}-\ref{norme(m,m+1)resolvante} that enables us to prove \eqref{norme-de-A_t} from \eqref{structure-derivee-Dp2/p-rescaled} also shows that \eqref{norme-de-A_t} is still true if we replace therein $\lambda_0$ by $\lambda \in \delta\cup\Delta$ and that moreover there exists $N>0$ such that $\|  \mathscr{A}_t(\lambda) \|_{t}^{0,0} \leq C |\lambda|^N$. Hence, as in \eqref{A_p(lambda_0)-bornee}, we have $\|  A_p(\lambda) \|_\infty \leq C |\lambda|^N$, i.e.,
\begin{equation}
\label{norme-de-A_p-lambda}
\left\|\n^{\End(\E_p)}_{U} D_p^2/p(\lambda_0 - D_p^2/p)^{-1}\right\|_\infty \leq C |\lambda|^N.
\end{equation}
Thus, decomposing $\n^{\End(\E_p)}_U \left(1-(\lambda-\lambda_0)(\lambda - C_p)^{-1}\right)^q$ as a polynomial in $\lambda$ whose coefficients have the form \eqref{termes-ds-der-resolvante^a}, and using \eqref{estimee-lambda}, \eqref{estimation1/pRp-lambda} and \eqref{norme-de-A_p-lambda}, we find \eqref{norme-infinie-autre-bout}.

Then, by \eqref{normeqetinfinie}, \eqref{estimeeaveclambda0eq}, \eqref{lambdaetlambda02-puissanceq-derivee} and \eqref{norme-infinie-autre-bout}, we find that there is $N'>0$ such that
\begin{equation}
p^{-n} \left\| \n^{\End(\E_p)} (\lambda-C_p)^{-q} \right\|_1 \leq C|\lambda|^{N'}.
\end{equation}
Hence,
\begin{equation}
\begin{aligned}
 p^{-n}\left|\n^{\Lambda^\bullet(T^*_\R B)} \trs\left[N_u \K_{p,u}\right] \right| & = p^{-n}\left| \trs\left[\n^{\End(\E_p)}(N_u \K_{p,u})\right]\right| \\
& \leq p^{-n}C \int_\Delta e^{-uK|\lambda|}\left\| \n^{\End(\E_p)} (\lambda-C_p)^{-q} \right\|_1d\lambda  \\
&\leq Ce^{-au}.
\end{aligned}
\end{equation}
This proves \eqref{termeavecKpu-eq} for $\ell=1$.

The proof of Proposition \ref{termeavecKpu} for $\ell\geq1$ relies on similar arguments.
\end{proof}

\subsubsection*{The term involving $\PP_{p,u}$} 

\begin{prop}
\label{termeavecPpu}
For any $\ell \in \N$, there is a $C>0$ such that for any $p\geq 1$ and $u\geq 1$, 
\begin{equation}
\label{termeavecPpueq}
p^{-n}\big| \trs\left[N_u \PP_{p,u}\right] \big|_{\mathscr{C}^\ell(B)}\leq \frac{C}{\sqrt{u}}.
\end{equation}
\end{prop}

\begin{proof}
We first rewrite $\PP_{p,u}$. As $C_p$ has no eigenvalues between the two circles $\delta$ and $\delta/u$, we have 
\begin{equation}
\begin{aligned}
\PP_{p,u} &= \frac{1}{2i\pi} \psi_{1/\sqrt{u}} \int_{\delta/u} e^{-u\lambda}(\lambda-C_p)^{-1}d\lambda  \\
&=\frac{1}{2i\pi} \psi_{1/\sqrt{u}} \int_{\delta} e^{-\lambda}(\lambda-uC_p)^{-1}d\lambda.
\end{aligned}
\end{equation}

We now use the technique of \cite[Sect. 9.13]{MR1623496}. Let $C_p^{(0)} = \frac{1}{p}D_p^2$ be the part of $C_p$ of degree 0 in $\Lambda^\bullet(T^*_\R B)$. We denote by $P_p$ the orthogonal projection form $\Omega^{0,\bullet}(X,\xi\otimes L^p)$ to the kernel of $D_p^2$, and $P_p^\perp=1-P_p$. We will make the abuse of notation $(C_p^{(0)})^{-1} = P_p^\perp(C_p^{(0)})^{-1}P_p^\perp$. Finally, we denote $R_p/p$ by $\tilde{R}_p$. Then for $\lambda\in \delta$,
\begin{align}
\label{dvlpe-lambdaXresolvante}
e^{-\lambda}(\lambda-uC_p)^{-1}&= \left( \sum_{k\geq0} \frac{(-1)^k}{k!}\lambda^k \right) \left( \sum_{\ell\geq0} (\lambda-uC^{(0)}_p)^{-1}(u\tilde{R}_p)\dots(u\tilde{R}_p)(\lambda-uC^{(0)}_p)^{-1} \right), \notag\\
(\lambda-uC^{(0)}_p)^{-1}& = \frac{1}{\lambda}P_p + (\lambda-uC^{(0)}_p)^{-1}P_p^\perp.
\end{align}
Moreover, $\lambda \mapsto (\lambda-uC^{(0)}_p)^{-1}P_p^\perp$ is an holomorphic function on the interior of $\delta$, so \eqref{dvlpe-lambdaXresolvante} yields to
\begin{equation}
\label{grossesomme}
\PP_{p,u}=\psi_{1/\sqrt{u}} \sum_{\ell=0}^{\dim_\R B} \sum\limits_{\substack{1\leq i_0 \leq \ell+1 \\ j_1,\dots,j_{\ell+1-i_0}\geq0  \\ \sum_{m=1}^{\ell+1-i_0}j_m\leq i_0-1 }} \frac{(-1)^{\ell-\sum_m j_m}}{(i_0-1-\sum_m j_m)!} T_{p,1}(u\tilde{R}_p)T_{p,2}\dots(u\tilde{R}_p)T_{p,\ell+1},
\end{equation}
where $P_p$ appears $i_0$ times among the $T_{p,j}$'s and the other terms are given respectively by $(uC^{(0)}_p)^{-(1+j_1)}$,\dots,$(uC^{(0)}_p)^{-(1+j_{\ell+1-i_0})}$.

As $R_p$ is the part of positive degree of $B_p^2$ and $B_p^{(0)}=D_p$ (see \eqref{B^0=D}),   we can decompose $R_p$ with respect to the degree in $\Lambda^\bullet(T^*_\R B)$:
\begin{equation}
R_p = R_p^{(1)}+R_p^{(\geq2)} \quad\text{with} \quad R_p^{(1)} = \left[B_p^{(1)},D_p\right].
\end{equation}

We can rewrite the sum \eqref{grossesomme} as a sum of products of terms
\begin{equation}
\label{A_ipossible}
\begin{aligned}
&A_1(u\psi_{1/\sqrt{u}}\tilde{R}^{(1)}_p)A_2 \quad \text{or} \quad A_1(u\psi_{1/\sqrt{u}}\tilde{R}^{(\geq2)}_p)A_2, \\
&  A_i \in \{P_p, (uC^{(0)}_p)^{-(1+j)},(uC^{(0)}_p)^{-(1+j)/2}\}.
\end{aligned}
\end{equation}
 Moreover, observe that 
\begin{equation}
P_p\left[B_p^{(1)},D_p\right]P_p = 0.
\end{equation}
As a consequence, the possible degrees in $u$ of a term $A_1(u\psi_{1/\sqrt{u}}\tilde{R}^{(1)}_p)A_2=A_1(\sqrt{u}\tilde{R}^{(1)}_p)A_2$ are:
\begin{equation}
\left\{
\begin{aligned}
&\deg_u P_p\sqrt{u}\tilde{R}^{(1)}_pP_p = -\infty ,\\
&\deg_u  P_p\sqrt{u}\tilde{R}^{(1)}_p(uC^{(0)}_p)^{-r} = \deg_u (uC^{(0)}_p)^{-r}\sqrt{u}\tilde{R}^{(1)}_pP_p  = \frac{1}{2} -r,  \\
&\deg_u (uC^{(0)}_p)^{-r}\sqrt{u}\tilde{R}^{(1)}_p(uC^{(0)}_p)^{-r'}=\frac{1}{2}-r-r'.
\end{aligned}
\right.
\end{equation}
In any case, by \eqref{A_ipossible}, these terms are polynomials in $1/\sqrt{u}$. 

Concerning the terms $A_1(u\psi_{1/\sqrt{u}}\tilde{R}^{(\geq2)}_p)A_2$, as $R^{(\geq2)}_p$ is a sum of terms of degree greater than 2 in $\Lambda^\bullet(T^*_\R B)$ we find that the powers of $u$ appearing are: 
\begin{equation}
\left\{
\begin{aligned}
 u^{-j/2} \quad &\text{in } \quad P_p(u\psi_{1/\sqrt{u}}\tilde{R}^{(\geq2)}_p)P_p,  \\
  u^{-r-j/2} \quad &\text{in } \quad P_p(u\psi_{1/\sqrt{u}}\tilde{R}^{(\geq2)}_p)(uC^{(0)}_p)^{-r} \text{ or } (uC^{(0)}_p)^{-r}(u\psi_{1/\sqrt{u}}\tilde{R}^{(\geq2)}_p)P_p ,   \\
 u^{-r-r'-j/2} \quad &\text{in } \quad (uC^{(0)}_p)^{-r}(u\psi_{1/\sqrt{u}}\tilde{R}^{(\geq2)}_p)(uC^{(0)}_p)^{-r'},
\end{aligned}
\right.
\end{equation}
where $r,r'\in \frac{1}{2}\N^*$ and $2\leq j \leq \dim_\R B$. This shows that $\PP_{p,u}$  is  in  $\C_N\Big[\frac{1}{\sqrt{u}}\Big]$ for some uniform $N\in \N$. Furthermore, in each term of the sum \eqref{grossesomme} $i_0\geq 1$ so $P_p$ --which is a projector on a finite dimensional space--   appears at least one time. Hence there exist $c_k(p)\in \Omega^\bullet(B)$ such that
\begin{equation}
\label{polynomial}
p^{-n}\trs\left[N_u \PP_{p,u}\right] = \sum_{k=0}^K c_k(p) u^{-k/2}.
\end{equation}

Moreover, by \eqref{normeinfinede1/pRp}, we have for $r,r'\geq \frac{1}{2}$
\begin{equation}
\label{normeinftermespossibles}
\left\{
\begin{aligned}
& ||P_p\tilde{R}_pP_p||_\infty \leq C  ,\\
&  ||P_p\tilde{R}_p(C^{(0)}_p)^{-r}||_\infty \: , \: ||(C^{(0)}_p)^{-r}\tilde{R}_pP_p||_\infty  \leq C \sup_{\mu \geq \sqrt{\nu}} \Big( (1+\mu) \mu^{-2r}\Big) \leq C',  \\
& ||(C^{(0)}_p)^{-r}\tilde{R}_p(C^{(0)}_p)^{-r'}||_\infty \leq C''.
\end{aligned}
\right.
\end{equation}
Therefore, each term in the sum \eqref{grossesomme} is a product of uniformly bounded terms, in which $P_p$ appears at least once (because $i_0\geq 1$). Thus,
\begin{equation}
\label{coeffbornŽs}
|c_k(p)| \leq p^{-n}C \dim \ker (D_p^2)=p^{-n}C\dim H^0(X,\xi\otimes L^p) \leq C.
\end{equation}
For the last inequality we have used Riemann-Roch-Hirzebruch theorem (see e.g. \cite[Thm. 1.4.6]{ma-marinescu}) and Kodaira vanishing theorem.

Finally, using Theorem \ref{thmutoinf}, \eqref{trsavecBupenfctiondePpuetKpu} and Proposition \ref{termeavecKpu} we have for $p$ large fixed
\begin{equation}
p^{-n}\trs\left[N_u \PP_{p,u}\right] \limarrow{u}{+\infty} p^{-n}\psi_{1/\sqrt{p}} \trs \left[ N_V \exp(-(\n^{H(X,\xi\otimes L^p|_{X})})^2) \right] =0.
\end{equation}
Thus $c_0(p)=0$ and by \eqref{polynomial} and \eqref{coeffbornŽs} we find \eqref{termeavecPpueq} in the case $\ell=0$.

We now turn to the case $\ell=1$. By decomposing as above $\PP_{p,u}$ in a sum of product of polynomial in $1/\sqrt{u}$, and then differentiating in the direction $U\in T_\R B$, we find that $\n^{\End(\E_p)}_U\PP_{p,u}$ is also a sum of product of polynomial in $1/\sqrt{u}$. Thus, here again there exist $c'_k(p)\in \Omega^\bullet(B)$ such that
\begin{equation}
\label{polynomial-der}
p^{-n}\n^{\Lambda^\bullet(T^*_\R B)}_U\trs\left[N_u \PP_{p,u}\right] =p^{-n}\trs\left[\n^{\End(\E_p)}_U N_u \PP_{p,u}\right]= \sum_{k=0}^K c'_k(p) u^{-k/2}.
\end{equation}
To conclude the proof as above, we need not only the uniform bounds given in \eqref{normeinftermespossibles}, but also of the derivative of the terms appearing therein. To obtain these bounds, we use similar reasonings that those undertaken in Propositions \ref{estimeeaveclambda0} and \ref{termeavecKpu} (in particular the proof of \eqref{norme-de-A_p-lambda}) to handle the derivatives.

 For $\ell\geq 1$, the reasoning is similar.
\end{proof}

With \eqref{trsavecBupenfctiondePpuetKpu} and Propositions \ref{termeavecKpu} and \ref{termeavecPpu}, we have proved Theorem \ref{dominationugrand}.


\section{Torsion forms associated with a direct image.}
\label{Sectionimagedirecte}

The purpose of this section is to prove Theorem \ref{mainthm2}.

We recall some notations. Let $N$, $M$ and $B$ be three complex manifolds. Let $\pi_1 \colon N \to M$ and $\pi_2 \colon M \to B$ be holomorphic fibrations with compact fiber $Y$ and $X$ respectively.  Then $\pi_3:=\pi_2\circ\pi_1\colon N \to B$ is a holomorphic fibration, whose compact fiber is denoted by $Z$. We denote by $n_X$ (resp. $n_Y$, $n_Z$) the complex dimension of $X$ (resp. $Y$, $Z$). Note that $\pi_1|_{Z}\colon Z\to X$ is a holomorphic fibration with fiber $Y$. This is summarized in the following diagram:
$$  \xymatrix{
 Y\ar[r]  & Z\ar[r] \ar[d]^{\pi_1} & N \ar[d]^{\pi_1} \ar[rd]^{\pi_3} \\
 & X \ar[r]& M\ar[r]_{\pi_2} & B
 }$$

Let $(\pi_2,  \omega^M)$ be a structure of Hermitian  fibration (see Section \ref{Sect-Hermitian fibration}). We denote by $T^H_BM$ the corresponding horizontal space.

Let $(\xi, h^\xi)$ be a holomorphic Hermitian vector bundle on $M$, and let $(\eta, h^\eta)$ be a holomorphic Hermitian vector bundle on $N$. Let $(L,h^L)$ be a holomorphic Hermitian line bundle on $N$. We denotes its Chern connection by $\n^L$, and the corresponding curvature by $R^L$. By Assumption \ref{RLpositive2}, $L$ is positive along the fibers of $\pi_3$. In particular, $\frac{\ic}{2\pi}R^L$  defines metric  $g^{T_\R Z}$ on  $T_\R Z$,  by the formula
\begin{equation}
\label{def-metrci-fibre}
g^{T_\R Z}(U,V) = \frac{\ic}{2\pi}R^L (U,J^{T_\R Z}V) \: , \quad U,V \in T_\R Z.
\end{equation}
Similarly, we get a metric $g^{T_\R Y}$ on $T_\R Y$.
  
Recall that
\begin{equation}
\label{def-THN}
T^H_BN = (TZ)^\perp , \quad T^H_MN = (TY)^\perp, \quad T^H_XZ=T^H_MN\cap TZ,
\end{equation}
where the orthogonal complements are taken with respect to $R^L$.  Also,  $\dot{R}^{X,L}\in  \pi_3^*\End(TX)$  is the Hermitian matrix such that for any $U,V\in TX$, if we denote their horizontal lifts by $U^H,V^H \in T^H_XZ$, then
\begin{equation}
\label{defRpoint2}
R^L(U^H,\overline{V}^H) = \langle \dot{R}^{X,L} U, V \rangle _{h^{TX}}.
\end{equation}
By Assumption \ref{RLpositive2}, $\dot{R}^{X,L}$ is positive definite. Finally, set
\begin{equation}
\Theta^N = \frac{\ic}{2\pi}R^L \quad \text{and} \quad \Theta^Z = \frac{\ic}{2\pi}R^L|_{T_{\R}Z\times T_{\R}Z}.
\end{equation}
We extend $\Theta^Z$ to $T_\R N=T_\R Z\oplus (T_\R Z)^{\perp,\Theta^N}$ by zero. 

Recall that we have assumed that (for $p$ large) the direct image both $R^\bullet\pi_{1*}(\eta\otimes L^p)$ is locally free. Let $F_p:=H^\bullet\big(Y,(\eta\otimes L^p)|_{Y}\big)$ the corresponding bundle, endowed with the $L^2$ metric $h^{F_p}$ induced by $h^\eta$, $h^L$ and $g^{T_\R Y}$. 

We have also assumed that (for $p$ large) $R^\bullet \pi_{2*}(\xi\otimes F_p)$ of  is locally free and that we have $R^\bullet \pi_{2*}(\xi\otimes F_p)\simeq R^\bullet \pi_{3*}(\pi^*_1\xi\otimes\eta\otimes L^p)$.

The objects corresponding to this situation will be denoted by
\begin{equation}
\begin{aligned}
& E_{p,b}^k = \smooth \left( X_b, \left(\Lambda^{0,k}(T^*X)\otimes \xi\otimes F_p \right)|_{X_b} \right),\\
& \n^p = \n^{E_p,LC}, \\
& \db^p = \text{Dolbeault operator of } E_p, \\ 
& D_p = \db^p + \db^{p,*},\\
& B_p\, , \, B_{p,u} = \text{associated superconnections as in \eqref{defBu}},\\
&\n^p_u = \text{connection corresponding to \eqref{defnablau} associated  with } \xi\otimes F_p=\n_u\otimes1 +1\otimes\n^{F_p}.
\end{aligned}
\end{equation}
Then we can construct as in Section \ref{torsion} the holomorphic analytic torsion forms $\mathscr{T}(\omega^M, h^{\xi\otimes F_p})$ associated with $\omega^M$ and $(\xi\otimes F_p,h^{\xi\otimes F_p})$.

The strategy of proof of Theorem \ref{mainthm2} will be formally the same as for Theorem \ref{mainthm}. However, the main difficulty is that in the case of a line bundle (that is $Y=\{*\}$), $F_p=L^p$ is of constant dimension 1 so locally the operators have their coefficients in a fixed space (see Remark \ref{independancedeSL}), whereas it is not the case here. To overcome this issue, we will use an approach inspired by \cite{BMZweb,MR2838248}, that is we will consider all the operators depending on $p$ at once with the formalism of Toeplitz operators of \cite{ma-marinescu}. More precisely, we will consider the family $\{B_{p,u}^2,p\in\N\}$ as a differential operators with coefficient in the Toeplitz algebra (see \eqref{Bup2=Toeplitz}). A crucial point is to use the operator norm on matrices to have boundedness properties of Toeplitz operators. Here, the first difficulty is that  there is no longer a limiting operator (as the space changes), but we can show that instead there is an asymptotic operator with Toepltiz coefficients. The problem is then that we cannot compute its heat kernel explicitly (with comparison to \eqref{noyaudelachaleurLinfini}), but using the properties of operator with Toeplitz coefficients developed in Section \ref{operatorwithToeplitzcoeff}, we can nonetheless give an asymptotic formula. An other difficulty comes from the fact that we cannot use the same method to prove the uniform  development of the heat kernel as $u\to 0$ as we did before (see the proofs of Theorems  \ref{dvlptuniformedunoyaudelachaleur} and  \ref{dvlptuniformedunoyaudelachaleur2}), and we cannot hope to prove that the coefficients converges. Instead, we prove that the coefficients are asymptotic to certain Toepltiz operators.

Once again, to simplify the statements in the following, we will assume that $B$ is compact. However, the reader should be aware of the fact that the constants appearing in the sequel depends on the compact subset of $B$ we are working on.

This section is organized as follows.  In Subsections \ref{toeplitzop} and \ref{Infinitedimensionalbundles}, we recall the formalism of Toepltiz operators. In Subsection \ref{operatorwithToeplitzcoeff}, we introduce operators with Toeplitz coefficients and show some properties of their Schwartz kernels. In Subsection \ref{localization2}, we show that our problem is local. In Subsection \ref{sectioncvce2}, we rescale the Bismut superconnection and compute the limit operator, then we obtain the convergence of the heat kernel in Theorem \ref{cvcenoyaudelachaleur2}. Then, in Subsection \ref{Asymptoticofthetorsionforms2}, we prove our main theorem, using two results which are proved in Subsections \ref{demothm1-2} and \ref{demothm2-2}.

\subsection{The algebra of Toeplitz operators}
\label{toeplitzop}

In this subsection, we describe the formalism of Toeplitz operators introduced by Berezin \cite{MR0395610} and Boutet de Monvel-Guillemin \cite{MR620794}, and developed by Bordemann-Meinrenken-Schlichenmaier \cite{MR1301849}, Schlichenmaier \cite{MR1805922} and Ma-Marinescu \cite{ma-marinescu}, \cite{MR2393271}.

We fix $m\in M$ for  this subsection, and we denote  $Y_m$ simply by $Y$.

Thus, we are given a complex manifold $Y$ of dimension $n_Y$, endowed with an Hermitian vector bundle $(\eta, h^\eta)|_Y$ and with a positive line bundle $(L,h^L)|_Y$. Recall that $R^L$ is the Chern curvature of $L$ and that 
\begin{equation}
\Theta^Y = \frac{\ic}{2\pi}R^L|_{T_{\R}Y\times T_{\R}Y},
\end{equation}
and $g^{T_\R Y}= \Theta^Y(\cdot,J\cdot)$ is the associated metric.

Let
\begin{equation}
\A = \smooth\big(Y, \End(\eta)\big),
\end{equation}
which we endow $\A$ with the $L^2$-metric induced by $g^{T_\R Y}$, $h^L$ and $h^\eta$.

For $p\in \N$ and $A\in \End( L^2(Y,\eta\otimes L^p))$, we will use the same notations as in Definition \ref{defnomreendo}, i.e.,  $||A||_\infty$ denotes the operator norm of $A$ and $||A||_1$ its trace norm (if $A$ is trace class).

Let $P_p$ be the orthogonal projection form $L^2(Y,\eta\otimes L^p)$ onto $H^0(Y,\eta\otimes L^p)$. By Riemann-Roch-Hirzebruch theorem and Kodaira vanishing theorem, we now that $\dim F_p \leq Cp^{n_Y}$, thus if $A\in \End( L^2(Y,\eta\otimes L^p))$ is such that $P_pAP_p=A$, we have
\begin{equation}
||A||_1\leq C ||A||_\infty p^{n_Y}.
\end{equation}

If $(V,h^V)$ is any finite dimensional Hermitian vector space and if $u\in \End(V)$, we denote by $\|u \|$ the operator norm of $u$.

For $f\in \A$, set
\begin{equation}
\|f\|_{\mathscr{C}^0} = \sup_{y\in Y} \|f(y)\|.
\end{equation}
This defines a metric on $\A$. 

For $f\in \A$, we denote by $T_{f,p}$ the \emph{Berezin-Toeplitz quantization} of $f$, that is
\begin{equation}
T_{f,p}=P_pfP_p.
\end{equation}
Observe that
\begin{equation}
\label{||Tfp||infinie}
||T_{f,p}||_\infty \leq\|f\|_{\mathscr{C}^0}.
\end{equation}

 Moreover, by \cite[(4.1.84), Lem. 7.2.4]{ma-marinescu}, as $p\to+\infty$, we have
\begin{equation}
\label{traceToeplitz}
\tr^{F_p}[T_{f,p}]=p^{n_Y} \int_Y \tr^\eta[f]e^{\Theta^Y} + O(p^{n_Y-1}).
\end{equation}

Recall that Toeplitz operators are defined in Definition \ref{def-op-de-Toepltiz}. As in \cite{ma-marinescu},  for a Toeplitz operator $T_p$ with corresponding sections $f_r$, we will use the notation
\begin{equation}
\label{notation-toeplitz}
T_p = \sum_{r=0}^{+\infty} p^{-r} T_{f_r,p} +O(p^{-\infty}).
\end{equation}
We denote by $\toep$ the space of Toeplitz operators on $Y$.

It follows from the above references that $\toep$ is an algebra. More precisely, it is proved in \cite[Thm. 0.3 Rem. 0.5]{MR2876259} that there are bidifferential operators  $C_r$ such that for $f,g\in \A$,
\begin{equation}
\begin{aligned}
\label{coeff-composition-Toep}
&T_{f,p}\circ T_{g,p} = \sum_{r=0}^{+\infty} p^{-r} T_{C_r(f,g),p} +O(p^{-\infty}),\\
&C_0(f,g)=fg.
\end{aligned}
\end{equation}
In particular,
\begin{equation}
\label{commutateurtoeplitz}
\begin{aligned}
&T_{f,p}\circ T_{g,p}=T_{fg,p}+O(p^{-1}),\\
&[T_{f,p}, T_{g,p}]=T_{[f,g],p}+O(p^{-1}),\\
&[T_{f,p}, T_{g,p}]_+=T_{[f,g]_+,p}+O(p^{-1}),
\end{aligned}
\end{equation}
where $[\,\cdot\,,\cdot\,]_+$ denotes the anti-commutator.

\subsection{Infinite dimensional bundles.}
\label{Infinitedimensionalbundles}

From now on, we will consider $\A$ and $\toep$ as infinite dimensional bundles of algebra on $M$: for $m\in M$,
\begin{equation}
\begin{aligned}
&\A_m = \smooth\big(Y_m, \End(\eta|_{Y_m})\big),\\
&\toep_m=\{\text{Toeplitz operators on the fiber }Y_m\}. 
\end{aligned}
\end{equation}
In particular, an element of $\toep$ define a family of elements of $\End(F_p)$, $p\in \N$. Moreover, $\|\cdot\|_{\mathscr{C}^0}$ defines a metric on the bundle $\A$, and $||\cdot||_\infty$ and $||\cdot||_1$ define two metrics on the bundle $\toep$. 

In the sequel, for any hermitian bundle $(\mathcal{V},h^{\mathcal{V}})$ on $M$, we will still denote by $||\cdot||_\infty$ and $||\cdot||_1$ the induced metrics on $\mathcal{V}\otimes \toep$.

We define a connection on $\A$ as follows: if $f\in \smooth(M,\A)=\smooth\big(N,\End(\eta)\big)$ and $U\in T_\R M$, then
\begin{equation}
\n^\A_U f= \n^\eta_{U^H} f,
\end{equation}
where $U^H$ is the horizontal lift of $U$ in $T^H_{M,\R}N$ (see \eqref{def-THN}).

Define also $\Fp$ as the infinite dimensional bundle:
\begin{equation}
\Fp{}_{,m} = \smooth\big(Y_m,(\eta\otimes L^p)|_{Y_m}\big).
\end{equation}
Then $F_p$ is a sub-bundle of $\Fp$ and $\Fp$ is endowed with the connection $\n^\Fp$ defined by
\begin{equation}
\n^\Fp_U s= \n^{\eta\otimes L^p}_{U^H} s,
\end{equation}
where $U^H$ is the horizontal lift of $U\in T_\R M$ in $T^H_{M,\R}N$.

Finally, $\A$ and $\Fp$ are equipped  with the $L^2$ metrics $h^\A$ and $h^\Fp$ associated to $g^{T_\R Y}$, $h^\eta$ and $h^L$. By Remark \ref{N->Mkahler} and  \cite[Thm. 1.5]{MR929147}, we know that $\n^\A$ and $\n^\Fp$ preserve the metrics $h^\A$ and $h^\Fp$. Furthermore, if  $\n^{F_p}$ is the Chern connection on $(F_p,h^{F_p})$, then by \eqref{defnablaEeq} and \eqref{nablaHetnablaEkahler}, we have
\begin{equation}
\n^{F_p} = P_p\n^\Fp P_p.
\end{equation}

Let $R^{F_p}$ be the curvature of $\n^{F_p}$. We denote again by $P_p$ the projection from $\Lambda^\bullet(T^*_\R M)\otimes \Fp$ onto $\Lambda^\bullet(T^*_\R M)\otimes F_p$. The following theorem of Ma-Zhang \cite[Thm 2.1]{MR2286586} is the cornerstone of our approach.
\begin{thm}
\label{RFpToeplitz}
Let $f\in\smooth\big(M,\A\big)$. The forms $\n^{F_p}T_{f,p}$ and $\frac{1}{p}R^{F_p}$  are Toeplitz operators valued  form, which means that there are $\varphi_r(f)\in \smooth\big( M, T^*_\R M\otimes \A\big) $ and $R_r\in \smooth \big( M, \Lambda^2(T^*_\R M)\otimes \A\big)$ such that
\begin{equation}
\begin{aligned}
&\n^{F_p}T_{f,p} = \sum_{r=0}^{+\infty} T_{\varphi_r(f),p} p^{-r} + O(p^{-\infty}), \\
&\frac{1}{p}R^{F_p} = \sum_{r=0}^{+\infty} T_{R_r,p} p^{-r} + O(p^{-\infty}).
\end{aligned}
\end{equation}
Moreover, For $U,V \in T_\R M$, we have
\begin{equation}
\begin{aligned}
&\varphi_0(f)(U)=\n^{\eta}_{U^H}f, \\
& R_0(U,V) = R^L(U^H,V^H).
\end{aligned}
\end{equation}
\end{thm}

Using the Lichnerowicz formula \eqref{Lichnerowiczeq} and Theorem \ref{RFpToeplitz}, we deduce that for $b\in B$,
\begin{equation}
\label{Bup2=Toeplitz}
B_{p,u}^2{}|_{X_b} \in \mathrm{Op}(X_b)\otimes \Lambda^\bullet(T^*_bB) \otimes \End\left(\Wedge(T^*X_b)\otimes\xi|_{X_b}\right)\otimes \C[p]\otimes \toep|_{X_b},
\end{equation}
where $\mathrm{Op}(X_b)$ is the algebra of scalar differential operators on $X_b$.

\subsection{Operators with Toeplitz coefficients.}
\label{operatorwithToeplitzcoeff}

In this section, we extend the results of \cite[Sects. 7.2-7.4]{ma-marinescu} to the case of Toeplitz operators with value in the algebra of bounded operator on a fixed Hilbert space. We use the notations of Sections \ref{toeplitzop} and \ref{Infinitedimensionalbundles}, and we work on a single fiber $Y_m$, which will be simply denoted by $Y$.

Let $(\mathcal{H},\langle\cdot ,\cdot\rangle_{\mathcal{H}})$ be a Hilbert space and $\mathcal{B(H)}$ the algebra of bounded operators on $\mathcal{H}$.  We denote again by $P_{p}$ the orthogonal projection 
\begin{equation}
P_{p}\otimes \Id_{\mathcal{H}} \colon L^2(Y, L^p\otimes\eta)\otimes \mathcal{H}=L^2(Y, L^p\otimes\eta\otimes \mathcal{H}) \to H^0(Y,L^p\otimes\eta)\otimes \mathcal{H},
\end{equation}
and for every smooth family $A(y) \in \End(\eta_y)\otimes \mathcal{B(H)}$, $y\in Y$, we can define the operator 
\begin{equation}
T_{A,p}=P_{p} A(\cdot) P_{p} \colon L^2(Y, L^p\otimes\eta\otimes \mathcal{H}) \to L^2(Y, L^p\otimes\eta\otimes \mathcal{H}).
\end{equation}

Here again, we denote by $\|\cdot \|_\infty$ the operator norm for bounded operators acting on the Hilbert space $L^2(Y, L^p\otimes\eta\otimes \mathcal{H})$.

We extend the definition of Toeplitz operators  to this situation: here again we call Toeplitz operator a family of operators $T_p\in\End( L^2(Y,L^p\otimes\eta\otimes \mathcal{H}))$ satisfying the two properties of Definition \ref{def-op-de-Toepltiz}, with $f_r \in \smooth\big(Y, \End(\eta)\otimes \mathcal{B(H)}\big)$.

The results of \cite[Sects. 7.2-7.4]{ma-marinescu} can be easily extended to the present situation, and the proofs of results below proceed as of the proofs of \cite{ma-marinescu}, replacing therein $\End(E_{x_0})$ by $\End(\eta_{y_0})\otimes \mathcal{B(H)}$ endowed with the operator norm. The important point is the we use the operator norm here, which has similar properties in finite and infinite dimensions. We will thus not give details of the proofs in the rest of this section.

\begin{lemme}
\label{toeplitztenseurhilbert1}
The operator $T_{A,p}$ has a smooth  Schwartz kernel 
\begin{equation}
T_{A,p}(y,y') \in (L^p\otimes\eta)_y\otimes (L^p\otimes\eta)^*_{y'}\otimes\mathcal{B(H)}
\end{equation}
  with respect to $dv_Y(y')$. 

For $\e>0$, $\ell,m\in \N$, there is  $C_{\ell,m,\e}>0$ such that for all $p\geq 1$ and $y,y'\in Y $ with $d(y,y')>\e$,
\begin{equation}
\label{awayfromdiag}
\left \| T_{A,p}(y , y') \right \|_{\mathscr{C}^m(Y\times Y)} \leq C'_{\ell,m,\e}p^{-\ell},
\end{equation}
where the $\mathscr{C}^m$-norm is induced by $\n^L$, $\n^\eta$, the usual derivation on $\mathcal{H}$ and $h^L$, $h^\eta$, $\| \cdot \|_{\mathcal{H}}$. 
\end{lemme}

Recall that $TY$ is endowed with the Hermitian structure induced by $R^L|_{TY\times TY}$. For $y_0 \in Y$, we choose $\{v_i\}_{i=1}^{n_Y}$ an orthonormal basis of $T_{y_0}Y$. Then $u_{2j-1}=\frac{1}{\sqrt{2}}(v_j+\ol{v}_j)$ and $u_{2j} =  \frac{\ic}{\sqrt{2}}(v_j-\ol{v}_j)$, $j=1,\dots,n_Y$, forms an orthonormal basis of $T_{\R,y_0}Y$, which gives use an isomorphism $T_{\R,y_0}Y\simeq \R^{2n_Y}$. We denote the dependence on the base point $y_0$ by adding a superscript $y_0$.
 
On $\R^{2n_Y}\simeq \C^{n_Y}$, we denote the coordinates by $(W_1,\dots,W_{2n_Y})$ or $(w_1,\dots,w_{n_Y})$, with $w_j = W_{2j-1}+\ic W_{2j}$. Let $\mathscr{P}$ be the operator on $L^2(\R^{2n_Y} )$ defined by its kernel with respect to $dW$:
 \begin{equation}
\mathscr{P}(W,W') = \frac{1}{(2\pi)^m} \exp\left(-\frac{1}{4}(|w|^2+|w'|^2-2w\cdot w')\right).
\end{equation}
Then $\mathscr{P}$ is the usual Bergman kernel on $\C^{n_Y}$.

 We fix $y_0\in Y$. As usually, for $\e>0$ small enough, we identify the geodesic ball $B^Y(y_0,4\e)$ with the ball $B^{T_{\R,y_0}Y}(0,4\e)$ in $T_{\R,y_0}Y$ via the exponential map. The various bundles appearing here on $B^{T_{\R,y_0}Y}(0,4\e)$ are trivialized by mean of orthonormal  frames at $y_0$ and of parallel transport for the corresponding connections along the rays $u\in [0,1]\mapsto uW$. Let $dv_{T_\R Y}$ be the volume form on $(T_{\R,y_0}Y,g^{T_{\R,y_0}Y})$, we denote by $\tau_{y_0}$ the function satisfying
\begin{equation}
dv_Y(W) = \tau_{y_0}(W) dv_{T_\R Y}(W), \quad \tau_{y_0}(0)=1.
\end{equation}

 Let $\mathrm{pr}_Y$  be the natural projection from the fiberwise product $T_\R Y\times_Y T_\R Y$ to $Y$. Consider an operator $\Xi_p\colon L^2(Y, L^p\otimes\eta)\otimes \mathcal{H} \to  L^2(Y, L^p\otimes\eta)\otimes \mathcal{H}$ which as a smooth kernel $\Xi_p(y,y')$ with respect to $dv_Y(y')$. Under our trivialization, this kernel induces a smooth section $\Xi_p^{y_0}(W,W')$ of $\mathrm{pr}_Y^*(\End{\eta})\otimes \mathcal{B(H)}$ over $\{(y,W,W') \,: \, |W|,|W'|\leq 4\e\}\subset T_\R Y\times_Y T_\R Y$.
 
  Let $Q_{r,y_0} \in \End(\eta_{y_0})\otimes \mathcal{B(H)}[W,W']$, $r\in \N$, be polynomials in $W,W'$ with values in $\End(\eta_{y_0})$ which depends smoothly on $y_0\in Y$.  In the sequel, we denote
 \begin{equation}
 \label{def-de-cong}
p^{-n_Y} \Xi_p^{y_0} (W,W') \cong \sum_{r=0}^k\big(Q_{r,y_0}\mathscr{P}\big)(\sqrt{p}W,\sqrt{p}W')p^{-\frac{r}{2}} + O(p^{-\frac{k+1}{2}})
\end{equation}
if there exist $0<\e'<4\e$ and $C_0>0$ such that for any $\ell \in \N$, there exist $C_{k,\ell},M>0$ such that for any $W,W' \in T_{\R,y_0}Y$, $|W|,|W'|<\e'$ and any $p$, we have 
\begin{multline}
 \label{def-de-cong-bis}
\left \|p^{-n_Y} \Xi_p^{y_0} (W,W')\tau_{y_0}^{1/2}(W)\tau_{y_0}^{1/2}(W')- \sum_{r=0}^k\big(Q_{r,y_0}\mathscr{P}\big)(\sqrt{p}W,\sqrt{p}W')p^{-\frac{r}{2}}\right\|_{\mathscr{C}^\ell(Y)} \\
\leq C_{k,\ell}p^{-\frac{k+1}{2}}(1+\sqrt{p}|W|+\sqrt{p}|W'|)^M e^{-\sqrt{C_0p}|W-W'|} + O(p^{-\infty}).
\end{multline}
Here, $\mathscr{C}^\ell(Y)$ denotes the $\mathscr{C}^\ell$-norm for the parameter $y_0\in Y$ induced by the operator norms on $\End(\eta_{y_0})$ and $\mathcal{B(H)}$, and by $O(p^{-\infty})$ we mean a term such that for any $\ell,\ell_1\in \N$, there exists $C_{\ell,\ell_1}>0$ such that its $\mathscr{C}^{\ell_1}$-norm is dominated by $C_{\ell,\ell_1} p^{-\ell}$.

Recall that by \cite[Lem. 7.2.3]{ma-marinescu}, there exist $J_{r,y_0} \in \End(\eta_{y_0})[W,W']$ polynomials in $W,W'$ with values in $\End(\eta_{y_0})$ with the same parity as $r$ and with $J_{0,y_0} = \Id_{\eta_{y_0}}$, such that
\begin{equation}
\label{expansionPpy0(W,W')}
p^{-n_Y} P_p^{y_0} (W,W') \cong \sum_{r=0}^k\big(J_{r,y_0}\mathscr{P}\big)(\sqrt{p}W,\sqrt{p}W')p^{-\frac{r}{2}} + O(p^{-\frac{k+1}{2}})
\end{equation}

\begin{lemme}
\label{toeplitztenseurhilbert2}
Let $A\in \smooth(Y,\End(\eta) \otimes \mathcal{B(H)})$. Then there exist a family of $\End(\eta_{y_0})\otimes \mathcal{B(H)}$-valued polynomials $\{ Q_{r,y_0}(A)\}_{r\in \N, y_0 \in Y}$ with the same parity as $r$ and smooth in $y_0\in Y$ such that for any $k\in \N$, $|Z|,|Z'| < \e/2$,
\begin{equation}
\label{expansionTA,p,y0}
p^{-n_Y} T_{A,p}^{y_0}(W,W') \cong \sum_{r=0}^k\big(Q_{r,y_0}(A)\mathscr{P}\big)(\sqrt{p}W,\sqrt{p}W')p^{-\frac{r}{2}} + O(p^{-\frac{k+1}{2}}),
\end{equation}
and moreover,
\begin{equation}
\label{Q0}
Q_{0,y_0}(A) = A(y_0).
\end{equation}
\end{lemme}

We now state the analogue of \cite[Thm 7.3.1]{ma-marinescu}, which gives a criterion for being a Toeplitz operator.
\begin{thm}
\label{critere-Toeplitz}
 Let $T_p \colon L^2(Y, L^p\otimes\eta\otimes \mathcal{H}) \to L^2(Y, L^p\otimes\eta\otimes \mathcal{H})$ be a family of bounded linear operators which satisfies the following three conditions:
 \begin{enumerate}[label=\emph{(\roman*)}]
 \item for any $p\in \N$, $P_pT_pP_p=T_p$;
 \item for any $\e_0>0$ and $\ell,m \in \N$, there exists $C_{\ell,m}>0$ such tat for all $p\geq 1$ and all $y,y' \in Y$ with $d(y,y')>\e_0$,
 \begin{equation}
\|T_p(y,y')\|_{\mathscr{C}^m(Y\times Y)}\leq C_{\ell,m}p^{-\ell};
\end{equation}
\item there exists a family of polynomial $\mathcal{Q}_{r,y_0}\in \End(\eta_{y_0})\otimes \mathcal{B(H)}[W,W']$ with the same parity as $r$ and depending smoothly in $y_0$  such that in the sense of \eqref{def-de-cong} and \eqref{def-de-cong-bis},
 \begin{equation}
p^{-n_Y} T_p^{y_0} (W,W') \cong \sum_{r=0}^k\big(\mathcal{Q}_{r,y_0}\mathscr{P}\big)(\sqrt{p}W,\sqrt{p}W')p^{-\frac{r}{2}} + O(p^{-\frac{k+1}{2}}).
\end{equation}
  \end{enumerate}
  Then $\{T_p\}_{p\geq 1}$ is a Toeplitz operator.
\end{thm}

The main result of this section (at least as far as the rest of this paper is concerned) provides an analogue of \eqref{coeff-composition-Toep}. It is proved in the same way as \cite[Thm. 7.4.1]{ma-marinescu}, using Lemma \ref{toeplitztenseurhilbert2} and Theorem \ref{critere-Toeplitz}.

\begin{thm}
\label{prodtoeplitztenseurhilbert}
For any $A,B\in \smooth (Y, \End(\eta) \otimes\mathcal{B(H)})$, the product of $T_{A,p}$ and $T_{B,p}$ is a Toeplitz operator. More precisely, there are bidifferential operators $C_r$ such that in the sense of \eqref{notation-toeplitz},
\begin{equation}
\label{prodtoeplitztenseurhilbert,eq}
T_{A,p}T_{B,p}=\sum_{r=0}^{+\infty} p^{-r} T_{C_r(A,B),p} +O(p^{-\infty}),
\end{equation}
and we have
\begin{equation}
\label{prodtoeplitztenseurhilbert,eq-2}
C_0(A,B) = AB.
\end{equation}
\end{thm}

\subsection{Localization.}
\label{localization2}

Fix  $b_0 \in B$. We use the same notations and trivializations that in Section \ref{localization}, except that we change therein $L^p$ by $F_p$, so that now
\begin{equation}
\begin{aligned}
\E_p&=\Lambda_{b_0}^\bullet(T^*_\R B)\otimes \left(\Wedge(T^*X)\otimes \xi \otimes F_p\right), \\
\E &=\Lambda_{b_0}^\bullet(T^*_\R B)\otimes \left(\Wedge(T^*X)\otimes \xi\right).
\end{aligned}
\end{equation}
Once again, we want to emphasize that the curtail difference with Section \ref{Sectionlinebundle} is that the dimension of $\E_p$ is not constant but grows to infinity. This is why we have to use  the operator norm on $\End(F_p)$ and Toeplitz operators (notably their boundedness and the properties of their derivatives).

We first prove that Lemma \ref{estimeeelliptique(p)} still holds in the present situation.
\begin{lemme}
\label{estimeeelliptique(p)2}
For any  $k\in \N$, there exists $C_k>0$  such that for any $p\geq 1$, $u>0$ and $s\in \sob^{2k+2}(X,\E_p)$,
\begin{equation}
||s||^2_{\sob^{2k+2}(p)}\leq C_k  p^{4k+4} \sum_{j=0}^{k+1}p^{-4j}||B_{p}^{2j}s||_{L^2}.
\end{equation}
\end{lemme}

\begin{proof}
As in the proof of Lemma \ref{estimeeelliptique(p)}, we work locally on one of the $U_{x_j}$'s and trivialize $\E_p$ in the way indicated at the beginning of Section \ref{localization}.

Let $\widetilde{e_i}(Z)$ be the parallel transport of $e_i$ with respect to $\nabla^{T_\R X}$ along the curve $t \in [0,1]\mapsto tZ$. Let $\Gamma^\xi$, $\Gamma^{F_p}$ and $\Gamma^{\Wedge,LC}$ be the connection form of $\n^\xi$, $\n^{F_p}$ and $\n^{\Wedge,LC}$ with respect to any fixed frame for $\xi$, $F_p$ and $\Wedge(T^*X)$ which is parallel along the curve $t\in [0,1]\mapsto tZ$ under the trivialization on $U_{x_k}$.

 Then 
\begin{multline}
\label{nupaveclesGamma2}
\n^p_{1,\widetilde{e_i}} = \n_{\widetilde{e_i}}+(\Gamma^{\Wedge,LC}+\Gamma^\xi+\Gamma^{F_p})(\widetilde{e_i})+\frac{1}{\sqrt{2}} S(\widetilde{e_i},\widetilde{e_j},f_\alpha)c(\widetilde{e_j})f^\alpha \\
+ \frac{1}{2} S(\widetilde{e_i},f_\alpha,f_\beta) f^\alpha f^\beta +\frac{1}{2}\Big(i_{\widetilde{e_i}} \left( \db^M-\partial^M\right)i\omega\Big)^c.
\end{multline}

Moreover, we know that the Lie derivative $\lie_Z \Gamma^{F_p}$ of $\Gamma^{F_p}$ is given by $\lie_Z \Gamma^{F_p}=i_ZR^{F_p}$ (see \cite[(1.2.32)]{ma-marinescu} for instance). Similarly, $\lie_Z \Gamma^L=i_ZR^L$. This, together with Theorem \ref{RFpToeplitz}, implies that $\Gamma^{F_p}$ is a Toeplitz operator and that there is a $\Gamma \in \smooth(Z_{b_0},T^*_\R N\otimes \C )$ such that
\begin{equation}
\label{devlpmtGammaFp}
\Gamma^{F_p}(U) = pT_{\Gamma(U^H),p} +O(1).
\end{equation}
 Hence, \eqref{nupaveclesGamma2} become
\begin{multline}
\label{nupaveclesGamma3}
\n^p_{u,e_i} = \n_{e_i}+\Gamma^{\Wedge,LC}+\Gamma^\xi+pT_{\Gamma(e_i^H),p}+\frac{1}{\sqrt{2u}} S_{i,j,\alpha}c(e_j)f^\alpha + \frac{1}{2u} S_{i,\alpha,\beta}f^\alpha f^\beta\\
+\frac{1}{2}\psi_{1/\sqrt{u}}\Big(i_{e_i} \left( \db^M-\partial^M\right)i\omega\Big)^c\psi_{\sqrt{u}} +O(1).
\end{multline}

We now prove that $B_{p}^2$ has a similar structure as in \eqref{Bup2=f(DX2,u,p)}. By \eqref{||Tfp||infinie}, we know that for $s\in \sob^{1}(U_{x_j},\E_{p,x_j})$, 
\begin{equation}
\label{normeSobetToep}
||T_{\Gamma(e_i^H),p}s||_{L^2}\leq C ||s||_{L^2} \quad \text{and} \qquad ||T_{\Gamma^H,p}\n_U s||_{L^2}\leq C ||s||_{\sob^{1}(p)}.
\end{equation}
Moreover, using \eqref{commutateurtoeplitz}, Theorem \ref{RFpToeplitz} and \eqref{devlpmtGammaFp}, we find
\begin{equation}
\label{[nablaU,TGamma]}
\n_U T_{\Gamma(e_i^H),p}= T_{\Gamma(e_i^H),p}\n_U +T_{U^H(\Gamma(e_i^H)),p} +O_0(1),
\end{equation}
where  $O_0(1)$ denotes a bounded family of operators of degree 0 acting on $F_p$. As a consequence, we have for $s\in \sob^{1}(U_{x_j},\E_{p,x_0})$, 
\begin{equation}
\label{normeSobetToep2}
 ||\n_UT_{\Gamma^H,p} s||_{L^2}\leq C ||s||_{\sob^{1}(p)}.
\end{equation}

Let $D^X=\db^X + \db^{X,*}$ be the Dirac operator on $ \Wedge(T^*X)\otimes \xi$. Using \eqref{Lichnerowiczeq}, \cite[Thm. 1.4.7]{ma-marinescu}, \eqref{nupaveclesGamma3}, \eqref{normeSobetToep} and \eqref{normeSobetToep2}, we find as in \eqref{Bup2=f(DX2,u,p)}:
\begin{equation}
\label{Bup2=f(DX2,u,p)2}
B_{p}^2 =  D^{X,2} + R + p \0_{p,1} + p\0_{p,0}^1+ p^2\0_{p,0}^2
\end{equation}
where $R$ is a  differential operators acting on $\Lambda_{b_0}^\bullet(T^*_\R B)\otimes \left(\Wedge(T^*X)\otimes \xi \right)_{x_j}$, and $\0_{p,1}$, $\0_{p,0}^1$ and $\0_{p,0}^2$ are  differential operators acting on $\E_{p,x_j}$ such that there is  $C>0$  such that for $p\geq 1$ and $s\in \sob^{k+1}(U_{x_j},\E_{p,x_0})$:
\begin{equation}
\begin{aligned}
\label{estimations0pij}
& || \0_{p,1}s||_{\sob^{k}(p)}\leq C ||s||_{\sob^{k+1}(p)}, \\
& || \0_{p,0}^is||_{\sob^{k}(p)}\leq C ||s||_{\sob^{k}(p)}, \quad i=1,2.
\end{aligned}
\end{equation}

The proof of Lemma \ref{nupaveclesGamma2} follows from \eqref{Bup2=f(DX2,u,p)2} and \eqref{estimations0pij} exactly in the same way as Lemma \ref{estimeeelliptique(p)} follows from \eqref{Bup2=f(DX2,u,p)}.
\end{proof}

Now, we want to prove an analogue of Proposition \ref{lepbestlocal}. The main ingredient in the proof of this proposition is the spectral gap of the Dirac operator. Thus, we begin with the following lemma. Recall that $D_p$ is the Dirac operator on $ \Wedge(T^*X)\otimes \xi\otimes F_p$.
\begin{lemme}
\label{spdeDp22}
There exist $C_0,C_L>0$ and $\mu_0>0$ such that 
\begin{equation}
\label{spdeDp22eq}
\Sp(D_p^2)\subset \{0\} \cup ]C_0 p-C_L,+\infty[.
\end{equation}
\end{lemme}

\begin{proof}
As done in \cite[Cor. 1.4.17]{ma-marinescu}, we can apply Nakano's inequality to the bundle $F_p\otimes \det(TX)^*$ and obtain that for $s\in \Omega^{(0,\bullet)}(X,F_p)$,
\begin{equation}
\label{Nakanotilde}
\frac{3}{2}\langle D_p^2s,s\rangle \geq \langle  R^{F_p\otimes \det(TX)^*}(w_j,\bw_k)\bw^k\wedge i_{\bw_j}s,s\rangle -C||s||^2_{L^2}.
\end{equation}
Here $C$ is independent on $p$ as is comes from the norm of the so-called \emph{Hermitian torsion operator} of $X$ (see \cite[(1.4.10)]{ma-marinescu}). From Theorem \ref{RFpToeplitz} and \eqref{||Tfp||infinie}, \eqref{Nakanotilde} we deduce that here are $C_0,C_L>0$ such that for $p\geq 1$ and $s\in \Omega^{(0,>0)}(X,F_p)$,
\begin{equation}
\label{estimationDppardessous}
||D_ps||^2_{L^2} \geq (C_0p-C_L)||s||^2_{L^2}.
\end{equation}
Finally, if $s\in \Omega^{(0,0)}(X,F_p)$ satisfies $D_p^2s=\lambda s$ for some $\lambda\neq 0$, then $0\neq D_ps \in \Omega^{(0,1)}(X,F_p)$ is still an eigenvector of $D_p^2$ for the eigenvalue $\lambda$, hence $\lambda \geq C_0p-C_L$. The proof of Lemma \ref{spdeDp22} is completed.
\end{proof}

Recall that the functions $\F_u$, $\G_u$ and $\HH_u$ and their tilded versions have been defined in \eqref{defFuGuHu} and \eqref{defFuGuHuavectilde}. 

We still denote by $\pi$ the projection $\pi \colon X\times_B X\to B$ be the projection from the fiberwise product $X\times_B X$ to $B$. Then $\tilde{\G}_u(vB_{p}^2)(\cdot,\cdot)$ is a section of $ \E_p \boxtimes \E_p^*$ over $X\times_B X$. Let $\n^{\E_p}$ be the connection on $\E_p$ induced by $\n^{\Lambda^\bullet(T^*_\R B)}$, $\n^{\Wedge,LC}$, $\n^{F_p}$ and $\n^\xi$, and let $\nabla^{ \E_p \boxtimes \E_p^*}$ be the induced connection on $ \E_p \boxtimes \E_p^*$. In the same way, let $h^{\E_p}$ be the metric on $\E_p$  induced by $h^{\Lambda^\bullet(T^*_\R B)}$, $h^{\Wedge,LC}$, $h^L$ and $h^\xi$, and let $h^{ \E_p \boxtimes \E_p^*}$ be the induced metric on $ \E_p \boxtimes \E_p^*$. Note that this metric restricts to the operator norm on the bundle $\End(\E_p)$ over $M\simeq \{(b,x,x')\in X\times_B X \, : \, x=x'\}$.
 We can now prove the analogue of Proposition \ref{lepbestlocal}:
\begin{prop}
\label{lepbestlocal2}
For any $m\in \N$, $\e>0$, there exist  $C>0$  and $N\in \N$ such that for any $u> 0$ and any $p\in \N^*$,
\begin{equation}
\left\| \tilde{\G}_{\frac{u}{p}} \Big(\frac{u}{p}B_{p}^2\Big)(\cdot,\cdot) \right\|_{\mathscr{C}^m} \leq Cp^N  \exp \left( -\frac{\e^2p}{16u} \right).
\end{equation}
Where the $\mathscr{C}^m$-norm is induced by  $\nabla^{ \E_p \boxtimes \E_p^*}$ and $h^{ \E_p \boxtimes \E_p^*}$.
\end{prop}

\begin{proof}
This proposition follows from Lemmas \ref{estimeeelliptique(p)2} and \ref{spdeDp22} exactly as Proposition \ref{lepbestlocal} follows from Lemma  \ref{estimeeelliptique(p)} and  \eqref{spdeDp2}. The only difference is that here we decompose $B_p^2$ as
\begin{equation}
\label{Bp2etDp2-2}
\begin{aligned}
&B_{p}^2 = D_p^2+R_p,\\
&R_p \in  \Lambda^{\geq 1}(T^*_\R B) \otimes \mathrm{Op}^{\leq 1}_X\big(\Wedge(T^*X)\otimes \xi\big)\otimes \C[p] \otimes \toep,
\end{aligned}
\end{equation}
and thus to obtain the analogues of \eqref{estimationRpresolvanteDp2} and \eqref{estimationAresolvanteBp2}, we also have to use the fact that Toeplitz operators are uniformly bounded for the operator norm (see \eqref{||Tfp||infinie}).
\end{proof}

\begin{cor}
 \label{lepbestlocal-cor2}
For any $m\in \N$, $\e>0$, there exist  $C(u)>0$ a rational fraction in $\sqrt{u}$ and $N\in \N$ such that for any $u> 0$ and any $p\in \N^*$,
\begin{equation}
\left \| \psi_{1/\sqrt{p}}  \tilde{\G}_{\frac{u}{p}} (B_{p,u/p}^2)(\cdot,\cdot) \right\|_{\mathscr{C}^m} \leq C(u)p^N  \exp \left( -\frac{\e^2p}{16u} \right).
\end{equation}
\end{cor}

\subsection{Convergence of the heat kernel.}
\label{sectioncvce2}

Here, we get the analogue of the results of Sections \ref{rescalingBpu} and \ref{sectioncvce}, and we prove Theorem \ref{cvcenoyaudelachaleur2}. By comparison to Section \ref{Sectionlinebundle}, the difficulty is twofold. Firstly, as above in Section \ref{localization2}, we have to take into account the fact that the dimension of $F_p$ grows to infinity, which is done thanks to Toeplitz operators. Secondly, if we can prove the convergence of the heat kernel of the rescaled operator to the heat kernel of some asymptotic operators in the vein Section \ref{sectioncvce}, we can no longer compute the ``limiting" heat kernel explicitly. However, using the results of Section \ref{operatorwithToeplitzcoeff}, we can give the asymptotic of this heat kernel, which will enable us to conclude.

Fix $u>0$, $b_0 \in B$ and $x_0\in X_{b_0}$. We use the same notations and trivializations that in Section \ref{rescalingBpu}, changing therein $L^p$ by $F_p$, and thus $p\Gamma^L$ by $\Gamma^{F_p}$. We get a connexion 
\begin{equation}
\label{nablax02}
\n^{\E_{p,x_0}} = \n + \rho(|Z|/\e)\left( \Gamma^{F_p} +\Gamma_1\right),
\end{equation}
on the trivial bundle  
\begin{equation}
\E_{p,x_0}=\Lambda^\bullet(T^*_{\R,b_0} B)\otimes \left(\Wedge(T^*X)\otimes \xi \otimes F_p\right)_{x_0}
\end{equation}
over $T_{x_0}X$, as well as a Laplacian $\Delta^{\E_{p,x_0}}$.

Recall that $\{f_\alpha\}$ denotes a frame of $T_\R B$, with dual frame $\{f^\alpha\}$.  Let $\widetilde{e_i}(Z)$ be the parallel transport of $e_i$ with respect to $\n^{T_\R X_0,LC}$ along the curve $t\in [0,1]\mapsto tZ$. Then $\{\widetilde{e_i}\}_i$ is an orthonormal frame of $T_\R X_0$.

Set
\begin{multline}
\label{defPhi2}
\Phi = \frac{K^X}{8} + \frac{1}{4} c(\widetilde{e_i})c(\widetilde{e_j})L'^\xi(\widetilde{e_i},\widetilde{e_j})+\frac{1}{\sqrt{2}}c(\widetilde{e_i})f^\alpha L'^\xi(\widetilde{e_i},f_\alpha)+ \frac{f^\alpha f^\beta}{2}L'^\xi(f_\alpha,f_\beta) \\
-\Big( \db^M\partial^Mi\omega\Big)^c  - \frac{1}{16}\Big \| \left( \db^X-\partial^X\right)i\omega^X \Big\|^2_{\Lambda^\bullet(T^*_\R X)}
\end{multline}
and
\begin{multline}
\label{Mpu2}
M_{p,x_0}  =  \frac{1}{2}\Delta^{\E_{p,x_0}} + \rho(|Z|/\e)  \Phi \\
+\rho(|Z|/\e) \left ( \frac{1}{4} c(\widetilde{e_i})c(\widetilde{e_j})R^{F_p}(\widetilde{e_i},\widetilde{e_j})+\frac{1}{\sqrt{2}}c(\widetilde{e_i})f^\alpha R^{F_p}(\widetilde{e_i},f_\alpha)+ \frac{f^\alpha f^\beta}{2}R^{F_p}(f_\alpha,f_\beta)\right).
\end{multline}
Then $M_{p,x_0}$ is a second order elliptic differential operator acting on $\smooth\left(T_{\R,x_0}X,\E_{p,x_0}\right)$.  Moreover, if  $\B_{x_0}$ is the algebra:
\begin{equation}
\B_{x_0}=\mathrm{Op}(T_{\R,x_0} X)\otimes \Lambda^\bullet(T_{\R,b_0}^*B) \otimes \End\left(\Wedge(T_{x_0}^*X)\otimes\xi_{x_0}\right)\otimes \C(\sqrt{p}) \otimes \toep_{x_0},
\end{equation}
then Theorem \ref{RFpToeplitz},  $\{(M_{p})_{_Z}\}_{p\geq 1}$ is  in  $\B_{x_0}$. Finally, near 0, $\n^{\E_p,x_0}=\n^p$ and $M_{p,x_0}=B_{p}^2$.

\begin{rem}
Working on $\E_{p,x_0}$ amount to replace the fibration $Z\overset{Y}{\to}X$ by the trivial fibration $T_{\R,x_0}X\times Y \to T_{\R,x_0}X$. However, as pointed out earlier, we cannot substitute $\E_{p,x_0}$ here by some fixed $\E_{x_0}$ as in Section \ref{rescalingBpu}.

\end{rem}

Let $\exp(-B_{p}^2)(Z,Z')$ and $\exp(-M_{p,x_0})(Z,Z')$ be the smooth heat kernel of $B_{p}^2$ and $M_{p,x_0}$ with respect to $dv_{X_0}(Z')$.
\begin{lemme}
\label{noyauBupetMup2}
For any $m\in \N$, $\e>0$, there exist $C>0$  and $N\in \N$ such that for  any $p\in \N^*$,
\begin{equation}
\left\| \exp\Big(-\frac{u}{p}B_{p}^2\Big)(x_0,x_0) -  \exp\Big(-\frac{u}{p}M_{p,x_0}\Big)(0,0)\right\|_{\mathscr{C}^m(M)} \leq Cp^N \exp(-\frac{\e^2p}{16u}),
\end{equation}
where $\|\cdot\|_{\mathscr{C}^m(M)}$ denotes the $\mathscr{C}^m$-norm in the parameters $b_0\in B$ and $x_0\in X$ induced by $\n^{\End(\E_p)}$ and the operator norm $h^{\End(\E_p)}$.
\end{lemme}

\begin{proof}
 As explain in the proof of Lemma \ref{noyauBupetMup}, we can prove Lemma \ref{noyauBupetMup2} by proving analogs of Lemma \ref{estimeeelliptique(p)2} and Porposition \ref{lepbestlocal2} for $M_{p,x_0}$, and using the finite propagation speed of the wave equation.
\end{proof}

In the sequel, if $U\in T_\R M$, we denote by $U^H$ its lift to $T^H_{M,\R}N$. Moreover, the basis $\{f_\alpha\}$ of $T_\R B$ has already been identified with a basis of $T_\R^HM$, and \textbf{when we write $f_\alpha^H$ we mean the lift in $T^H_{M,\R}N$ of $f_\alpha$ wiewed as an element of $T_\R^HM$} (which is not necessarily the same as the lift of $f_\alpha \in T_\R B$ in $T^H_{B,\R}N$). If $e_{a_1},e_{a_2}$ are some vectors among the $e_i$ and the $f_\alpha$ we set
\begin{equation}
R^L_{a_1,a_2} = R^L(e_{a_1}^H,e_{a_2}^H).
\end{equation}
To simplify the notations, we also write $c^i$ for $c(e_i^H)$.

Similarly to what is done in \eqref{defnrescaled}, we define for $t=\frac{1}{\sqrt{p}}$, $s\in \smooth(T_{x_0}X, \E_{p,x_0})$ and $Z\in T_{x_0}X$:
\begin{equation}
\label{defLup}
\begin{aligned}
&(S_ts)(Z)  = s(Z/t), \\
&\n_t = tS_t^{-1} \kappa^{1/2} \nabla^{\E_{p,x_0}} \kappa^{-1/2} S_t, \\
&\LL_t=  t^2 S_t^{-1} \kappa^{1/2}  M_{p,x_0}  \kappa^{-1/2} S_t, \\
\end{aligned}
\end{equation}

Recall that $[\cdot\, , \cdot]_+$ is our notation for the anti-commutator. We define for $U\in T_{\R,x_0}X$:
\begin{equation}
\label{defulLup}
\begin{aligned}
& \ul{\n}_{t,U}= \nabla_U + T_{\frac{1}{2}R^L( Z^H, U^H),p}(x_0) , \\
&\ul{\LL}_{t} = - \frac{1}{2}\sum_i \left\{\n_{e_i}{}^2+[\n_{e_i},T_{\frac{1}{2}R^L( Z^H, e_i^H),p}(x_0)]_+ +T_{(\frac{1}{2}R^L( Z^H, e_i^H))^2,p}(x_0)\right\}\\
& \qquad \qquad+T_{ \frac{1}{4} c^ic^jR^L_{i,j}+\frac{1}{\sqrt{2}}c^if^\alpha R^L_{i,\alpha}+ \frac{f^\alpha f^\beta}{2}R^L_{\alpha,\beta},p}(x_0).
\end{aligned}
\end{equation}

\begin{prop}
\label{asympL2}
When $t\to 0$, we have the following asymptotic in $\B_{x_0}$
\begin{equation}
\label{asympL2eq}
\n_{t,e_i} = \ul{\n}_{t,e_i} + O(t) \text{ and} \quad \LL_{t} = \ul{\LL}_{t} +O(t).
\end{equation}
\end{prop}

\begin{proof}
First, by Theorem \ref{RFpToeplitz}, if $f$ is a smooth section of $\A$ over a compact subset of $M$, there is a $C>0$ such that
\begin{equation}
\label{derivee-Toep}
||\n^{F_p}T_{f,p}||_\infty\leq C.
\end{equation}

By \eqref{nablax02} and \eqref{defLup}, we have
\begin{equation}
\label{natiup2}
\n_{t,e_i}(Z) = \kappa^{1/2}\left(tZ\right)\left\{ \n_{e_i} + \rho(t|Z|/\e)\left( t\Gamma^{F_p}_{tZ}(e_i) +t\Gamma_{1,tZ}(e_i)\right)\right\}\kappa^{-1/2}\left(tZ\right).
\end{equation}
Moreover,  observe that in \eqref{dvlptGamma}, the term $O(|Z|^2)$ is given by the norm of the derivatives of the curvature, thus, by \eqref{dvlptGamma} and \eqref{derivee-Toep} and Theorem \ref{RFpToeplitz}, we know that 
\begin{equation}
t\Gamma^{F_p}_{tZ}(U) =  \frac{t^2}{2}R^{F_p}_{x_0}(Z,U) +O\left( t^3 \right) = T_{\frac{1}{2}R^{L}_{x_0}(Z^H,U^H),p}(x_0)+O(t^2).
\end{equation}
Hence, by \eqref{dvlptGammaup} and \eqref{natiup2}, and the fact that $\rho(0)=\kappa(0)=1$, we find the first asymptotic development of Proposition \ref{asympL2}.

As in \eqref{dvptLB} and  \eqref{dvlptAup}, we have
\begin{multline}
\label{dvlptLBetAup}
\LL_{t} = -g^{ij}(tZ)  \left(\n_{t,e_i}\n_{t,e_j} -t\n_ {t,\nabla ^{TX_0}_{e_i}e_j} \right) \\
 +t^2\rho(t|Z|/\e))\left \{ \kappa^{1/2}\left(\Phi+  \frac{1}{4} c(\widetilde{e_i})c(\widetilde{e_j})R^{F_p}(\widetilde{e_i},\widetilde{e_j})\right.\right. \\
 \left.\left.+\frac{1}{\sqrt{2}}c(\widetilde{e_i})f^\alpha R^{F_p}(\widetilde{e_i},f_\alpha)+ \frac{f^\alpha f^\beta}{2}R^{F_p}(f_\alpha,f_\beta)\right)\kappa^{-1/2} \right\}_{tZ}.
\end{multline}

From Theorem \ref{RFpToeplitz}, the first  development in \eqref{asympL2eq} and \eqref{dvlptLBetAup}, 
we find
\begin{equation}
\label{presqueasympL2}
\LL_{t} = - \frac{1}{2}\sum_i \left(\ul{\n}_{t,e_i}\right)^2+T_{ \frac{1}{4} c^ic^jR^L_{i,j}+\frac{1}{\sqrt{2}}c^if^\alpha R^L_{i,\alpha}+ \frac{f^\alpha f^\beta}{2}R^L_{\alpha,\beta},p}(x_0)+O(t).
\end{equation}
Using the first equation of \eqref{commutateurtoeplitz} and \eqref{presqueasympL2}, we get the second identity of \eqref{asympL2eq}. The proof of Proposition \ref{asympL2} is completed.
\end{proof}

The next step is to prove an analogue of Theorem \ref{cvcenoyauLup->noyauLuinfini}.

Let $e^{-\LL_{t}}(Z,Z')$, $e^{-\ul{\LL}_{t} }(Z,Z')$ be the smooth kernels of the operators $e^{-\LL_{t}}$, $e^{-\ul{\LL}_{t} }$ with respect to $dv_{TX}(Z')$. Let $\mathrm{pr}_X$  be the projection from the fiberwise product $T_\R X\times_X T_\R X$ to $X$, then these kernels are sections of $\mathrm{pr}_X^*\big(\End(\E_p)\big)$ over $T_\R X\times_X T_\R X$.

\begin{thm}
\label{cvcenoyauLup->noyauLuinfini2}
For $u>0$ fixed, there exists $C>0$ such that for $t>0$ and $Z,Z'\in T_{\R,x_0}X$ with $|Z|,|Z'|\leq 1$, we have the following estimates form the operator norm:
\begin{equation}
\Big \| \left(e^{-u\LL_{t}}-e^{-u\ul{\LL}_{t} }\right)(Z,Z') \Big \| \leq Ct^{1/(2n_X+1)}.
\end{equation}
\end{thm}

The proof of Theorem \ref{cvcenoyauLup->noyauLuinfini2} follows the same strategy as the proof of Theorem \ref{cvcenoyauLup->noyauLuinfini} in Section \ref{sectioncvce}. Here again, the difficulties coming from the fact that the dimension on $F_p$ tend to infinity are dealt with the properties of Toeplitz operators. 

Recall that we add a superscript $(0)$ to the objects introduced above to denote their part of degree 0 in $\Lambda^\bullet(T^*_{\R,b_0} B)$.

Let $||\cdot ||_{t,0}$ be the $L^2$ norm on $\smooth(T_{\R,x_0}X,\E_{p,x_0})$ induced by $h^{\Lambda^\bullet(T^*_\R B)}_{x_0}$, $h^{\Wedge}_{x_0}$, $h^{\xi}_{x_0}$, $h^{F_p}_{x_0}$ and the volume form $dv_{TX}(Z)$. For $s\in\smooth(X_0,\E_{p,x_0})$, $m\in \N$, and $p\in \N^*$, set
\begin{equation}
\label{defnorme(p,m)2}
\begin{aligned}
&||s||_{t,m}^2=\sum_{\ell\leq m} \sum_{i_1,\dots,i_\ell}  ||\n_{t,e_{i_1}}^{(0)}\cdots\n_{t,e_{i_\ell}}^{(0)}s||_{t,0}^2, \\
&||s||_{\ul{t},m}^2=\sum_{\ell\leq m} \sum_{i_1,\dots,i_\ell}  ||\ul{\n}_{t,e_{i_1}}\cdots\ul{\n}_{t,e_{i_\ell}}s||_{t,0}^2.
\end{aligned}
\end{equation}
We denote by $\sob^m_t$ the Sobolov space $\sob^m(X_0,\E_{p,x_0})$ endowed with the norm $||\cdot ||_{t,m}$, and by $\sob^{-1}_t$ the Sobolev space of order $-1$ endowed with the norm
\begin{equation}
||s||_{t,-1}=\sup_{s'\in\sob^1_p\setminus\{0\}} \frac{\langle s,s'\rangle_{p,0}}{||s'||_{t,1}}\, .
\end{equation}

Finally, if $A\in \LL(\sob^k_t, \sob^m_t)$, we denote by $||A||^{k,m}_t$ the operator norm of $A$ associated with $||\cdot||_{t,k}$ and $||\cdot||_{t,m}$.

 Let
\begin{equation}
\mathscr{R}_{t} = \LL_{t} -\LL_{t}^{(0)}. 
\end{equation}

\begin{prop}
\label{estimationLupprop2}
There exist constants $C_1,C_2,C_3>0$ such that for any $t>0$ and any $s,s'\in\smooth(X_0,\E_{p,x_0})$,
\begin{equation}
\label{estimationsLup2}
\begin{aligned}
&\langle \LL_{t}^{(0)} s, sÊ\rangle_{t,0} \geq C_1||s||_{t,1}^2 - C_2||s||_{t,0}^2, \\
& \left|\langle \LL_{t}^{(0)} s, s'Ê\rangle_{t,0}\right| \leq C_3 ||s||_{t,1} ||s'||_{t,1}, \\
& \left \| \mathscr{R}_{t} s\right \|_{t,0} \leq C_4 ||s||_{t,1}.
\end{aligned}
\end{equation}
\end{prop}

\begin{proof}
 By \eqref{dvlptLBetAup}, we have 
\begin{equation}
\langle \LL_{t}^{(0)} s, sÊ\rangle_{t,0} = \frac{1}{2}||\n_t^{(0)}s||_{t,0}^2 + \left\langle T_{\frac{1}{4} c^ic^jR^L_{i,j}(x_0),p}s,s\right\rangle_{t,0}+ O(t)||s||_{t,0}^2.
\end{equation}
Together with \eqref{||Tfp||infinie}, this gives the first two estimates of \eqref{estimationsLup2}. 

By \eqref{dvlptGammaup}, \eqref{natiup2} and \eqref{dvlptLBetAup}, we see that \eqref{diffnatietnati(0)} and \eqref{dvlpmt-Rt} are still true, hence the last estimate of \eqref{estimationsLup2} holds.
\end{proof}

We define a contour $\Gamma$  in $\C$ as in Figure \ref{contour-Gamma} in Section \ref{sectioncvce}, but using the $C_2$ of Theorem \ref{estimationLupprop2}.

\begin{prop}
\label{estimationsresoprop2}
There exist $C>0$, $a,b \in \N$ such that for any $t>0$ and any $\lambda \in \Gamma$, the resolvant $\left( \lambda-\LL_t\right)^{-1}$ exists and
\begin{equation}
\label{estimationsreso2}
\begin{aligned}
&\left\| \left( \lambda-\LL_t\right)^{-1} \right\|^{0,0}_t \leq C(1+|\lambda|^2)^a, \\
&\left\| \left( \lambda-\LL_t\right)^{-1} \right\|^{-1,1}_t \leq C(1+|\lambda|^2)^b.
\end{aligned}
\end{equation}
\end{prop}

\begin{proof}
Proposition \ref{estimationsresoprop2} follows from Proposition \ref{estimationLupprop2} exactly as Proposition \ref{estimationsresoprop} follows from Proposition \ref{estimationLupprop}.
\end{proof}

\begin{prop}
\label{estimationcommutateursprop2}
Take $m\in \N^*$. Then there exists a contant $C_m>0$ such that for any $t>0$,  $Q_1,\dots,Q_m \in \left\{  \n^{(0)}_{t,e_i},Z_i\right\}_{i=1}^{2n_X}$ and  $s,s'\in \smooth_c(T_{\R,x_0}X, \E_{p,x_0})$,
\begin{equation}
\label{estimationcommutateurs2}
\left| \Big\langle [Q_1,[Q_2,\dots[Q_m,\LL_{t}]\dots]]s,s'\Big\rangle_{t,0} \right| \leq C_m||s||_{t,1}||s'||_{t,1}.
\end{equation}
\end{prop}

\begin{proof}
This Proposition is proved in the same way as Proposition \ref{estimationcommutateursprop}.
\end{proof}

From Proposition \ref{estimationcommutateursprop2}, we can deduce the following result as done for Proposition \ref{norme(m,m+1)resolvante}.
\begin{prop}
\label{norme(m,m+1)resolvante2}
For any $t>0$, $\lambda \in \Gamma$ and $m\in \N$,
\begin{equation}
(\lambda-\LL_t)^{-1}\big( \sob^m_t\big) \subset \sob^{m+1}_t.
\end{equation}

Moreover, for any $\alpha \in \N^{2n_X}$, there exist $K\in \N$ and $C_{\alpha,m}>0$ such that for any $t>0$, $\lambda \in \Gamma$ and $s\in \smooth_c(T_{\R,x_0}X,\E_{p,x_0})$,
\begin{equation}
\left\| Z^\alpha (\lambda-\LL_t)^{-1}s\right\|_{t,m+1} \leq C_{\alpha,m}(1+|\lambda|^2)^K\sum_{\alpha'\leq \alpha} ||Z^{\alpha'}s||_{t,m}.
\end{equation}
\end{prop}

Let $e^{-\LL_t}(Z,Z')$ be the smooth kernel of the operator $e^{-\LL_t}$ with respect to $dv_{TX}(Z')$. Let $\mathrm{pr}_X \colon T_\R X\times_X T_\R X\to X$ be the projection from the fiberwise product $T_\R X\times_M T_\R X$ to $M$, then $e^{-\LL_t}(\cdot,\cdot)$ is a section of $\mathrm{pr}_X^*\left(\End\left( \E_p \right)\right)$ over $T_\R X\times_M T_\R X$. Let $\n^{\End(\E_p)}$ be the connection on the bundle $\End(\E_p)$ over $M$ induced by $\n^{\Lambda^\bullet(T^*_\R B)}$, $\n^{\Wedge,LC}$, $\n^\xi$ and $\n^{F_p}$, and let $\nabla^{\mathrm{pr}_X^*\End(\E_p)}$ be the induced connection on $\mathrm{pr}_X^*\End(\E_p)$. Then $\nabla^{\mathrm{pr}_X^*\End(\E_p)}$ and the operator norm on $\End(\E_p)$ induce naturally a $\mathscr{C}^m$-norm for the parameters $b_0\in B$ and $x_0\in X_{b_0}$.

\begin{thm}
\label{estimationderiveesnoyaudeLpu2}
For any $m,m'\in \N$, there is $C>0$ such that for any $t>0$, $Z,Z'\in T_{\R,x_0}X$ with $|Z|,|Z'|\leq 1$,
\begin{equation}
\label{estimationderiveesnoyaudeLpu:eq2}
\sup_{|\alpha|,|\alpha'|\leq m} \left\| \frac{\partial^{|\alpha|+|\alpha'|}}{\partial {Z}^{\alpha}\partial Z'^{\alpha'}} e^{-\LL_t}(Z,Z')\right\|_{\mathscr{C}^{m'}(M,\mathrm{pr}_X^*\End(\E_p))} \leq C,
\end{equation}
where $|\cdot|_{\mathscr{C}^{m'}(M,\mathrm{pr}_X^*\End(\E_p))}$ denotes the $\mathscr{C}^{m'}$ norm with respect to the parameters $b_0$ in a compact subset of $B$ and $x_0\in X_{b_0}$.
\end{thm}

\begin{proof}
For $m\in \N$ and $p\in \N^*$, let
\begin{equation}
\mathcal{Q}^m =\left\{  \n_{t,e_{i_1}}^{(0)} \cdots \n_{t,e_{i_j}}^{(0)}\right\}_{j\leq m}.
\end{equation}

As in the proof of Theorem \ref{estimationderiveesnoyaudeLpu} (see \cite[(1.6.48)-(1.6.52)]{ma-marinescu}), it follows from Proposition \ref{norme(m,m+1)resolvante2} that there exists $C_m>0$ such that for $p\in \N^*$ and $Q,Q'\in \mathcal{Q}^m$,
\begin{equation}
\label{Qe(Lup)Q'2}
\left\| Qe^{-\LL_t}Q' \right \|_t^{0,0} \leq C_m.
\end{equation}

Here, we \emph{a priori} cannot conclude with a Sobolev inequality for a fixed Sobolev norm as in  the proof of Theorem \ref{estimationderiveesnoyaudeLpu}, because the space is changing. However, we will show a uniformity result in the Sobolev inequality for the ``standard" Soboloev norm.
\begin{lemme}
\label{Sobolevuniforme}
For every $d\in \N^*$, we endow $M_d(\C)$ (the space of $d\times d$ matrices with coefficients in $\C$) with the operator norm $\|\cdot \|$. This induces a Sobolev norm on $\smooth_c \big(\R^{N},M_d(\C)\big)$. We denote the corresponding Sobolev space by $H^k\big(\R^{N},M_d(\C)\big)$. 

Then for every $k, \ell \in \N$ such that $k-\ell >N/2$, there exists $C_{k,\ell,N}>0$ such that for every $d\in \N^*$ and $\varphi\in H^k\big(\R^{N},M_d(\C)\big)$,
\begin{equation}
\varphi \text{ is } \mathscr{C}^\ell  \text{ and }  \|\varphi\|_{\mathscr{C}^\ell} \leq C_{k,\ell,N} ||\varphi||_k,
\end{equation}
where  $\|\cdot\|_{\mathscr{C}^\ell}$ denotes the $\mathscr{C}^\ell$-norm on $\smooth_c\big(\R^{N},M_d(\C)\big)$.
\end{lemme}

\begin{proof}
 Suppose first that $\ell = 0$. For $\varphi\in \smooth_c\big(\R^{N},M_d(\C)\big)$, we denote by $\widehat{\varphi}$ the Fourier transform of $\varphi$. By the Fourier inversion formula, to show that $\varphi$ is continuous, it suffices to prove that $ \widehat{\varphi}(\xi) $ is in $L^1\big(\R^N, M_d(\C)\big)$. 

Set $\langle \xi \rangle = \left(1+|\xi|^2\right)^{1/2}$. Then $\varphi \in H^k\big(\R^{N},M_d(\C)\big)$ if and only if $\langle \xi \rangle^k \widehat{\varphi} \in L^2\big(\R^{N},M_d(\C)\big)$. Moreover, there exists $c_{k,N}>0$ independent of $d$ such that for $\varphi \in H^k\big(\R^{N},M_d(\C)\big)$,
\begin{equation}
\frac{1}{c_{k,N}}\| \langle \xi \rangle^k \widehat{\varphi} \|_{L^2} \leq \| \varphi \|_k \leq c_{k,N} \| \langle \xi \rangle^k \widehat{\varphi} \|_{L^2}.
\end{equation}
Now, we use Cauchy-Scwarz inequality:
\begin{equation}
\begin{aligned} 
\int \|\widehat{\varphi}(\xi) \| d\xi &\leq   \int \|\langle \xi \rangle^k \widehat{\varphi}(\xi)\|\times \|\langle \xi \rangle^{-k} \Id\| d\xi \\
& \leq \| \langle \xi \rangle^k \widehat{\varphi} \|_{L^2} \int \langle \xi \rangle^{-2k}  d\xi \quad \leq C_{k,0,N} \| \varphi \|_k.
\end{aligned}
\end{equation}

The case $\ell \geq 1$ follows from the case $\ell=0$ applied to the derivatives of $\varphi$.
\end{proof}

We can now finish the proof of Theorem \ref{estimationderiveesnoyaudeLpu2}, applying Lemma \ref{Sobolevuniforme} to our situation. Let $m\in \N$, as $e^{-\LL_t}(\cdot \,, \cdot) \in \smooth \big((T_{\R,x_0}X)^2, \End(\E_p)\big)$, there is  $k\in \N$ and a constant $C>0$ independent on $p$ such that for $|\alpha|,|\alpha'|\leq m$ and $|Z|,|Z'|\leq 1$,
\begin{equation}
\label{onutiliseSobolevunif}
\left\| \frac{\partial^{|\alpha|+|\alpha'|}}{\partial {Z}^{\alpha}\partial Z'^{\alpha'}} e^{-\LL_t}(Z,Z')\right\| \leq C \left \|e^{-\LL_t}(\cdot \,,\cdot)|_{B(0,1)^2}\right \|_k.
\end{equation}
Now, by \eqref{natiup2} and \eqref{defnorme(p,m)2}, for any $m\in \N$ there exists $C'_m>0$ independent on $t$ such that for $\varphi \in \smooth \big((T_{\R,x_0}X)^2, \End(\E_p)\big)$ with support in $B^{T_{\R,x_0}}(0,1)^2$,
\begin{equation}
\label{equivalenceSob2}
\frac{1}{C'_m} \|\varphi\|_{t,m} \leq \|\varphi\|_m \leq C'_m \|\varphi\|_{t,m}.
\end{equation}
With \eqref{Qe(Lup)Q'2}, \eqref{onutiliseSobolevunif} and \eqref{equivalenceSob2}, we see that \eqref{estimationderiveesnoyaudeLpu:eq} holds when $m'=0$.

For $m'\geq 1$, we use the same arguments as in Theorem \ref{estimationderiveesnoyaudeLpu} (see \cite[(1.6.55)]{ma-marinescu}).
\end{proof}

\begin{thm}
\label{diffdesresolvantesthm2}
There are constants $C>0$ and $M\in \N^*$ such that for $t >0$, 
\begin{equation}
\label{diffdesresolvantes2}
\left\| \big((\lambda-\LL_t)^{-1} - (\lambda-\ul{\LL}_t )^{-1}\big)s \right\|_{t,0} \leq Ct (1+|\lambda|^2)^M\sum_{|\alpha|\leq3}||Z^\alpha s||_{t,0}.
\end{equation}
\end{thm}

\begin{proof}

From \eqref{natiup2} and \eqref{defnorme(p,m)2}, for $p \geq 1$ and $m\in \N$ we find
\begin{equation}
\|s\|_{t,m} \leq C\sum_{|\alpha|\leq m} \|Z^\alpha s\|_{\ul{t},m}.
\end{equation}
Moreover, for $s,s'$ with compact support, using Theorem \ref{RFpToeplitz} and a Taylor expansion of \eqref{dvlptLBetAup}, we find
\begin{equation}
\label{Lup-Luinfty2}
\begin{aligned}
&\left | \big\langle (\LL_t-\ul{\LL}_t )s, s' \big\rangle_{t,0} \right | \leq Ct \|s'\|_{t,1} \sum_{|\alpha|\leq 3} \|Z^\alpha s \|_{\ul{t},1}, \\
&\| (\LL_t-\ul{\LL}_t )s \|_{t,-1} \leq C\sum_{|\alpha|\leq 3} \|Z^\alpha s \|_{\ul{t},1}.
\end{aligned}
\end{equation}

Note that
\begin{equation}
\label{Lup-Luinfty,resolvantes2}
(\lambda-\LL_t)^{-1}-(\lambda-\ul{\LL}_t )^{-1} = (\lambda-\LL_t)^{-1}(\LL_t-\ul{\LL}_t )(\lambda-\ul{\LL}_t )^{-1}.
\end{equation}
Moreover, Propositions \ref{estimationsresoprop2}, \ref{estimationcommutateursprop2} and \ref{norme(m,m+1)resolvante2} still hold for the operator $\ul{\LL}_t $, the norms $\|\cdot\|_{\ul{t},m}$ and the family of test operators for commutators $\left\{  \ul{\n}_{t,e_i},Z_i\right\}_{i=1}^{2n_X}$. Thus, Proposition \ref{norme(m,m+1)resolvante2}, \eqref{Lup-Luinfty2} and \eqref{Lup-Luinfty,resolvantes2} yields to \eqref{diffdesresolvantes2}.
\end{proof}

\begin{proof}[Proof of Theorem \ref{cvcenoyauLup->noyauLuinfini2}]
 By Theorems \ref{estimationderiveesnoyaudeLpu2} and \ref{diffdesresolvantesthm2}, we can prove Theorem \ref{cvcenoyauLup->noyauLuinfini2} exactly as Theorem \ref{cvcenoyauLup->noyauLuinfini}.
\end{proof}

Define 
\begin{equation}
\label{def-ul(L)-t,u}
\ul{\LL}_{t,u} = u\psi_{1/\sqrt{u}} \ul{\LL_{t}}\psi_{\sqrt{u}}.
\end{equation}
Whereas in Section \ref{sectioncvce} we could use a closed formula for the heat kernel of $\LL_{0,u}
$ to derive Theorem \ref{cvcenoyaudelachaleur} from Theorem \ref{cvcenoyauLup->noyauLuinfini}, here we cannot compute $e^{-\ul{\LL}_{t,u} }(0,0)$ exactly to get the asymptotic of $ \psi_{1/\sqrt{p}}\exp(-B_{p,u/p}^2)(x_0,x_0)$. The difficulty is that here, the harmonic oscillator $\ul{\LL}_t $ has its coefficients in the non-commutative algebra $\toep_{x_0}$. However, by \eqref{commutateurtoeplitz}, the coefficients of $\ul{\LL}_t $ tends to commute increasingly, so we can expect to have at least a equivalent of $e^{-\ul{\LL}_{t,u} }(0,0)$.

 For $y\in Y_{x_0}$, we define the operator $\har_{x_0} (y)$ acting on the space
\begin{equation}
\label{C(V,...)}
\smooth \left(T_{\R,x_0}X,\Lambda^\bullet(T^*_{\R,b_0} B)\otimes \left(\Wedge(T^*X)\otimes \xi \right)_{x_0}\right)
\end{equation}
by
\begin{equation}
\label{defhary}
\begin{aligned}
\har_{x_0} (y) = - \frac{1}{2}\sum_i \Big(\nabla_{e_i} &+ \frac{1}{2}R^L_{(x_0,y)}( Z^H, e_i^H)\Big)^2 \\
&+ \frac{1}{4} c^ic^jR^L_{i,j}(x_0,y)+\frac{1}{\sqrt{2}}c^if^\alpha R^L_{i,\alpha}(x_0,y)+ \frac{f^\alpha f^\beta}{2}R^L_{\alpha,\beta}(x_0,y).
\end{aligned}
\end{equation}
Set also
\begin{equation}
\label{defharu}
\har_{x_0,u}(y) = u\psi_{1/\sqrt{u}} \har_{x_0} (y)\psi_{\sqrt{u}}.
\end{equation}

Then $y\mapsto \har_{x_0} (y)$ is a smooth function from $Y_{x_0}$ to the space of differential operators acting on the space given in \eqref{C(V,...)}. As a consequence, the family  $\{P_{p,x_0}\har_{x_0} (y)P_{p,x_0}\}_p$ is a family of differential operators that belongs to the algebra $\B_{x_0}$. Now, as $\n_{e_i}$ and $P_{p,x_0}$ commute, it is easy to see that for any $p\in \N^*$,
\begin{equation}
\label{ul(L)=toep}
\ul{\LL}_t = P_{p,x_0}\har_{x_0} (\cdot)P_{p,x_0}.
\end{equation}

We denote by  $e^{-\ul{\LL}_t }(Z,Z')$ and $e^{-\har_{x_0} (y)}(Z,Z')$  the smooth kernels of the operators $e^{-\ul{\LL}_t }$ and $e^{-\har_{x_0} (y)}$ with respect to $dv_{TX}(Z')$. Then for $Z,Z' \in T_{\R,x_0}X$,
\begin{equation}
\left\{y\mapsto e^{-\har_{x_0} (y)}(Z,Z')\right\} \in \smooth\Big(Y_{x_0},\Lambda^\bullet(T_{b_0}^*B) \otimes \End\big(\Wedge(T_{x_0}^*X)\otimes\xi_{x_0}\big)\Big).
\end{equation}

\begin{thm}
\label{chaleurettoeplitzcommutent}
For $u>0$ fixed and for all $Z,Z' \in T_{\R,x_0}X$ we have as $t\to 0$
\begin{equation}
e^{-u\ul{\LL}_t }(Z,Z') = T_{e^{-u\har_{x_0} (\cdot)}(Z,Z'),p}+o(1),
\end{equation}
where $o(1)$ denotes a term converging to 0 for the operator norm.
\end{thm}

\begin{proof}
For $\lambda \in \Gamma$ (see Figure \ref{contour-Gamma}), both $\lambda-P_{p,x_0}\har_{x_0} (y)P_{p,x_0}$ and $\lambda-\har_{x_0} (y)$ are invertible, so we can use a contour integral to get
\begin{multline}
\label{diffdesnoyauxdelachaleure}
e^{-uP_{p,x_0}\har_{x_0} (y)P_{p,x_0}} - P_{p,x_0}e^{-u\har_{x_0} (\cdot)}P_{p,x_0} = \\
\frac{1}{2i\pi}\int_\Gamma e^{-u\lambda}\left[ (\lambda-P_{p,x_0}\har_{x_0} P_{p,x_0})^{-1} -P_{p,x_0}(\lambda-\har_{x_0} )^{-1}P_{p,x_0}\right] d\lambda.
\end{multline}
Moreover, setting $P^\perp_{p,x_0} = 1-P_{p,x_0}$, we have 
\begin{align}
\label{diffdesinverses}
 (\lambda-P_{p,x_0}\har_{x_0} P_{p,x_0})^{-1} &-P_{p,x_0}(\lambda-\har_{x_0} )^{-1}P_{p,x_0} \notag \\
 &=(\lambda-P_{p,x_0}\har_{x_0} P_{p,x_0})^{-1}(P_{p,x_0}\har_{x_0} P_{p,x_0}-\har_{x_0} )P_{p,x_0}(\lambda-\har_{x_0} )^{-1}P_{p,x_0}  \\
 &=(\lambda-P_{p,x_0}\har_{x_0} P_{p,x_0})^{-1}P_{p,x_0}\har_{x_0} P^\perp_{p,x_0}(\lambda-\har_{x_0} )^{-1}P_{p,x_0}.\notag
\end{align}

By Propositions \ref{estimationsresoprop2} for $\ul{\LL}_t $, there are  constants $C>0$ and $a\in \N$ such that for $\lambda\in \Gamma$,
\begin{equation}
\label{normeinversetoeplitz}
\left \| (\lambda-P_{p,x_0}\har_{x_0} P_{p,x_0})^{-1} \right \|_t^{0,0} \leq C(1+|\lambda|^2)^a.
\end{equation}

Let $f \in \smooth (T_{\R,x_0}X\times Y_{x_0},\C)$. Note that $y\mapsto (\lambda-\har_{x_0} (y))^{-1}$ is a smooth function on $Y_{x_0}$ with values in the algebra  of bounded operator acting on the Hilbert space
\begin{equation}
L^2\big(T_{\R,x_0}X,\Lambda^\bullet(T^*_{\R,b_0} B)\otimes \left(\Wedge(T^*X)\otimes \xi \right)_{x_0}\big).
\end{equation}
Thus, we can apply Theorem \ref{prodtoeplitztenseurhilbert} to 
\begin{equation}
\begin{aligned}
& \mathcal{H} = L^2\big(T_{\R,x_0}X,\Lambda^\bullet(T^*_{\R,b_0} B)\otimes \left(\Wedge(T^*X)\otimes \xi \right)_{x_0}\big), \\
& A(y) = f(\cdot, y) , \quad B(y) = (\lambda-\har_{x_0} (y))^{-1}.
\end{aligned}
\end{equation}
We then get
\begin{equation}
P_{p,x_0}fP_{p,x_0}(\lambda-\har_{x_0} )^{-1}P_{p,x_0} - P_{p,x_0}f(\lambda-\har_{x_0} )^{-1}P_{p,x_0} = O(p^{-1}).
\end{equation}
Here, the term $O(p^{-1})$ depends of course on $\lambda$. To get the expansion \eqref{expansionTA,p,y0}, we used the Taylor expansion of $A$. Thus, in \eqref{expansionTA,p,y0}, we can bound the error term $O(p^{-\frac{k+1}{2}})$ using the derivatives of $A$ of order less than $k+1$. Applying this argument to $(\lambda-\har_{x_0} )^{-1}$ and using Proposition \ref{estimationsresoprop2}, we find that there exists $M\in \N^*$ such that
\begin{equation}
\label{[resolvante,Toepltiz]=0}
\left \|P_{p,x_0}fP_{p,x_0}(\lambda-\har_{x_0} )^{-1}P_{p,x_0} - P_{p,x_0}f(\lambda-\har_{x_0} )^{-1}P_{p,x_0} \right \|_t^{0,0} \leq Cp^{-1}(1+|\lambda|^2)^M.
\end{equation}
Hence, as $\n_{e_i}$ commutes with $P_{p,x_0}$, using \eqref{defhary} we find 
\begin{equation}
\label{normeestimee}
\left \| P_{p,x_0}\har_{x_0} P^\perp_{p,x_0}(\lambda-\har_{x_0} )^{-1}P_{p,x_0} \right \|_t^{0,0} \leq Cp^{-1}(1+|\lambda|^2)^M.
\end{equation}

With \eqref{diffdesnoyauxdelachaleure}, \eqref{diffdesinverses}, \eqref{normeinversetoeplitz} and \eqref{normeestimee} we infer that 
\begin{equation}
\label{norme[exp,toep]}
\left \| e^{-P_{p,x_0}\har_{x_0} (y)P_{p,x_0}} - P_{p,x_0}e^{-\har_{x_0} (\cdot)}P_{p,x_0}\right \|_t^{0,0} \leq Cp^{-1}.
\end{equation}

Note that $P_{p,x_0}e^{-\har_{x_0} (\cdot)}P_{p,x_0}$ satisfies a estimate analogous to \eqref{estimationderiveesnoyaudeLpu:eq2}. Indeed, we have
\begin{equation}
P_{p,x_0}e^{-\har_{x_0} (\cdot)}P_{p,x_0}(Z,Z')=P_{p,x_0}e^{-\har_{x_0} (\cdot)}(Z,Z')P_{p,x_0},
\end{equation}
and we can apply \eqref{estimationderiveesnoyaudeLpu:eq} to $\har_{x_0} (y)$ (which correspond for $y$ fixed to $\LL_{0}$ in Section \ref{rescalingBpu}) and \eqref{||Tfp||infinie} to conclude. Thus, by \eqref{estimationderiveesnoyaudeLpu:eq2} applied to $e^{-P_{p,x_0}\har_{x_0} (y)P_{p,x_0}}$ and $P_{p,x_0}e^{-\har_{x_0} (\cdot)}P_{p,x_0}$, and by \eqref{norme[exp,toep]}, we can apply the method of Theorem \ref{cvcenoyauLup->noyauLuinfini} to complete the proof of Theorem \ref{chaleurettoeplitzcommutent}.
\end{proof}

Using the analogue of Lemma \ref{noyauBupetMup2}, Theorems \ref{cvcenoyauLup->noyauLuinfini2} and \ref{chaleurettoeplitzcommutent}, and \eqref{BupetLinfini}   we get that  
\begin{equation}
\label{BupetLinfini2}
\psi_{1/\sqrt{p}} e^{-B_{p,u/p}^2}(x_0,x_0)=p^{n_X} T_{e^{-\har_{x_0,u} (\cdot)}(0,0),p} + o(p^{n_X})
\end{equation}
for the operator norm and the operator norm of the derivatives. 

Recall that $\dot{R}^{X,L}$ is define in \eqref{defRpoint2}, that $\{ w_j\}$ is an orthonormal frame of $(TX,h^{TX})$, with dual frame $\{w^j\}$ and that $\{f_\alpha\}$ is a frame of $T_\R B\simeq T^H_{B,\R}M$ with dual basis $\{f^\alpha\}$. Define
\begin{equation}
\label{defOmegau2}
\Omega_u =  uR^L(w_k^H,\bw_\ell^H)\bw^{\ell}\wedge i_{\bw_k} + \sqrt{\frac{u}{2}}c(e_i)f^\alpha R^L(e_i^H,f_\alpha^H)+ \frac{f^\alpha f^\beta}{2}R^L(f_\alpha^H,f_\beta^H).
\end{equation}
By comparing the definitions of $\har_{x_0,u} $ in \eqref{defhary} and \eqref{defharu} and of $\LL_{0,u}$ in \eqref{defnrescaled}, and using \eqref{noyaudelachaleurLinfini}, we find that
\begin{equation}
\label{toep-exp(har)}
T_{e^{-\har_{x_0,u} (\cdot)}(0,0),p} = (2\pi)^{-n_X} P_{p,x_0} \exp(-\Omega_{u,(x_0,\cdot)})\frac{\det(\dot{R}^{X,L}_{(x_0,\cdot)})}{\det \big(1-\exp(-u\dot{R}^{X,L}_{(x_0,\cdot)})\big)}\otimes \Id_{\xi}P_{p,x_0},
\end{equation}

Finally, we have proved that for any $k\in \N$, as $p\to +\infty$, uniformly as $u$ varies in a compact subset of $\R^*_+$, we have the following asymptotic for for the operator norm on $\End(\E_p)$ and the operator norm of the derivatives up to order $k$:
\begin{multline}
\psi_{1/\sqrt{p}}\exp(-B_{p,u/p}^2) (x_0,x_0)\\
= \frac{p^{n_X}}{(2\pi)^{n_X}} P_{p,x_0} e^{-\Omega_{u,(x_0,\cdot)}}\frac{\det(\dot{R}^{X,L}_{(x_0,\cdot)})}{\det \big(1-\exp(-u\dot{R}^{X,L}_{(x_0,\cdot)})\big)}\otimes \Id_{\xi_{x_0}}P_{p,x_0} + o(p^{n_X}).
\end{multline}

Theorem \ref{cvcenoyaudelachaleur2} is proved.

\subsection{Asymptotic of the torsion forms.}
\label{Asymptoticofthetorsionforms2}

The method is the same as in Section \ref{Asymptoticofthetorsionforms}. Let $b_0\in B$, we denote $X_{b_0}$ and $Z_{b_0}$ simply by $X$ and $Z$. Recall that $n_M=\dim M$.

Let $\Lambda  \in \smooth\left( Z , \pi_1^*\End\big(\Lambda^\bullet(T^*_{\R,b_0} B)\otimes\Wedge(T^*X)\big)\right)$ be defined by
\begin{equation}
\label{def-Lambda_u}
\Lambda_u(z) =e^{-\har_{u} (z)}(0,0) = (2\pi)^{-n_X} \exp(-\Omega_{u,z})\frac{\det(\dot{R}^{X,L}_{z})}{\det \big(\Id-\exp(-u\dot{R}^{X,L}_{z})\big)},
\end{equation}
and let $\RR_u \in \smooth (Z,\C)$ be defined by
\begin{equation}
\label{defRu2}
\RR_u(z)=  \tr_s\left[ N_u\Lambda_u(z)\right].
\end{equation}

Let $A_j\in \smooth\left(Z, \pi_1^*\End\big(\Lambda^\bullet(T^*_{\R,b_0} B)\otimes\Wedge(T^*X)\big)\right)$ be such that as $u\to 0$
\begin{equation}
\label{defAj2}
\Lambda_u (z) = \sum_{j=-n_M}^k A_j(z)u^j + O(u^{k+1}),
\end{equation}
and here again we set $A_{-n_M-1}=0$.

\begin{thm}
\label{dvlptuniformedunoyaudelachaleur2}
There exist $\{A_{p,j}\}\in \smooth(X,\End(\E_p))$ such that for any $k,\ell \in \N$, there exist $C>0$ such that for any $u\in ]0,1]$ and $p\geq 1$, 
\begin{equation}
\label{dvlptuniformedunoyaudelachaleureq1-2}
\left\|p^{-n_X}\psi_{1/\sqrt{p}} \exp\left( -B_{p,u/p}^2 \right)(x,x) - \sum_{j=-n_M}^k A_{p,j}(x)u^j\right\|_{\mathscr{C}^\ell(M)}\leq Cu^{k+1}.
\end{equation}
Moreover, as $p\to+\infty$, we have for any $j\geq-n_M$
\begin{equation}
\label{dvlptuniformedunoyaudelachaleureq2-2}
A_{p,j}(x)=P_{p,x}A_j(x,\cdot)\otimes \Id_{\xi_x} P_{p,x} + o(1),
\end{equation}
for the operator norm on $\End(\E_p)$ and the operator norm of the derivatives up to order $\ell$.
\end{thm}

Theorem \ref{dvlptuniformedunoyaudelachaleur2} will be proved in Section \ref{demothm1-2}.

For $j\geq -n_M-1$, set
\begin{equation}
\label{defAtildej2}
\tilde{A}_j(z) =  \tr_s\left[ N_VA_j(z)+i\omega^{H}A_{j+1}(z)\right].
\end{equation}
Then by \eqref{defNueq}, \eqref{defRu2} and \eqref{defAj2}, we have
\begin{equation}
\label{Atildej2}
\RR_u (z) = \sum_{j=-n_M-1}^k \tilde{A}_j(z)u^j + O(u^{k+1}).
\end{equation}
 Set also
 \begin{equation}
 \label{defBpj/Bj2}
 \begin{aligned}
 &B_{p,j}=\int_Z\tr_s\left[N_VA_{p,j}(z)+i\omega^{H}A_{p,j+1}(z)\right] \frac{\Theta^{Y,n_Y}}{n_Y!}dv_X, \\
 &B_j=\int_Z\tilde{A}_j(z) \frac{\Theta^{Y,n_Y}}{n_Y!}dv_X.
 \end{aligned}
 \end{equation}

\begin{rem}
The operator convergences in Theorems \ref{cvcenoyaudelachaleur2} and \ref{dvlptuniformedunoyaudelachaleur2} implies the convergence of the corresponding supertraces divided by $p^{\dim Z}$. Indeed, it is classical that $\dim F_p \leq Cp^{\dim Y}$ for some constant $C$, and thus for $D\in \End( \Lambda^\bullet(T_{\R,b}^* B)\otimes \Wedge(T^*X_b)\otimes  F_p )$ and $f\in \End( \Lambda^\bullet(T_{\R,b}^* B)\otimes \Wedge(T^*X_b))$, we know that 
\begin{equation}
\|D-p^{\dim X}T_{f,p}\| =o(p^{\dim X}) \: \Longrightarrow \: |p^{-\dim Z}\trs(D_1)-p^{-\dim Y}\trs(T_{f,p})|=o(1).
\end{equation}
Thus, we can conclude using \eqref{traceToeplitz}. In particular, we have the following result.
\end{rem}
 
 Recall that $n_Z=n_X+n_Y$.
 \begin{cor}
 \label{cordvlptuniforme2}
 For any $k,\ell\in \N$, there exists $C>0$ such that for any $u\in ]0,1]$ and $p\geq 1$, 
\begin{equation}
\left|p^{-n_Z}\psi_{1/\sqrt{p}}\tr_s \left[ N_{u/p}\exp\left( -B_{p,u/p}^2 \right)\right] - \sum_{j=-n_M-1}^k B_{p,j}u^j\right|_{\mathscr{C}^\ell(M)} \leq Cu^{k+1}.
\end{equation}
Moreover, as $p\to+\infty$, we have for any $j\geq-d-1$
\begin{equation}
B_{p,j}=\mathrm{rk}(\xi)\mathrm{rk}(\eta) B_j  + O\left( \frac{1}{\sqrt{p}}\right).
\end{equation}
 \end{cor}


\begin{thm}
\label{dominationugrand2}
There exists $C>0$ such that for $u\geq 1$ and $p\geq 1$,
\begin{equation}
\label{dominationeq2}
\left|p^{-n_Z} \psi_{1/\sqrt{p}} \tr_s \left[ N_{u/p} \exp\left( -B_{p,u/p}^2 \right) \right] \right|_{\mathscr{C}^\ell(B)}\leq \frac{C}{\sqrt{u}}.
\end{equation}
\end{thm}

Theorem \ref{dominationugrand2} will be proved in Section \ref{demothm2-2}.

Let $p_0$ be such that for all $p\geq p_0$, the direct images $R^\bullet\pi_{1*}(\eta\otimes L^p)$ is locally free, $R^i\pi_{1*}(\eta\otimes L^p)=0$ for $i>0$ and the direct images $R^\bullet \pi_{2*}(\xi\otimes F_p)$ and $R^\bullet \pi_{3*}(\pi^*_1\xi\otimes\eta\otimes L^p)$ are locally free. 

 As in Section \ref{Asymptoticofthetorsionforms}, we define for $p\geq p_0$
\begin{equation}
\tilde{\zeta}_p(s) = -\frac{p^{-n_Z}}{\Gamma(s)}\int_0^{+\infty} u^{s-1}\psi_{1/\sqrt{p}}\Phi\left\{ \trs \left[N_{u/p} \exp(-B_{p,u/p}^2) \right]   \right\}du.
\end{equation}
Then if $\zeta_p$ denotes the zeta function \eqref{defzeta} associated with $B_{p,u}$, we have
\begin{equation}
\label{zetaetzetatilde2}
 p^{-n_Z}\psi_{1/\sqrt{p}}\zeta'_p(0) =  \log(p)B_{p,0}+\tilde{\zeta}'_p(0).
\end{equation}

Let
\begin{equation}
\tilde{\zeta}(s)=-\frac{1}{\Gamma(s)}\int_0^{+\infty}\int_Z \RR_u(z)dv_Z(z)u^{s-1}du.
\end{equation}
As in Section \ref{Asymptoticofthetorsionforms}, by \eqref{traceToeplitz} and Theorem \ref{cvcenoyaudelachaleur2}, and by dominated convergence (justified by Corollary \ref{cordvlptuniforme2} and Theorem \ref{dominationugrand2})  we find that
\begin{equation}
\label{zetatildeptozetatilde2}
\tilde{\zeta}'_p(0) \limarrow{p}{+\infty} \mathrm{rk}(\xi)\mathrm{rk}(\eta)   \Phi\tilde{\zeta}'(0).
\end{equation}

Let $T_B^{H'}N \subset T_M^HN$ be the space obtained by lifting in $TN$ the subspace $T_B^H M$ of $TM$. In particular, $T_B^{H'}N$ is orthogonal to $TY$. Let $\{f'_\alpha\}$ be an orthonormal basis of $T_{B,\R}^{H'}N$ with dual basis $\{f'^\alpha\}$. Set
\begin{equation}
\label{defF^H2}
\FF^H = \exp \left( - f'^\alpha f'^\beta R^L(f'_\alpha,f'_\beta) \right).
\end{equation}
Repeating the computations done in the proof of Theorem \ref{expressionzetatilde'(0)thm} which yield to \eqref{liensaveclecasdupoint2} and \eqref{independancedeTHM}, we find here again that
\begin{equation}
\label{liensaveclecasdupoint2-2}
\begin{aligned}
&\tilde{A}_j = 0 \text{ for } j\leq-2, \\
& \RR_u - \frac{\tilde{A}_{-1}}{u}-\tilde{A}_0 = \left\{\RR_u^{\{*\}} - \frac{\tilde{A}^{\{*\}}_{-1}}{u}-\tilde{A}^{\{*\}}_0\right\} \FF^H.
\end{aligned}
\end{equation}
Thus, we have
\begin{equation}
\label{expressionzetatilde'(0)2}
\tilde{\zeta}'(0)=\frac{1}{2} \int_Z \det\left( \frac{\dot{R}^{X,L}}{2\pi}\right) \log \left[ \det\left( \frac{\dot{R}^{X,L}}{2\pi}\right) \right] \FF^H \frac{\Theta^{Y,n_Y}}{n_Y!} dv_X.
\end{equation}
Moreover, by \eqref{defF^H2}, we know that
\begin{equation}
\begin{aligned}
 &\det\left( \frac{\dot{R}^{X,L}}{2\pi}\right) \frac{\Theta^{Y,n_Y}}{n_Y!} dv_X = \frac{\Theta^{Z,n_Z}}{n_Z!}, \\
 & \Phi \FF^H e^{\Theta^Z}  =e^{\Theta^N}.
 \end{aligned}
\end{equation}
Thus, by Corollary \ref{cordvlptuniforme2},  \eqref{zetaetzetatilde2}, \eqref{zetatildeptozetatilde2},  \eqref{expressionzetatilde'(0)2} and as in \eqref{ccldumainthm},  we have as $p\to +\infty$
\begin{equation}
\begin{aligned}
\psi_{1/\sqrt{p}}\zeta'_p(0) &=  \log(p)p^{n_Z}B_{0}+p^{n_Z}\Phi\tilde{\zeta}'(0) +o(p^{n_Z}) \\
&=\frac{\mathrm{rk}(\xi)\mathrm{rk}(\eta)}{2}  \int_Z  \log \left[ \det\left( \frac{p\dot{R}^{X,L}}{2\pi}\right) \right] e^{\Theta^N+(p-1)\Theta^Z}  +o(p^{n_Z}).
\end{aligned}
\end{equation}
Theorem \ref{mainthm2} is proved.


\subsection{Proof of Theorem \ref{dvlptuniformedunoyaudelachaleur2}}
\label{demothm1-2}

First, we would like to point out that we cannot use the same method to prove Theorem \ref{dvlptuniformedunoyaudelachaleur}   and Theorem \ref{dvlptuniformedunoyaudelachaleur2}. Indeed, the point was to see $t$ as a parameter, in the same way as $x_0$, and to use the fact that the development of the heat kernel on a compact space acting on a \emph{fixed} bundle is smooth in the parameters. However, here we cannot fixe the bundle, so we have to reprove directly the uniform development of the heat kernel. The techniques in this section are inspired by \cite[Sect. 4.1]{ma-marinescu}.

 Let $\n$ be the usual derivation and let $\Delta^{T_{\R,x_0}X}$ be the usual Bochner Laplacian on $T_{\R,x_0}X$. Recall that $\rho$ is defined in  \eqref{def-rho}, and define
\begin{equation}
\label{def-LL_2t}
\LL_{2,t} = \rho(|Z|/\e ) \LL_t + \big( 1- \rho(|Z|/\e )\big) \Delta^{T_{\R,x_0}X}.
\end{equation}

Then using the fact that
\begin{equation}
\sup_{a \in \Gamma} \left|a^m\tilde{\G}_u(\sqrt{u}a)\right| \leq C_m\exp\Big(-\frac{\e^2}{16u}\Big),
\end{equation}
 as in Proposition \ref{lepbestlocal2} and Lemma \ref{noyauBupetMup2}, we find
\begin{equation}
\label{exp-L-et-L2}
\Big\| e^{-u\LL_t}(0,0) -  e^{-u\LL_{2,t}}(0,0)\Big\|_{\mathscr{C}^m(M)} \leq C \exp(-\frac{\e^2p}{32u}).
\end{equation}

For $v=\sqrt{u}$, set (with $S_v$ in \eqref{defnrescaled})
\begin{equation}
\label{def-LL_3t^v}
\begin{aligned}
&\LL_{3,t}^v = v^2 S_v^{-1}\LL_{2,t}S_v, \\
&\LL_{3,t}^0 = \Delta^{T_{\R,x_0}X}.
\end{aligned}
\end{equation}
Then as in \eqref{noyauxdeMupetdeLup}, we have
\begin{equation}
\label{exp-L2-et-L3}
e^{-u\LL_{2,t}}(0,0) = u^{-n_X}e^{-\LL_{3,t}^v}(0,0).
\end{equation}

We will use the usual Sobolev norm $\|\cdot\|_k$ (see Lemma \ref{Sobolevuniforme}) on $\smooth_c(\R^{2n_X}, \E_{p,x_0})$.

 Using the fact that uniformly in $t$ we have
\begin{equation}
\LL_{3,t}^v = \Delta^{T_{\R,x_0}X} +O(v),
\end{equation}
we can prove results analogous to Propositions \ref{estimationLupprop2} to \ref{norme(m,m+1)resolvante2}, replacing $\n_t^{(0)}$, $\LL_t$ and $\|\cdot\|_{t,k}$ by $\n$, $\LL_{3,t}^v$ and $\|\cdot\|_k$. {\bf In the rest of this section, we will use these propositions for $\LL_{3,t}^v$ without further notice.}

 For $k,q \in \N^*$, set
\begin{equation}
I_{k,r} = \Big\{ (\boldrm{k},\boldrm{r})=(k_i,r_i) \in (\N^*)^{j+1}\times (\N^*)^{j} \: : \: \sum_{i=0}^j k_i = k+j \, , \: \sum_{i=1}^j r_i = r \Big\}.
\end{equation}
For $(\boldrm{k},\boldrm{r}) \in I_{k,r}$, $\lambda \in \Gamma$ (see Figure \ref{contour-Gamma} in Section \ref{sectioncvce}), $t>0$ and $v\geq 0$ set
\begin{equation}
\label{def-Ark(lambda,t,v)}
A_\boldrm{r}^\boldrm{k}(\lambda,t,v) = (\lambda-\LL_{3,t}^v)^{-k_0}\derpar{^{r_1}\LL_{3,t}^v}{v^{r_1}}(\lambda-\LL_{3,t}^v)^{-k_1} \cdots \derpar{^{r_j}\LL_{3,t}^v}{v^{r_j}}(\lambda-\LL_{3,t}^v)^{-k_j}.
\end{equation}
Then there exist $a_\boldrm{r}^\boldrm{k} \in \R$ such that
\begin{equation}
\label{derivee-resolvante^m-avec-A_r^k}
\derpar{^{r}}{v^{r}}(\lambda-\LL_{3,t}^v)^{-k} = \sum_{(\boldrm{k},\boldrm{r}) \in I_{k,q}} a_\boldrm{r}^\boldrm{k}A_\boldrm{r}^\boldrm{k}(\lambda,t,v).
\end{equation}

For $\ell\in \N$, let $\mathcal{Q}^\ell$ be the set of operators 
\begin{equation}
\mathcal{Q}^\ell=\{\n_{e_{i_1}}\dots \n_{e_{i_j}}\}_{j\leq \ell}.
\end{equation}

\begin{thm}
\label{QA_k^rQ'}
 For any $\ell \in \N$, $k>2(\ell+r+1)$ and $(\boldrm{k},\boldrm{r}) \in I_{k,r}$, there are $C_m>0$ and $N\in \N$ such that for any $\lambda \in \Gamma$, $t>0$, $v\geq 0$ and $Q,Q'\in \mathcal{Q}^\ell$,
 \begin{equation}
 \label{QA_k^rQ'-eq}
\big\| QA_\boldrm{r}^\boldrm{k}(\lambda,t,v) Q's\big\|_{0} \leq C(1+|\lambda|)^N\sum_{|\beta|\leq 2r} \|Z^\beta s\|_0.
\end{equation}
\end{thm}

\begin{proof}
 First, note that as in the proof of Theorem \ref{estimationderiveesnoyaudeLpu} (see \cite[(1.6.49), (1.6.51)]{ma-marinescu}), Proposition \ref{norme(m,m+1)resolvante2}  leads to
 \begin{equation}
 \label{QA_k^rQ'-r=0}
\big\| Q(\lambda - \LL_{3,t}^v)^{-m}\big\|^{0,0} \leq C(1+|\lambda|)^N \, , \quad \big\| (\lambda - \LL_{3,t}^v)^{-m}Q'\big\|^{0,0} \leq C(1+|\lambda|)^N.
\end{equation}
With this estimate and Proposition \ref{estimationsresoprop2}, we get \eqref{QA_k^rQ'-eq} for $r=0$.

Assume now $r>0$. By \eqref{natiup2}, \eqref{dvlptLBetAup}, \eqref{def-LL_2t}, \eqref{def-LL_3t^v} and Theorem \ref{RFpToeplitz}, we know that $\derpar{^r}{v^r} \LL_{3,t}^v$ is a combination of
\begin{equation}
\label{dervivees-p/-v-possibles}
\begin{aligned}
& \Big(\derpar{^{r_1}}{v^{r_1}}a_{ij}(t,vZ)\Big)\Big( \derpar{^{r_2}}{v^{r_2}}\n_{3,t,e_i}^v \Big)\Big( \derpar{^{r_3}}{v^{r_3}}\n_{3,t,e_j}^v \Big), & \derpar{^{r_1}}{v^{r_1}}b(t,vZ), \\
& \Big(\derpar{^{r_1}}{v^{r_1}}c_i(t,vZ)\Big)\Big( \derpar{^{r_2}}{v^{r_2}}\n_{3,t,e_i}^v \Big), &\Big(\derpar{^{r_1}}{v^{r_1}}d(t,vZ)\Big)\Delta^{T_{\R,x_0}X},
\end{aligned}
\end{equation}
where $a_{ij}$, $b$, $c_i$ and $d$ are of the form $f(Z)g(tZ)$ with $f(Z)$ and $g(Z)$ and their derivatives in $Z$ uniformly bounded for $Z\in \R^{2n_X}$ (recall that for Toeplitz operators, we take the operator norm).

From this decomposition and Proposition \ref{norme(m,m+1)resolvante2}, we can prove Theorem \ref{QA_k^rQ'} using a similar reasoning as in \cite[Thm. 4.1.13]{ma-marinescu}: we write the derivatives in \eqref{dervivees-p/-v-possibles} in the form $f(vZ)g(tvZ)Z^\beta$ with $f(Z)$ and $g(Z)$ and their derivatives in $Z$  uniformly bounded for $Z\in \R^{2n_X}$ and then we move all the terms $Z^\beta$ of $A_\boldrm{r}^\boldrm{k}(\lambda,t,v) Q'$ to the right-hand side of the operator, using the commutator trick of \cite{ma-marinescu}, i.e., commuting only the factors $Z_j$ each at a time. Finally, we move all the terms $\n_{3,t,e_i}^v$ in $\derpar{^r}{v^r} \LL_{3,t}^v$ to the right-hand side and we obtain \eqref{QA_k^rQ'-eq} using  Proposition \ref{norme(m,m+1)resolvante2} for $\LL_{3,t}^v$.
 \end{proof}

\begin{thm}
\label{derivee(v)-derivee(0)}
 For any $r\geq 0$ and $k>0$, there exist $C>0$ and $N\in \N$ such that for $\lambda \in \Gamma$, $t>0$ and $v\geq 0$,
 \begin{equation}
 \label{derivee(v)-derivee(0)-eq}
\begin{aligned}
& \left\| \Big(\derpar{^{r}\LL_{3,t}^v}{v^{r}} - \derpar{^{r}\LL_{3,t}^v}{v^{r}}\Big|_{v=0} \Big)s\right\|_{-1} \leq Cv \sum_{|\alpha|\leq r+3} \|Z^\alpha s\|_{1}, \\
&\bigg\| \Big(\derpar{^{r}}{v^{r}}(\lambda-\LL_{3,t}^v)^{-k} -  \sum_{(\boldrm{k},\boldrm{r}) \in I_{k,r}} a_\boldrm{r}^\boldrm{k}A_\boldrm{r}^\boldrm{k}(\lambda,t,0)\Big)s\bigg\|_0 \leq Cv(1+|\lambda|)^N\sum_{|\alpha|\leq 4r+3} \|Z^\alpha s\|_0.
\end{aligned}
\end{equation}
\end{thm}

\begin{proof}
As in the proof of Theorem \ref{diffdesresolvantes-thm}, the first line of \eqref{derivee(v)-derivee(0)-eq} just follows from a Taylor expansion in $v$ of $\LL_{3,t}^v$ and the fact that this expansion is uniform in $t>0$. We also get an analogue of \eqref{diffdesresolvantes}:
\begin{equation}
\label{diff-resolvantes-v-et-0}
\left\| \big((\lambda-\LL_{3,t}^v)^{-1} - (\lambda-\LL_{3,t}^0)^{-1}\big)s \right\|_{0} \leq Cv (1+|\lambda|^2)^M\sum_{|\alpha|\leq3}||Z^\alpha s||_{0}.
\end{equation}
Moreover, using Propositions \ref{estimationsresoprop2} and \ref{norme(m,m+1)resolvante2}, and \eqref{diff-resolvantes-v-et-0}, we have for any $m\in \N^*$
\begin{equation}
\label{diff-resolvantes^m-v-et-0}
\begin{aligned}
&\left\| \big((\lambda-\LL_{3,t}^v)^{-m} - (\lambda-\LL_{3,t}^0)^{-m}\big)s \right\|_{0}\\
 &\qquad\quad= \Big\| \sum_{i=0}^{m-1}(\lambda-\LL_{3,t}^v)^{-i} \big((\lambda-\LL_{3,t}^v)^{-1} - (\lambda-\LL_{3,t}^0)^{-1}\big)(\lambda-\LL_{3,t}^0)^{-(m-i-1)} s\Big\| \\
 &\qquad\quad\leq Cv (1+|\lambda|^2)^M\sum_{|\alpha|\leq3}||Z^\alpha s||_{0}.
\end{aligned}
\end{equation}

For $(\boldrm{k},\boldrm{r}) \in I_{k,r}$, set $a_i = (\lambda-\LL_{3,t}^v)^{-k_i}$, $b_i=\derpar{^{r_i}\LL_{3,t}^v}{v^{r_i}}$, $a'_i = (\lambda-\LL_{3,t}^0)^{-k_i}$ and  $b'_i=\left.\derpar{^{r_i}\LL_{3,t}^v}{v^{r_i}}\right|_{v=0}$. Then
\begin{multline}
A_\boldrm{r}^\boldrm{k}(\lambda,t,v)-A_\boldrm{r}^\boldrm{k}(\lambda,t,0) = a_0 b_1a_1\cdots b_ja_j  -   a'_0 b'_1a'_1\cdots b'_ja'_j  \\
 = \sum_{i=1}^j a_0b_1 \cdots a_{i-1} (b_i-b'_i)a'_i \cdots b'_ja'_j + \sum_{i=0}^j a_0b_1 \cdots b_i(a_i-a'_i) b'_{i+1} \cdots b'_ja'_j.
\end{multline}
Using this and \eqref{derivee-resolvante^m-avec-A_r^k}, the first inequality of \eqref{derivee(v)-derivee(0)-eq} and \eqref{diff-resolvantes^m-v-et-0}, we find the second inequality of \eqref{derivee(v)-derivee(0)-eq}.
\end{proof}

\begin{thm}
\label{exp-L3-bornee}
 For any $\ell, \ell' , r\in \N$ and $q>0$, there is $C>0$ such that for $t>0$, $v\geq 0$ and $Z,Z'\in T_{\R,x_0}X$ with $|Z|,|Z'|\leq q$, we have
 \begin{equation}
\sup_{|\alpha|,|\alpha'|\leq \ell} \left\| \frac{\partial^{|\alpha|+|\alpha'|}}{\partial {Z}^{\alpha}\partial Z'^{\alpha'}} \derpar{^{r}}{v^{r}}e^{-\LL_{3,t}^v}(Z,Z')\right\|_{\mathscr{C}^{\ell'}(M,\mathrm{pr}_X^*\End(\E_p))} \leq C.
\end{equation}
\end{thm}

\begin{proof}
Using the integral representation
\begin{equation}
\label{integral-rep-exp(L3)}
\derpar{^{r}}{v^{r}}e^{-\LL_{3,t}^v} = \frac{(-1)^k(k-1)!}{2i\pi} \int_\Gamma e^{-\lambda} \derpar{^{r}}{v^{r}}(\lambda - \LL_{3,t}^v)^{-1}d\lambda,
\end{equation}
Theorem \ref{exp-L3-bornee} is proved from \eqref{derivee-resolvante^m-avec-A_r^k} and Theorem \ref{QA_k^rQ'} exactly as Theorem \ref{estimationderiveesnoyaudeLpu2} is proved from \eqref{Qe(Lup)Q'2}.
\end{proof}

For $k$ large enough, set
\begin{equation}
\label{def-Brtv}
\begin{aligned}
&\mathscr{B}_{r,t} = \frac{(-1)^k(k-1)!}{2i\pi r!} \int_\Gamma e^{-\lambda}  \sum_{(\boldrm{k},\boldrm{r}) \in I_{k,r}} a_\boldrm{r}^\boldrm{k}A_\boldrm{r}^\boldrm{k}(\lambda,t,0) d\lambda, \\
&\mathscr{B}_{r,t,v} = \frac{1}{r!} \derpar{^{r}}{v^{r}}e^{-\LL_{3,t}^v} - \mathscr{B}_{r,t}.
\end{aligned}
\end{equation}
Then $\mathscr{B}_{r,t}$ and $\mathscr{B}_{r,t,v}$ do not depend on the choice on $k$ large. We denote by $\mathscr{B}_{r,t}(Z,Z')$ (resp. $\mathscr{B}_{r,t,v}(Z,Z')$) the smooth kernerl of $\mathscr{B}_{r,t}$ (resp. $\mathscr{B}_{r,t,v}$) with respect to $dv_{TX}(Z')$.

\begin{thm}
\label{estimation-Brtv}
 For $r\in \N$ and $q>0$, there exists $C>0$ such that for $t>0$, $v\geq 0$ and $Z,Z'\in T_{\R,x_0}X$ with $|Z|,|Z'|\leq q$, we have
 \begin{equation}
 \label{estimation-Brtv-eq}
\big\| \mathscr{B}_{r,t,v}(Z,Z') \big\| \leq Cv^{1/(2n_X+1)}.
\end{equation}
\end{thm}

\begin{proof}
The proof is the same as the proof of Theorem \ref{cvcenoyauLup->noyauLuinfini}, using Theorem \ref{derivee(v)-derivee(0)} and \eqref{integral-rep-exp(L3)} instead of Theorem \ref{diffdesresolvantes-thm} and \eqref{e(-Lup)=integralledecontour} respectively.
\end{proof}

\begin{thm}
\label{expansion-exp-L3}
  For any $\ell, \ell' , k\in \N$ and $q>0$, there is $C>0$ such that for $t>0$, $v\geq 0$ and $Z,Z'\in T_{\R,x_0}X$ with $|Z|,|Z'|\leq q$, we have
\begin{equation}
\label{expansion-exp-L3-eq}
\sup_{|\alpha|,|\alpha'|\leq \ell} \left\| \frac{\partial^{|\alpha|+|\alpha'|}}{\partial {Z}^{\alpha}\partial Z'^{\alpha'}} \derpar{^{r}}{v^{r}}\Big(e^{-\LL_{3,t}^v}(Z,Z')-\sum_{r=0}^k\mathscr{B}_{r,t}v^r \Big)(Z,Z')\right\|_{\mathscr{C}^{\ell'}(M,\mathrm{pr}_X^*\End(\E_p))} \leq Cv^{k+1}.
\end{equation}
\end{thm}

\begin{proof}
 By \eqref{def-Brtv} and \eqref{estimation-Brtv-eq}, we have
 \begin{equation}
\frac{1}{r!} \left.  \derpar{^{r}}{v^{r}}e^{-\LL_{3,t}^v} \right|_{v=0} = \mathscr{B}_{r,t}.
\end{equation}
Now by Theorem \ref{exp-L3-bornee}, \eqref{def-Brtv} and the Taylor expansion
\begin{equation}
f(v)-\sum_{r=0}^k \frac{1}{r!}  \derpar{^{r}f}{v^{r}}(0)v^r = \frac{1}{k!}\int_0^v (v-v_0)^k \derpar{^{k+1}f}{v^{k+1}}(v_0)dv_0,
\end{equation}
we get \eqref{expansion-exp-L3-eq}.
\end{proof}

Now, by \eqref{exp-L2-et-L3} and the asymptotic expansion heat kernels (see \cite{berline-getzler-vergne} for instance), we know that $e^{-\LL_{3,t}^v}(0,0)$ has an asymptotic expansion as $v=\sqrt{u}\to 0$ in powers of $u$, so we have
\begin{equation}
\label{B_impaire=0}
\mathscr{B}_{2r+1,t}(0,0) = 0.
\end{equation}

Theorem \ref{expansion-exp-L3}, along with \eqref{exp-L-et-L2}, \eqref{exp-L2-et-L3} and \eqref{B_impaire=0}, yields to
\begin{equation}
\Big\| u^{n_X} e^{-u\LL_t}(0,0) - \sum_{r=0}^k \mathscr{B}_{2r,t}(0,0)u^r \Big\|_{\mathscr{C}^{\ell'}(M,\mathrm{pr}_X^*\End(\E_p))} \leq Cu^{k+1}.
\end{equation}
Thus, by the analogue of \eqref{BupetLinfini}, we have uniformly in $p$
\begin{equation}
\begin{aligned}
p^{-n_X}\psi_{1/\sqrt{p}} e^{-B_{p,u/p}^2}(x_0,x_0) &= \psi_{1/\sqrt{u}}e^{-u\LL_t}(0,0) \\
&=\psi_{1/\sqrt{u}} \sum_{r=0}^k \mathscr{B}_{2r,t}(0,0)u^{r-n_X} +O(u^{k+1}).
\end{aligned}
\end{equation}

In conclusion, we have proved \eqref{dvlptuniformedunoyaudelachaleureq1-2} with
\begin{equation}
\label{Apj=f(B(2r,t))}
A_{p,j} = \sum_{r-\alpha=j+n_X} \mathscr{B}_{2r,t}(0,0)^{(2\alpha)}.
\end{equation}

We now prove \eqref{dvlptuniformedunoyaudelachaleureq2-2}. To do so, we fixe $r\in \N$ and study the asymptotic as $t\to 0$ of $\mathscr{B}_{2r,t}(0,0)$.

We define $\ul{\LL}_{3,t}^v$, $\ul{A}_\boldrm{r}^\boldrm{k}(\lambda,t,v)$ and $\ul{\mathscr{B}}_{2r,t}$ to be the objects corresponding to $\LL_{3,t}^v$, $A_\boldrm{r}^\boldrm{k}(\lambda,t,v)$ and $\mathscr{B}_{2r,t}$ above when we replace $\LL_t$ by $\ul{\LL}_t$ in their definitions. Then all Theorems \ref{QA_k^rQ'}-\ref{expansion-exp-L3} also hold for this underlined objects. 

Also, similarly to Theorems \ref{derivee(v)-derivee(0)} and \ref{estimation-Brtv}, we can prove first that for any $r\geq 0$ and $k>0$, there exist $C>0$ and $N\in \N$ such that for $\lambda \in \Gamma$ and $t>0$ ,
 \begin{equation}
\begin{aligned}
& \left\| \Big(\derpar{^{r}\LL_{3,t}^v}{v^{r}}\Big|_{v=0} - \derpar{^{r}\ul{\LL}_{3,t}^v}{v^{r}}\Big|_{v=0} \Big)s\right\|_{-1} \leq Ct \sum_{|\alpha|\leq r+3} \|Z^\alpha s\|_{1}, \\
&\bigg\| \Big(  \sum_{(\boldrm{k},\boldrm{r}) \in I_{k,r}} a_\boldrm{r}^\boldrm{k}A_\boldrm{r}^\boldrm{k}(\lambda,t,0)- a_\boldrm{r}^\boldrm{k}\ul{A}_\boldrm{r}^\boldrm{k}(\lambda,t,0)\Big)s\bigg\|_0 \leq Ct(1+|\lambda|)^N\sum_{|\alpha|\leq 4r+3} \|Z^\alpha s\|_0.
\end{aligned}
\end{equation}
And secondly that for $r\in \N$ and $q>0$, there exists $C>0$ such that for $t>0$  and $Z,Z'\in T_{\R,x_0}X$ with $|Z|,|Z'|\leq q$, we have
 \begin{equation}
 \label{diff-Brt-ulBrt}
\big\| \big(\mathscr{B}_{r,t}-\ul{\mathscr{B}}_{r,t}\big)(Z,Z') \big\| \leq Ct^{1/(2n_X+1)}.
\end{equation}

Recall that $\har_{x_0}(y)$, $y\in Y_{x_0}$, is defined in \eqref{defhary}. Once again, we define $\har_{x_0,3}^v(y)$, $\widetilde{A}_\boldrm{r}^\boldrm{k}(\lambda,v)(y)$ and $\widetilde{\mathscr{B}}_{2r}(y)$ to be the objects corresponding to $\LL_{3,t}^v$, $A_\boldrm{r}^\boldrm{k}(\lambda,t,v)$ and $\mathscr{B}_{2r,t}$ above when we replace $\LL_t$ by $\har_{x_0}(y)$ in their definitions. Then, once again, Theorems \ref{QA_k^rQ'}-\ref{expansion-exp-L3} also hold for this objects.

By \eqref{ul(L)=toep}, we then have
\begin{equation}
\label{ul(L)3v=toep}
\ul{\LL}_{3,t}^v = P_{p,x_0} \har_{x_0,3}^v(\cdot) P_{p,x_0}.
\end{equation}
As $\Delta^{T_{\R,x_0}X}$ commutes with $P_{p,x_0}$, we have $(\lambda-P_{p,x_0}\Delta^{T_{\R,x_0}X}P_{p,x_0})^{-1}=P_{p,x_0}(\lambda-\Delta^{T_{\R,x_0}X})^{-1}P_{p,x_0}$. As a consequence, using \eqref{def-Ark(lambda,t,v)} and the same reasoning as for \eqref{[resolvante,Toepltiz]=0} (in particular Theorem \ref{prodtoeplitztenseurhilbert}), we find that for any $(\boldrm{k},\boldrm{r})\in I_{k,r}$, there exist $C>0$ and $K\in \N$ such that
\begin{equation}
\label{ul(A)=toep}
\left\|\ul{A}_\boldrm{r}^\boldrm{k}(\lambda,t,0) - P_{p,x_0} \widetilde{A}_\boldrm{r}^\boldrm{k}(\lambda,0)(\cdot) P_{p,x_0}\right\|^{0,0} \leq Cp^{-1}(1+|\lambda|^2)^K.
\end{equation}
Thus by \eqref{def-Brtv},
\begin{equation}
\left\|\ul{\mathscr{B}}_{2r,t} - P_{p,x_0} \widetilde{\mathscr{B}}_{2r} P_{p,x_0}\right\|^{0,0}\leq Cp^{-1}.
\end{equation}
As the proof of Theorem \ref{cvcenoyauLup->noyauLuinfini}, this implies that for the operator norm,
\begin{equation}
\label{ul(B)=toep}
\ul{\mathscr{B}}_{2r,t} (0,0)= P_{p,x_0} \widetilde{\mathscr{B}}_{2r} (0,0)P_{p,x_0} + O(p^{-1/(2n_X+1)}).
\end{equation}

Recall that $A_j$ is defined in \eqref{def-Lambda_u} and \eqref{defAj2}. With the same reasoning which led to \eqref{Apj=f(B(2r,t))}, we find
\begin{equation}
\label{Aj=f(Btilde(2r))}
A_{j} = \sum_{r-\alpha=j+n-x} \tilde{\mathscr{B}}_{2r}(0,0)^{(2\alpha)}.
\end{equation}

With \eqref{Apj=f(B(2r,t))}, \eqref{diff-Brt-ulBrt}, \eqref{ul(B)=toep} and \eqref{Aj=f(Btilde(2r))}, we find \eqref{dvlptuniformedunoyaudelachaleureq2-2} for the $\mathscr{C}^0$-norm.

Finally, using the fact that $\nabla^{\mathrm{pr}_M^*\End(\E_p)}_U\LL_{3,t}^v$ has the same structure as $\LL_{3,t}^v$, we can show that all the estimates in this section also hold for the derivatives of the operators involved. Thus, \eqref{dvlptuniformedunoyaudelachaleureq2-2} holds for the $\mathscr{C}^\ell$-norm.

 The proof of Theorem \ref{dvlptuniformedunoyaudelachaleur2} is completed.


\subsection{Proof of Theorem \ref{dominationugrand2}}
\label{demothm2-2}

We use here the same notations and definitions as in Section \ref{demothm2}. Also, we assume here again that \eqref{SpCp} holds for $p\geq 1$. As $\Sp(B_{p,1}^2)=\Sp(D_p^2)$ and by Lemma \ref{spdeDp22}, we have once again a decomposition
\begin{equation}
\label{trsavecBupenfctiondePpuetKpu2}
p^{-n_X}\psi_{1/\sqrt{p}}\trs \left[ N_{u/p}e^{-B_{p,u/p}^2} \right] = p^{-n_X} \trs \left[ N_u (\PP_{p,u}+\K_{p,u})\right].
\end{equation}

\begin{lemme}
\label{estimeeaveclambda02}
Let $\lambda_0\in \R_-^*$. Then there exists $q_0$ such that for $q\geq q_0$, for $U\in T_\R B$ and $\ell \in \N$, there is a $C>0$ such that for $p\geq 1$
\begin{equation}
\label{estimeeaveclambda0eq2}
p^{-n_X}\left\| \big(\n^{\pi^*\End(\E_p)}_U\big)^\ell(\lambda_0-C_p)^{-q} \right\|_1 \leq C.
\end{equation}
\end{lemme}

\begin{proof}
As in \eqref{estimeedeg0}, we find using $H_p = D_p^2/p - \lambda_0$ that
\begin{equation}
\label{estimeedeg02}
p^{-n_X}\left\| ( \lambda_0-D_p^2/p)^{-q} \right\|_1  \leq C.
\end{equation}

Recall that $B_p^2=D_p^2+R_p$. A look at Bismut's Lichnerowicz formula \eqref{Lichnerowiczeq} and \eqref{convention} shows that locally, under the trivialization on $U_{x_k}$ (see Sections \ref{localization} and  \ref{localization2}), we have
\begin{equation}
\label{1/pRpavecOi,p}
\frac{1}{p}R_p = \frac{1}{p} \0_{1,p} +\0_{0,p},
\end{equation}
were $\0_{k,p}$ is an operator of order $k$ acting on $\E_{p}$ with bounded coefficients (with respect to the operator norm). Thus,
\begin{equation}
\begin{aligned}
\label{estimationsOip,i=1,2}
& || \0_{1,p}s||_{\sob^{k}(p)}\leq C ||s||_{\sob^{k+1}(p)}, \\
& || \0_{0,p}s||_{\sob^{k}(p)}\leq C ||s||_{\sob^{k}(p)}.
\end{aligned}
\end{equation}
From these estimates, we can conclude the proof as in Lemma~\ref{estimeeaveclambda0}.
\end{proof}

\begin{prop}
\label{termeavecKpu2}
For any $\ell\in \N$, there exist $a,C>0$ such that for $p\geq 1$ and $u\geq 1$,
\begin{equation}
p^{-n_X}\big| \trs\left[N_u \K_{p,u}\right] \big|_{\mathscr{C}^\ell(B)} \leq Ce^{-au}.
\end{equation}
\end{prop}

\begin{proof}
Proposition \ref{termeavecKpu2} follows from Lemma \ref{estimeeaveclambda02} exactly as Proposition \ref{termeavecKpu} follows from Lemma \ref{estimeeaveclambda0}.
\end{proof}

\begin{prop}
\label{termeavecPpu2}
For any $\ell\in \N$, there is a $C>0$ such that for any $p\geq 1$ and $u\geq 1$, 
\begin{equation}
\label{termeavecPpueq2}
p^{-n_Z}\big|  \trs\left[N_u \PP_{p,u}\right] \big|_{\mathscr{C}^\ell(B)}\leq \frac{C}{\sqrt{u}}.
\end{equation}
\end{prop}

\begin{proof}
 The proof is exactly the same as the proof of Proposition \ref{termeavecPpu}, the only change is that to prove the analogue of \eqref{normeinftermespossibles}, we substitute \eqref{coeffbornŽs} by
 \begin{equation}
p^{-n_Z} \dim \ker (D_p^2)=p^{-n_Z}\dim H^0(X,\xi\otimes F_p) = p^{-n_Z}\dim H^0(Z,\pi^*_1\xi\otimes\eta\otimes L^p )\leq C.
\end{equation}
\end{proof}

With \eqref{trsavecBupenfctiondePpuetKpu2} and Propositions \ref{termeavecKpu2} and \ref{termeavecPpu2}, we have proved Theorem \ref{dominationugrand2}.



\begin{thebibliography}{10}

\bibitem{MR0395610}
F.~A. Berezin.
\newblock Quantization.
\newblock {\em Izv. Akad. Nauk SSSR Ser. Mat.}, 38:1116--1175, 1974.

\bibitem{berline-getzler-vergne}
N.~Berline, E.~Getzler, and M.~Vergne.
\newblock {\em Heat kernels and {D}irac operators}.
\newblock Springer-Verlag, 1992.

\bibitem{MR1305280}
A.~Berthomieu and J.-M. Bismut.
\newblock Quillen metrics and higher analytic torsion forms.
\newblock {\em J. Reine Angew. Math.}, 457:85--184, 1994.

\bibitem{MR813584}
J.-M. Bismut.
\newblock The {A}tiyah-{S}inger index theorem for families of {D}irac
  operators: two heat equation proofs.
\newblock {\em Invent. Math.}, 83(1):91--151, 1986.

\bibitem{MR886814}
J.-M. Bismut.
\newblock Demailly's asymptotic {M}orse inequalities: a heat equation proof.
\newblock {\em J. Funct. Anal.}, 72(2):263--278, 1987.

\bibitem{MR1316553}
J.-M. Bismut.
\newblock Equivariant immersions and {Q}uillen metrics.
\newblock {\em J. Differential Geom.}, 41(1):53--157, 1995.

\bibitem{MR1623496}
J.-M. Bismut.
\newblock Holomorphic families of immersions and higher analytic torsion forms.
\newblock {\em Ast{\'e}risque}, (244):viii+275, 1997.

\bibitem{MR3099098}
J.-M. Bismut.
\newblock {\em Hypoelliptic {L}aplacian and {B}ott-{C}hern cohomology : a
  theorem of {R}iemann-{R}och-{G}rothendieck in complex geometry}, volume 305
  of {\em Progress in Mathematics}.
\newblock Birkh{\"a}user/Springer, Cham, 2013.

\bibitem{MR929147}
J.-M. Bismut, H.~Gillet, and C.~Soul{{\'e}}.
\newblock Analytic torsion and holomorphic determinant bundles. {II}. {D}irect
  images and {B}ott-{C}hern forms.
\newblock {\em Comm. Math. Phys.}, 115(1):79--126, 1988.

\bibitem{bismut-gillet-soule}
J.-M. Bismut, H.~Gillet, and C.~Soul{{\'e}}.
\newblock Analytic torsion and holomorphic determinant bundles. {III}.
  {Quillen} metrics on holomorphic determinants.
\newblock {\em Commun. Math. Phys.}, 115(2):301--351, 1988.

\bibitem{bismut-kohler}
J.-M. Bismut and K.~K{\"o}hler.
\newblock Higher analytic torsion forms and anomaly formulas.
\newblock {\em J. Algebraic Geom.}, 1:647--684, 1992.

\bibitem{bismut2004holomorphic}
J.-M. Bismut and X.~Ma.
\newblock Holomorphic immersions and equivariant torsion forms.
\newblock {\em J. Reine Angew. Math.}, pages 189--235, 2004.

\bibitem{MR2838248}
J.-M. Bismut, X.~Ma, and W.~Zhang.
\newblock Op{\'e}rateurs de {T}oeplitz et torsion analytique asymptotique.
\newblock {\em C. R. Math. Acad. Sci. Paris}, 349(17-18):977--981, 2011.

\bibitem{BMZweb}
J.-M. Bismut, X.~Ma, and W.~Zhang.
\newblock Asymptotic torsion and {T}oeplitz operators.
\newblock {\em Journal of the Institute of Mathematics of Jussieu}, Available
  on CJO 2015 doi:10.1017/S1474748015000171, 2015.

\bibitem{bismut-vasserot-89}
J.-M. Bismut and E.~Vasserot.
\newblock The asymptotics of the {Ray--Singer} analytic torsion associated with
  high powers of a positive line bundle.
\newblock {\em Commun. Math. Phys.}, 125:355--367, 1889.

\bibitem{bismut-vasserot-90}
J.-M. Bismut and E.~Vasserot.
\newblock The asymptotics of the {Ray-Singer} analytic torsion of the symmetric
  powers of a positive vector bundle.
\newblock {\em Ann. Inst. Fourier (Grenoble)}, 40(4):835--848, 1990.

\bibitem{MR1301849}
M.~Bordemann, E.~Meinrenken, and M.~Schlichenmaier.
\newblock Toeplitz quantization of {K}{\"a}hler manifolds and {${\rm gl}(N)$},
  {$N\to\infty$} limits.
\newblock {\em Comm. Math. Phys.}, 165(2):281--296, 1994.

\bibitem{MR620794}
L.~Boutet~de Monvel and V.~Guillemin.
\newblock {\em The spectral theory of {T}oeplitz operators}, volume~99 of {\em
  Annals of Mathematics Studies}.
\newblock Princeton University Press, Princeton, NJ, 1981.

\bibitem{Gil:2012fk}
J.~I. Burgos~Gil, G.~Freixas~i Montplet, and R.~Li{\c{t}}canu.
\newblock The arithmetic grothendieck-riemann-roch theorem for general
  projective morphisms.
\newblock {\em arXiv 1211.1783 (J. Fac. Sci. Univ. Toulouse to appear)}, 2012.

\bibitem{MR3165730}
J.~I. Burgos~Gil, G.~Freixas~i Montplet, and R.~Li{\c{t}}canu.
\newblock Generalized holomorphic analytic torsion.
\newblock {\em J. Eur. Math. Soc. (JEMS)}, 16(3):463--535, 2014.

\bibitem{MR894572}
E.~Getzler.
\newblock In{\'e}galit{\'e}s asymptotiques de {D}emailly pour les fibr{\'e}s
  vectoriels.
\newblock {\em C. R. Acad. Sci. Paris S{\'e}r. I Math.}, 304(16):475--478,
  1987.

\bibitem{MR2473633}
H.~Gillet, D.~R{{\"o}}ssler, and C.~Soul{{\'e}}.
\newblock An arithmetic {R}iemann-{R}och theorem in higher degrees.
\newblock {\em Ann. Inst. Fourier (Grenoble)}, 58(6):2169--2189, 2008.

\bibitem{MR1189489}
H.~Gillet and C.~Soul{{\'e}}.
\newblock An arithmetic {R}iemann-{R}och theorem.
\newblock {\em Invent. Math.}, 110(3):473--543, 1992.

\bibitem{MR1872550}
K.~K{{\"o}}hler and D.~Roessler.
\newblock A fixed point formula of {L}efschetz type in {A}rakelov geometry.
  {I}. {S}tatement and proof.
\newblock {\em Invent. Math.}, 145(2):333--396, 2001.

\bibitem{MR1881571}
K.~K{{\"o}}hler and D.~Roessler.
\newblock A fixed point formula of {L}efschetz type in {A}rakelov geometry.
  {II}. {A} residue formula.
\newblock {\em Ann. Inst. Fourier (Grenoble)}, 52(1):81--103, 2002.

\bibitem{MR1765553}
X.~Ma.
\newblock Formes de torsion analytique et familles de submersions. {I}.
\newblock {\em Bull. Soc. Math. France}, 127(4):541--621, 1999.

\bibitem{ma2000submersions}
X.~Ma.
\newblock Submersions and equivariant quillen metrics.
\newblock {\em Ann. Inst. Fourier (Grenoble)}, 50(5):1539--1588, 2000.

\bibitem{ma-marinescu}
X.~Ma and G.~Marinescu.
\newblock {\em Holomorphic {M}orse inequalities and {B}ergman kernels}, volume
  254 of {\em Progress in Mathematics}.
\newblock Birkh{\"a}user Verlag, Basel, 2007.

\bibitem{MR2393271}
X.~Ma and G.~Marinescu.
\newblock Toeplitz operators on symplectic manifolds.
\newblock {\em J. Geom. Anal.}, 18(2):565--611, 2008.

\bibitem{MR2876259}
X.~Ma and G.~Marinescu.
\newblock Berezin-{T}oeplitz quantization on {K}{\"a}hler manifolds.
\newblock {\em J. Reine Angew. Math.}, 662:1--56, 2012.

\bibitem{MR2286586}
X.~Ma and W.~Zhang.
\newblock Superconnection and family {B}ergman kernels.
\newblock {\em C. R. Math. Acad. Sci. Paris (full version to appear)},
  344(1):41--44, 2007.

\bibitem{MR836726}
V.~Mathai and D.~Quillen.
\newblock Superconnections, {T}hom classes, and equivariant differential forms.
\newblock {\em Topology}, 25(1):85--110, 1986.

\bibitem{asymp-torsion-CRAS}
M.~Puchol.
\newblock The asymptotic of the holomorphic torsion forms.
\newblock {\em C. R. Math. Acad. Sci. Paris (to appear)}, 2015.

\bibitem{thesis_puchol}
M.~Puchol.
\newblock {\em In\'{e}galit\'{e} de {M}orse holomorphes $G$-invariantes et
  formes de torsion asymptotiques}.
\newblock PhD thesis, Universit\'{e} Paris Diderot,
  http://webusers.imj-prg.fr/~martin.puchol/, 2015.

\bibitem{MR0383463}
D.~B. Ray and I.~M. Singer.
\newblock Analytic torsion for complex manifolds.
\newblock {\em Ann. of Math. (2)}, 98:154--177, 1973.

\bibitem{MR1805922}
M.~Schlichenmaier.
\newblock Deformation quantization of compact {K}{\"a}hler manifolds by
  {B}erezin-{T}oeplitz quantization.
\newblock In {\em Conf{\'e}rence {M}osh{\'e} {F}lato 1999, {V}ol. {II}
  ({D}ijon)}, volume~22 of {\em Math. Phys. Stud.}, pages 289--306. Kluwer
  Acad. Publ., Dordrecht, 2000.

\bibitem{MR1208731}
C.~Soul{{\'e}}.
\newblock {\em Lectures on {A}rakelov geometry}, volume~33 of {\em Cambridge
  Studies in Advanced Mathematics}.
\newblock Cambridge University Press, Cambridge, 1992.
\newblock With the collaboration of D. Abramovich, J.-F. Burnol and J. Kramer.

\bibitem{yoshikawa2004k3}
K.-Y. Yoshikawa.
\newblock K3 surfaces with involution, equivariant analytic torsion, and
  automorphic forms on the moduli space.
\newblock {\em Invent. Math.}, 156(1):53--117, 2004.

\end{thebibliography}
\end{document}